\documentclass[11pt, a4paper]{amsart}

\calclayout

\usepackage[table]{xcolor}
\usepackage{tikz}
\usepackage{graphicx}

\usepackage{amsmath}
\usepackage{amssymb}
\usepackage{amsthm}
\usepackage{amsfonts}

\usepackage{graphicx}
\usepackage{subcaption}
\usepackage{cprotect}
\usepackage{cite}

\usepackage[T1]{fontenc}

\usepackage{enumerate}

\usepackage{todonotes, changes}
\usepackage{comment}

\usepackage[shortlabels]{enumitem}
\newlist{mycases}{enumerate}{1}
\setlist[mycases]{
    label=({\roman*}), 
    ref=(\roman*),                  
    leftmargin=2cm
}

\newcommand{\Tterm}{T_0}
\newcommand{\Tst}{\Delta_2}

\newcommand{\RR}{\mathbb{R}}
\newcommand{\KK}{\mathbb{K}}
\newcommand{\ZZ}{\mathbb{Z}}
\newcommand{\NN}{\mathbb{N}}
\newcommand{\QQ}{\mathbb{Q}}

\newcommand{\K}{\mathcal{K}}

\newcommand{\conv}{\mathrm{conv}}

\newcommand{\pos}{\mathrm{pos}}

\newcommand{\aff}{\mathrm{aff}}

\newcommand{\inter}{\mathrm{int}}
\newcommand{\relint}{\mathrm{relint}}

\DeclareMathOperator{\vol}{vol}

\newcommand{\LE}{\mathrm{G}}
\newcommand{\LEinter}[1]{\hyperlink{def:Gcirc}{\color{black}\LE^\circ(#1)}}

\newcommand{\wdt}{\hyperlink{def:wdt}{\color{black}\mathrm{w}}}
\newcommand{\BP}{\hyperref[def:blocking-polytope]{\color{black}B_P}}
\newcommand{\Bhex}{\hyperlink{target:Bhex}{\color{black}B^{\mathrm{hex}}}}
\newcommand{\Bpent}{\hyperlink{target:Bpent}{\color{black}B^{\mathrm{pent}}}}
\newcommand{\Bcross}{\hyperlink{target:Bcross}{\color{black}B^{\mathrm{cross}}}}
\newcommand{\Bkite}{\hyperlink{target:Bkite}{\color{black}B^{\mathrm{kite}}}}
\newcommand{\Bpyr}{\hyperlink{target:Bpyr}{\color{black}B^{\mathrm{pyr}}}}
\newcommand{\Btrap}{\hyperlink{target:Btrap}{\color{black}B^{\mathrm{trap}}}}
\newcommand{\Bst}{\hyperlink{target:Bst}{\color{black}B^{\mathrm{st}}}}
\newcommand{\Bterm}{\hyperlink{target:Bterm}{\color{black}B^{\mathrm{term}}}}

\DeclareRobustCommand{\flt}{\hyperlink{target:flt}{\color{black}\mathrm{Flt}}}

\newcommand{\GL}{\mathrm{GL}}

\newcommand{\va}{{\boldsymbol a}}
\newcommand{\vb}{{\boldsymbol b}}
\newcommand{\vc}{{\boldsymbol c}}
\newcommand{\vd}{{\boldsymbol d}}
\newcommand{\ve}{{\boldsymbol e}}

\newcommand{\vo}{{\boldsymbol o}}
\newcommand{\vp}{{\boldsymbol p}}
\newcommand{\vq}{{\boldsymbol q}}
\newcommand{\vr}{{\boldsymbol r}}
\newcommand{\vs}{{\boldsymbol s}}
\newcommand{\vt}{{\boldsymbol t}}

\newcommand{\vu}{{\boldsymbol u}}
\newcommand{\vv}{{\boldsymbol v}}
\newcommand{\vw}{{\boldsymbol w}}
\newcommand{\vx}{{\boldsymbol x}}
\newcommand{\eins}{{\boldsymbol 1}}

\newcommand{\vy}{{\boldsymbol y}}
\newcommand{\vz}{{\boldsymbol z}}

\newcommand{\sph}{\mathbb{S}}

\newcommand{\hex}{\mathcal{A}}

\renewcommand{\vec}[1]{\boldsymbol{#1}}

\usepackage{hyperref}
\hypersetup{
colorlinks=true,
  citecolor=magenta, linkcolor=cyan, urlcolor=blue}
  
\title[Exact flatness constant]{Exact flatness constant for one-point convex bodies and the discrete isominwidth problem: the planar case}
\author[G.\ Averkov]{Gennadiy Averkov}
\author[G.\ Codenotti]{Giulia Codenotti}
\author[A.\ Freyer]{Ansgar Freyer}
\author[K.\ Huang]{Kyle Huang}
\address{FU Berlin, AG Diskrete Geometrie und Topologische Kombinatorik, Arnimallee 2, 14195 Berlin, Germany}
\email{\{giulia.codenotti, a.freyer\}@fu-berlin.de} 
\date{}
\address{Brandenburg University of Technology Cottbus-Senftenberg, Platz der Deutschen Einheit~1, 03046 Cottbus, Germany}
\email{\{averkov, kylehuang\}@b-tu.de}
\thanks{All of the authors are funded by the Deutsche Forschungsgemeinschaft (DFG, German Research Foundation) -
539867386.}

\numberwithin{equation}{section}

\usepackage{mathtools}
\mathtoolsset{showonlyrefs=true}

\theoremstyle{plain}
\newtheorem{theorem}{Theorem}[section]
\newtheorem{lemma}[theorem]{Lemma}
\newtheorem{corollary}[theorem]{Corollary}

\newtheorem{proposition}[theorem]{Proposition}
\newtheorem*{proposition*}{Proposition}

\newtheorem{thmx}{Theorem}

\newtheorem{definition}[theorem]{Definition}

\newtheorem{remark}[theorem]{Remark}
\newtheorem{claim}{Claim}
\newtheorem*{claim*}{Claim}
\newtheorem*{remark*}{Remark}

\newcommand{\ansgar}[1]{\todo[size=\tiny,color=blue!30]{#1 \\ \hfill --- A.}}
\newcommand{\Ansgar}[1]{\todo[size=\tiny,inline,color=blue!30]{#1 \\ \hfill --- A.}}

\newcommand{\kyle}[1]{\todo[size=\tiny,color=red!30]{#1 \\ \hfill --- K.}}

\newcommand{\giulia}[1]{\todo[size=\tiny,color=green!30]{#1 \\ \hfill --- G.}}

\begin{document}

\begin{abstract}
A variant of the flatness problem from integer programming is studied, in which one considers convex bodies in $\RR^d$ with at most $k$ interior lattice points. The maximum lattice width of such a body is denoted by $\flt(d,k)$ and it is related to the classical flatness constant as well as a conjectural dual version of Minkowski's convex body theorem due to Makai. Moreover, it is shown that $\flt(2,1)=3$, i.e., any planar convex body with at most one interior point has lattice width at most three. This leads to an isominwidth inequality for the lattice point enumerator of planar convex bodies.
\end{abstract}

\maketitle

\section{Introduction}

\subsection{Setting and main results} 

We investigate flatness-type problems from integer programming and geometry of numbers. The prototypical problem motivating our study is the classical flatness problem whose objective is maximizing the lattice width among all hollow convex bodies $K \subset \RR^d$ in a given dimension $d$. Here, a \emph{convex body} is a non-empty compact convex set. The set of convex bodies in $\RR^d$ is denoted by $\K^d$. The body $K$ being \emph{hollow} means that the interior of $K$  contains no lattice points. The \emph{lattice width} is defined as the infimum of the width function 
\[
	\wdt(K,\vu) := \max_{\vx \in K} \langle \vu, \vx \rangle - \min_{\vx \in K } \langle \vu , \vx \rangle
\]
of the underlying body $K$ over non-zero lattice vectors: 
\hypertarget{def:wdt}{} 
\[
\wdt(K) := \inf_{\vu \in \ZZ^d \setminus \{0\}} \wdt(K,\vu)
\]
If $K$ has non-empty interior the infimum in the definition of $\wdt(K)$ is actually a minimum and is a positive number. Otherwise, the infimum is zero and it is not necessarily attained (cf.\ Remark \ref{wdt:0:remark}). 
We also use $\wdt(K;X) := \inf_{\vu \in X} \wdt(K,\vu)$ for some set of non-zero vectors $X\subset\RR^d$.  In particular, $\wdt(K) = \wdt(K;\ZZ^d\setminus\{0\})$.
The lattice width is invariant with respect to \emph{unimodular transformations}, i.e., affine maps of the form $\vx \mapsto U\vx + \vt$, where $U\in\GL_d(\ZZ)$ and $\vt\in\ZZ^d$. Here $\GL_d(\ZZ)$ denotes the group of matrices $U\in\ZZ^{d\times d}$ for which $U^{-1}$ exists in $\ZZ^{d\times d}$. We say that two sets $A,B\subset\RR^d$ are \emph{unimodularly equivalent}, if there exists a unimodular transformation $g$ such that $A=g(B)$.

In generalized flatness-type problems, one still maximizes the lattice width but replaces the hollowness constraint by another condition. The natural more general condition, which was already addressed in the seminal paper by  Kannan and Lov\'asz \cite[(4.1)]{coveringminima}, is having a fixed number $k$  of interior lattice points. The case $k=0$ corresponds to the classical setting (see Section \ref{subsec:classical_flatness} below). We call the respective convex bodies \emph{$k$-point convex bodies} and we study
\hypertarget{target:flt}{} 
\begin{equation}
\label{eq:fltdk}
\mathrm{Flt}(d,k) := \sup\{ \wdt(K) : K\in\K^d,~\LEinter{K} \leq k\},
\end{equation}
where $\hypertarget{def:Gcirc}{\LE^\circ(K)}$ denotes the number of interior lattice points of $K$. 

We establish structural results for $\flt(d,k)$ in arbitrary dimension (Section~\ref{sec:ddim}). One of our general insights is the algebraicity of the number $\flt(d,k)$ for any choice of $d$ and $k$. This makes the problem of determination of the exact value of $\flt(d,k)$ for given $d$ and $k$ a well-defined task. The problem of determining $\flt(d,k)$ is non-trivial already in the case $d=2, k=0$. For this case, Hurkens determined that $\flt(2,0) = 1+\tfrac 2 {\sqrt{3}}$ \cite{hurkens} (cf.\ Theorem \ref{thm:hurkens}). 
When $d=2$ and $k$ is getting larger, the complexity of the problem seems to increase substantially. Already for $d=2, k=1$ determination of $\flt(2,1)$ is very challenging, in particular, because the number of combinatorial situations to be considered grows dramatically and the cases themselves are more complex (with more degrees of freedom involved). Our main theorem states that $\flt(2,1) = 3$ and that, up to unimodular equivalence, the standard triangle scaled by factor $3$, is the unique maximizer of the flatness problem among planar $1$-point convex bodies.  We write $\Tst$ for the \emph{standard lattice triangle}, that is, the triangle with vertices $(0,0)$, $(0,1)$ and $(1,0)$. 

\begin{theorem}
\label{thm:flt21}
Let $K\subset\RR^2$ be a convex body with $\LEinter{K} \leq 1$. Then, $\wdt(K) \leq 3$. Equality holds if and only if $K$ is unimodularly equivalent to $3 \Tst$.
\end{theorem}

Our main result implies two geometric inequalities, which we find remarkable. We discuss one here and postpone the other to Section~\ref{sssec:transf}.

A central geometric inequality in metric geometry is the isoperimetric inequality in Euclidean geometry, which relates volume to surface area. A related version of this result is the isodiametric inequality. It relates volume to diameter and states that, among convex bodies of a given diameter, Euclidean balls maximize the volume. 

Our main result is intimately related to the lattice version of the dual counterpart to the isodiametric inequality, namely the isominwidth inequality. 
While in the isodiametric case one minimizes the length of a longest segment at fixed volume, in the isominwidth setting one maximizes the width of a smallest enclosing strip at fixed volume. Note that segments and strips are dual objects.

Analogous problems arise in lattice geometry when the Euclidean minimum width is replaced by the lattice width. One may either keep a continuous volume or replace it by a discrete analogue (lattice point enumerator). Isominwidth problems, both in the Euclidean and in the lattice setting, are notoriously difficult. Exact solutions are typically known only in low dimensions, if any:
The exact solution to the Euclidean isominwidth problem is only known for dimension two, which is a result of  P\'al \cite{pal} stating that for  every $K\subset\RR^2$ with $\vol(K) > 0$ one has
\begin{equation}
\label{eq:pal}
\frac{\wdt(K;\sph^1)}{\vol(K)^{1/2}} \leq \frac{\wdt(T_{\mathrm{reg}};\sph^1)}{\vol(T_{\mathrm{reg}})^{1/2}}.
\end{equation}
Here, $\sph^1$ is the Euclidean unit circle and
$T_{\mathrm{reg}}$  is a regular triangle. Generalizing P\'al's inequality to higher dimensional convex bodies is a longstanding open problem. 
The reader should be warned that the regular simplex is known to be \emph{not} an optimal solution in dimensions $d > 2$. A conjectural solution for $d=3$ is described in \cite{horvath}.

The version of the isominwidth problem for lattice width and volume was solved Fejes-T\'oth and Makai  \cite{FejMak,Makai} in the plane. The solution is the inequality
\begin{equation}
\label{eq:makai}
\frac{\wdt(K)}{\vol(K)^{1/2}} \leq \frac{\wdt(\Tterm)}{\vol(T_0)^{1/2}} = \sqrt{\frac 83},
\end{equation}
with the \emph{terminal triangle} 
$T_0 = \conv\{(-1,-1),(0,1),(1,0)\}$ being the unique equality case up to unimodular equivalence and dilation. 
The corresponding inequality in three dimensional space was recently proven by Aliev in \cite{Aliev}.

Here we complement the picture by providing what we call 
an inhomogeneous isominwidth inequality, in which the lattice point enumerator takes the role of the volume.

\begin{corollary}
\label{cor:lattice_isominwidth}
Let $K\subset\RR^2$ be a convex body with $\LEinter{K} > 0$. Then,
\[
\frac{\wdt(K)}{\LEinter{K}^{1/2}} \leq \frac{\wdt(3\Tst)}{\LEinter{3\Tst}^{1/2}}
\]
and equality holds if and only if $K$ is unimodularly equivalent to $3 \Tst$.
\end{corollary}


\subsubsection*{Outline}
In the remainder of this section, we discuss the role of $\flt(d,k)$ in several branches of mathematics. In Section \ref{sec:ddim}, we prove results on $\flt(d,k)$ in arbitrary dimension. 
  Section \ref{sec:flt21}  is devoted to the proof of Theorem \ref{thm:flt21}, which is carried out by a classification of inclusion maximal 1-point bodies.  In Section~\ref{sec:consequences} we apply the main theorem to prove Corollary \ref{cor:lattice_isominwidth} and we obtain explicit bounds in two transference theorems. We conclude with open questions in Section \ref{sec:outlook}.

\subsection{Motivation and applications}
\label{sec:studies}
The flatness problem plays a fundamental role in many research areas, including integer programming, algebra, and algebraic geometry. In this section we explain the connection of flatness to these other topics, review results from these areas, many of which are also needed for our arguments, and present some consequences of our main results.

\subsubsection{The classical flatness problem}
\label{subsec:classical_flatness}
In recent years, substantial progress has been made in understanding the asymptotic behavior of the optimal value attained in the flatness problem -- known as the \emph{flatness constant} -- as a function of the dimension $d$, see the recent breakthrough by Reis and Rothvoss \cite{RR23}. One motivation for this is the study of asymptotic efficiency of algorithms for solving integer linear optimization problems, when the number $d$ of integer variables goes to infinity. The first results of this type   were given by Lenstra \cite{lenstra}, with Reis and Rothvoss \cite{RR23} providing a recent update on the asymptotic complexity of integer programming in dimension $d$. 

Complementary to the asymptotic studies, determining the flatness constant  $\flt(d,0)$ exactly for concrete values of $d$ is a natural problem, in particular, because there is not a canonical candidate for an extremal body. 
In fact, knowledge of values of  $\flt(d,0)$ is very limited. Trivially, one has $\flt(1,0)=1$ while for $d=2$ Hurkens obtained the following result:

\begin{thmx}[Hurkens, \cite{hurkens}]
\label{thm:hurkens}
Let $K\subset\RR^2$ be a convex body with $\LEinter{K}=0$. Then, $\wdt(K) \leq 1+\tfrac {2}{\sqrt{3}}$ and the inequality is tight.
\end{thmx}

An alternative proof of Hurkens' result can be found in \cite{codenottifreyer}, significantly reducing the complexity of the resulting optimization problem. The proof of Theorem \ref{thm:flt21} borrows some ideas from this proof (cf.\ Section \ref{sssec:terminal}).

Another way to phrase Theorem \ref{thm:hurkens} is to say that $\flt(2,0) =1+\tfrac {2}{\sqrt{3}}$. Beyond the case $d=2$, no exact value of the flatness constant has been determined so far. Based on a symmetry argument, it has been conjectured in \cite{acms_local_3d_2021} that $\flt(3,0) = 2 + \sqrt{2}$. So, the non-asymptotic version of the classical flatness problem remains a challenge even in the dimension $d=3$.  Mayrhofer, Schade and Weltge have provided conjectural realizers of $\flt(4,0)$ and $\flt(5,0)$ \cite{MSW}, following the ideas from \cite{acms_local_3d_2021}.

While the flatness problem in the fixed-dimensional setting is less well known than its asymptotic counterpart, it is nonetheless of significant importance. 
Cutting plane theory, within integer optimization, is one context where this plays a key role. 
One looks for additional linear inequalities in order to get closer to an integer solution. The new linear inequality is called a \emph{cutting plane}.
%
%
In \cite{balas_cuts_1972}, Balas established a general way to derive cutting planes from general hollow bodies. The respective cutting planes are known as \emph{intersection cuts} (see \cite{balas2018disjunctive} and \cite{conforti2014integer}).

Cutting planes are derived from hollow bodies in fixed small dimensions. 
For example, mixed-integer rounding cuts, which are omnipresent in integer optimization solvers, 
are intersection cuts from the $1$-dimensional hollow segments $[a,a+1]$ with $a \in \ZZ$. But it turns out that by making use of hollow bodies of higher dimensions, one can generate cuts that are significantly stronger; see Theorem~2 in  \cite{ave_basu_paat_2018}.

Understanding the strength of the cutting planes
depends crucially on the value of the flatness constant, and the values in fixed small dimensions are of particular importance. For instance, Dash et al.\ \cite{dash2014lattice} 
show that so-called $21$-branch disjunctions are enough to derive any valid inequality for mixed integer models with three integer variables
by making use of the upper bounds on $\flt(3,0)$. 

Computing the flatness constants in fixed dimensions also arises in toric geometry. Here, the study of the singularities corresponds to the analysis and classification of the so-called empty simplices--hollow lattice simplices whose only lattice points are their vertices. For $d=3$, White's theorem says that all empty simplices have lattice width exactly one. But for $d=4$, the situation is much more complicated. The classification project started by Mori, Morrison and Morrison \cite{mmm1988four}, and has been completed only recently by Iglesias-Valiño and Santos \cite{iglesias_santos}, making use of the classification of maximal hollow lattice tetrahedra from \cite{aww2011maximal} and \cite{akw2017notions} and, in turn, the classification result in \cite{akw2017notions} relies on the knowledge of bounds for the flatness constant in dimension $3$.



\subsubsection{Flatness problem for \(k\)-point convex bodies}
We now discuss what is known about the $k$-point body version of the flatness problem, first addressed by Kannan and Lov\'asz in \cite{coveringminima}.  Theorem~(4.1) from \cite{coveringminima} states that $\flt(d,k) \leq \mathcal O (d^2(k+1)^{1/d})$. 
This was subsequently improved to $\mathcal O(d(d^{1/2} + k^{1/d}))$ by Banaszczyk, Litvak, Pajor and Szarek in \cite{banalitvak}. As we will prove in Theorem \ref{thm:asymp}, an improvement on the asymptotic order of $\flt(d,0)$ also gives improved bounds on the order of $\flt(d,k)$. By the aforementioned bound of Reis and Rothvoss, we obtain $\flt(d,k) \leq \mathcal{O} (d(\log(d)^3 + k^{1/d}))$ (cf.\ Corollary \ref{cor:k_flatness_bound}).

In integer programming, \(k\)-point convex bodies can be used to generate cutting planes, as suggested by Averkov, Gonz\'alez Merino, Paschke, Schymura, and Weltge \cite[Section 3]{dhelly}. 
At the same time, in toric geometry, rational polygons with a single interior point encode so-called log canonical del Pezzo surfaces, see \cite{bohnert_springer_classifying_2024, hhs_classifying_2025}. 

Determining \(\flt(2,2)\) follows easily from work of Codenotti, Hofscheier, and Hall \cite{gen_flatness_chh21}. A convex body in the plane which contains at most $2$ interior lattice points is in particular a $\ZZ$-$\Delta_2$-free convex body, as defined by Averkov, Hofscheier and Nill in
 \cite{AHN}: these are convex bodies in the plane whose interiors do not contain a unimodular triangle. Since the maximum width of such $\ZZ$-$\Delta_2$-free bodies was determined by Codenotti, Hall and Hofscheier in \cite{gen_flatness_chh21}, this gives an upper bound for $\flt(2,2)$. To show that this bound is equal to $\flt(2,2)$, it is sufficient to observe that the unique maximizer of the width of $\ZZ$-$\Delta_2$-free (see Figure 9 in \cite{gen_flatness_chh21}) contains exactly $2$ interior lattice points. Thus, with $\Tterm = \conv\{(-1,-1),(0,1),(1,0)\}$, we have the following.

\begin{thmx}[Codenotti, Hofscheier, Hall, \cite{gen_flatness_chh21}]
\label{thm:flt22}

Let $K\subset \RR^2$ be a convex body with $\LEinter{K} \leq 2$. Then, $\wdt(K) \leq \tfrac{10}{3}$. Equality holds if and only if $K$ is unimodularly equivalent to $\tfrac 53 T_0+\tfrac 13\ve_1$.
\end{thmx} 

Our proof of Theorem \ref{thm:flt21} extends the methods of \cite{gen_flatness_chh21}, although our casework is notably more involved, as the maximizer $3 \Tst$ arises as a limit of every parameter space requiring us to prove several sharp bounds, see Section \ref{sec:flt21}. 


\subsubsection{Transference theorems}
\label{sssec:transf}

For a convex body $K$ with $\vo\in\inter(K)$, one defines its $i$th successive minimum as
\begin{equation}
\label{eq:succ_min_def}
\lambda_i(K) = \min\{\lambda\geq 0 :\dim (\lambda K \cap \ZZ^d) \geq i\}
\end{equation}
so that, in particular, 
\[
\lambda_1(K) = \min\{\lambda \geq 0 : \lambda K \cap \ZZ^d \neq \{\vo\}\}.
\]
The successive minima are critical parameters in the geometry of numbers as they give a good idea of the scale of the body $K$ relative to the lattice.
It is easy to verify that $\lambda_1(K) \geq 1$ holds if and only if $\inter K \cap \ZZ^d = \{\vo\}$. Moreover, we have $\lambda_1((K-K)^*) = \wdt(K)$, where $X^* = \{\vy\in\RR^d \colon \langle \vx,\vy\rangle\leq 1,~\forall \vx\in X\}$ is the polar set and $K-K = \{\vx-\vy \colon \vx,\vy \in K\}$ is the so-called difference body.
In this setting, an equivalent formulation of Theorem \ref{thm:flt21} is as follows (simply by applying the theorem to $\frac{1}{\lambda_1(K)}K$).

\begin{corollary}
\label{cor:non_sym_transf}
For any convex body $K\subset\RR^2$ with $\vo\in\inter(K)$, we have
\[
\lambda_1(K) \,\lambda_1((K-K)^*) \leq 3,
\]
and equality holds, if and only if \[
K = \lambda U \cdot\conv\{(-1,-1),(-1,2),(2,-1)\},
\]
for some $\lambda > 0$ and $U\in\GL_2(\ZZ)$.
\end{corollary}

This inequality is an example of a ``transference theorem'', by which one means an inequality relating a lattice quantity of $K$ to its polar.
Another example is the inequality
\begin{equation}
\label{eq:hurkens_trans}
\mu(K)\lambda_1((K-K)^*) \leq 1+ \tfrac 2{\sqrt{3}},
\end{equation}
where $\mu(K)$ is the \emph{covering radius} of $K$ (cf.\ Section \ref{sec:ddim}). 
This is simply a reformulation of Hurkens' Theorem (cf.\ Theorem \ref{thm:hurkens}) using the alternative definition of the flatness constant in Equation \eqref{eq:covrad_flatness}.

Interest in such transference theorems also stems from the interpretation of a convex body as the unit ball of a certain norm. Since this is most natural if $K$ is origin symmetric, i.e., $-K=K$ holds, it is common to work in this setting. 
Using bounds on the so-called Mahler volume one can show that
\begin{equation}
\label{eq:mahler_transf}
\lambda_1(K) \lambda_1(K^*) \leq cd
\end{equation}
holds for any origin symmetric convex body $K\in\K^d$ with non-empty interior and some universal constant $c>0$. However, the constant $c$ that is obtained this way is not optimal.

In \cite[Theorem 3]{baranyfuredi}, Barany and F\"uredi sketch an argument that shows that the quotient of lattice width and ``lattice diameter'' of a planar body is bounded by $\tfrac 43$. 
From this one obtains an explicit upper bound of $\tfrac 43$ in \eqref{eq:mahler_transf} via a simple limit argument. As a byproduct of our investigation we obtain a different proof of this fact, which additionally characterizes the equality case.

\begin{theorem}
\label{thm:transf}
Let $K\subset\RR^2$ be an origin symmetric convex body with $\vo\in\inter (K)$. Then,
\begin{equation}
\label{eq:transf_dimtwo}
\lambda_1(K)\,\lambda_1(K^*) \leq \tfrac 43.
\end{equation}
Moreover, equality holds if and only if \[
K=\lambda U \cdot \conv\{\pm (2,1),\pm (1,2), \pm (-1,1) \}
\] for some $\lambda > 0$ and $U\in\GL_2(\ZZ)$.
\end{theorem}

A similar theorem for the product of the first and second minimum of a planar convex body and its polar was proven by Ambro and Ito in \cite{ambroito}.
On the asymptotic side, strong transference theorems have been obtained by Banaszczyk \cite{banas1,banas2,banas3}, who obtained, for instance,
\begin{equation}
\label{eq:banaszczyk}
\lambda_1((K-K)^*)\lambda_d(K-K) \leq c\,d\log(d)
\end{equation}
for an absolute constant $c>0$ and any $K\in\K^d$.

\subsubsection{Discretization of convexity}
\label{sssec:disc_conv}

Our Corollary \ref{cor:lattice_isominwidth} can be seen as a ``discretization'' of the isominwidth inequality, where the area is replaced by the lattice point enumerator. This is inspired by the interpretation of volume as the lattice point count with respect to an infinitesimally fine lattice. More precisely, the volume of a compact set can be computed as the Riemann integral of its indicator.

The question that arises naturally is which theorems concerning the volume functional already hold for the lattice point enumerator? In recent years this question was studied for many classical theorem in convex geometry, such as the Brunn-Minkowski inequality \cite{GardnerGronchi, Halikiasetal, Iglesiasetal} the Rogers-Shephard inequality \cite{alonsoetal} as well as several results from discrete tomography \cite{ahz, ggz, dcgpaper, freyerhenk}.

\section{Results in arbitrary dimension}
\label{sec:ddim}

\subsection{Asymptotic results and general inequalities} 

We give a summary of inequalities that link $\flt(d,k)$ for different choices of $(d,k)$ and asymptotic results following from these inequalities. These inequalities and asymptotic statements are either directly available in the literature or follow directly from known results and techniques in the literature. While it is not our primary purpose to provide a survey, since results for $\flt(d,k)$ are scattered in the literature, we find providing an overview helpful. Further, we need some of the summarized results in our subsequent discussion. 

Codenotti and Santos  \cite{codenottisantos} established that the sequence of flatness constants $\flt(d,0)$ for hollow convex bodies grows strictly with the dimension $d$. 
\begin{theorem}[Codenotti \& Santos] 
	\label{flt:monotonicity} 
	One has 
	$\flt(m+n, 0) \ge \flt(m,0) + \flt(n,0)$ 
	for all $m,n \in \ZZ_{>0}$. 
	In particular, $\flt(d+1,0) \ge \flt(d,0)+ 1$ for all 
	$d \in \ZZ_{>0}$. 
\end{theorem}

In \cite{acddf_chvatal_gomory_2013}, the generalized flatness constant for $k=1$ appears in the context of the cutting plane theory of integer optimization problems; this paper studies the so-called reverse \emph{Chv\'atal-Gomory rank}. As an auxiliary result, Corollary~5 in \cite{acddf_chvatal_gomory_2013} gives an equality that can be interpreted as $\flt(d,1) \le 2 \flt(d,0)$. Based on the same idea, we can generalize this inequality as follows. 

\begin{proposition}
	\label{change:k:ineq}
	For any  $d \in \ZZ_{>0}$ and $k,m\in\ZZ_{\geq 0}$, one has
	\[
	\flt(d,k) \leq m \cdot \flt(d, \lfloor \tfrac {k} {m^d} \rfloor).
	\]
	In particular, 
	\[
	\flt(d,k) \le \left\lceil (k+1)^{1/d}  \right\rceil  \cdot \flt(d,0). 
	\]
\end{proposition}

\begin{proof}
	Let $K\subset \RR^d$ be a $k$-point body. By the pigeonhole principle, there is a coset of $\ZZ^d / m\ZZ^d$ that contains at most $\ell = \lfloor \tfrac k {m^d} \rfloor$ lattice points of $\inter(K)$. After applying a lattice translation, we may assume that this coset is the zero coset. Hence, $K$ is an $\ell$-point body with respect to $m\ZZ^d$. Thus,
	\[
	\tfrac 1m \wdt(K) = \wdt(K;m \ZZ^d\setminus\{\vo\}) \leq \flt(d, \ell),
	\]
	which proves the first part of the claim. The second part follows from the first one by taking $m = \left\lceil (k+1)^{1/d}  \right\rceil$. 
\end{proof}

An alternative access to the growth of $\flt(d,k)$ is through volume approximation. To this end, we consider the constant
\[
\flt(d,\infty) = \sup\{\wdt(K) : K\in\K^d,~ \vol(K)\leq 1\}.
\]
It was conjectured by Makai that $\flt(d,\infty)$ is realized by the simplex $\conv\{-\eins,\ve_1,\dots,\ve_d\}$ \cite{Makai}. The conjecture is proven to be true for $d\leq 3$ \cite{FejMak,Aliev} (cf.\ \eqref{eq:makai})


Intuitively, as $k$ grows, the value $\flt(d,k)/k^{1/d}$ should approach $\flt(d,\infty)$, since the latter quantity ``counts lattice points in an infinitesimally fine lattice''.
We can verify and quantify this convergence:

\begin{theorem}
	\label{thm:asymp}
	Let $d,k\in\NN$ and $k\geq 1$. Then,
	\[
	\left|\frac{\flt(d,k)}{k^{1/d}} - \flt(d,\infty)\right| \leq \frac{\flt(d,0)}{k^{1/d}}.
	\]
	In particular, we have $\lim k^{-1/d}\,\flt(d,k) = \flt(d,\infty)$ as $k\to\infty$.
\end{theorem}

For the proof we use bounds on the lattice point enumerator in terms of volume and covering radius.
Let $\mu(K)$ denote the covering radius of a convex body $K\subset\RR^d$ with respect to the integer lattice, i.e.,
\[
\mu(K) = \min\{\mu\geq 0 : \mu K + \ZZ^d = \RR^d\}.
\]
An alternative description of the covering radius is as the maximum dilate that admits a hollow translate, i.e.,
\[
\mu(K) = \max\{\mu\geq 0 : \exists \vt\in\RR^d ~\LEinter{\mu K-\vt} = 0\}.
\]
Therefore, the classical flatness constant may be expressed as
\begin{equation}
\label{eq:covrad_flatness}
\flt(d,0) = \max\{\mu(K)\cdot \wdt(K) : K\in \K^d\}.
\end{equation}
In \cite{dadush,mofmpaper} it was shown that
\[
\vol(K)^{1/d}(1-\mu(K)) \leq \LEinter{K}^{1/d} \leq \LE(K)^{1/d} \leq \vol(K)^{1/d}(1+\mu(K)),
\]
where for the lower bound $\mu(K) \leq 1$ is necessary, i.e., $K+\ZZ^d = \RR^d$.
Combining this with \eqref{eq:covrad_flatness} yields
\begin{equation}
\label{eq:wdt_discrepancy}
\vol(K)^{1/d} \big( 1 - \tfrac{\flt(d,0)}{\wdt(K)}\big) \leq \LEinter{K}^{1/d} \leq \LE(K)^{1/d} \leq \vol(K)^{1/d}\big(1+\tfrac{\flt(d,0)}{\wdt(K)}\big).
\end{equation}
We are now ready for the proof of Theorem \ref{thm:asymp}.

\begin{proof}[Proof of Theorem \ref{thm:asymp}]
	Let $\varepsilon>0$. Let $K_k\in\K^d$ be such that $\wdt(K_k)\geq\flt(d,k)-\varepsilon$, for $k\in\NN\cup\{\infty\}$.
	First, we will show that
	\begin{equation}
	\label{eq:asymp_lower_bound}
	k^{1/d} \geq \frac{\flt(d,k) - \flt(d,0)}{\flt(d,\infty)}.
	\end{equation}
	For this, we may assume that $\flt(d,k) > \flt(d,0)$. Otherwise, the inequality becomes trivial. For $1\leq k < \infty$ we have $\LEinter{K_k} \leq k$.
	We apply the lower bound of \eqref{eq:wdt_discrepancy} and obtain
	\[
	k^{1/d} \geq \LEinter{K_k}^{1/d} \geq \vol(K_k)^{1/d} \left(1-\frac{\flt(d,0)}{\wdt(K_k)}\right).
	\]
	By the homogeneity of width and volume and the definition of $\flt(d,\infty)$, we have
	\[
	k^{1/d} \geq \frac{\wdt(K_k)}{\flt(d,\infty)}\left(1-\frac{\flt(d,0)}{\wdt(K_k)}\right) \geq \frac{\flt(d,k) - \varepsilon - \flt(d,0)}{\flt(d,\infty)}.
	\]
	We obtain Equation \eqref{eq:asymp_lower_bound} by letting $\varepsilon\to 0$.
	For the reverse direction we want to prove
	\begin{equation}
	\label{eq:asymp_upper_bound}
	k^{1/d} \leq \frac{\flt(d,k) + \flt(d,0)}{\flt(d,\infty)}.
	\end{equation}
	After applying a dilation, we may assume that $\vol(K_\infty) = 1$. We choose $\lambda\in\RR_{>0}$ to be maximal such that $\LEinter{\lambda K_\infty} \leq k$ and, consequently, $k < \LEinter{(\lambda+\varepsilon)K_\infty}$.
	The upper bound from \eqref{eq:wdt_discrepancy} applied to $(\lambda+\varepsilon) K_\infty$ gives
	\[\begin{split}
	k^{1/d} & < \LEinter{(\lambda+\varepsilon)K_\infty}^{1/d}\leq(\lambda+\varepsilon) \left(1+\frac{\flt(d,0)}{\wdt((\lambda+\varepsilon) K_\infty)}\right)\\
	& = \left (\lambda + \varepsilon + \frac{\flt(d,0)}{\wdt(K_\infty)}\right) \leq \frac{\flt(d,k)+\flt(d,0)}{\flt(d,\infty)-\varepsilon} + \varepsilon,
	\end{split}\]
	where we used
	\[
	\lambda\wdt(K_\infty)= \wdt(\lambda K_\infty) \leq \flt(d,k)
	\]
	for the last inequality. Equation \eqref{eq:asymp_upper_bound} follows from letting $\varepsilon\to 0$.
\end{proof}

From Theorem \ref{thm:asymp} we obtain the upper bound
\begin{equation}
\label{eq:fltdk_ub}
\flt(d,k) \leq \flt(d,0) + \flt(d,\infty)k^{1/d}.
\end{equation}
While neither $\flt(d,0)$, nor $\flt(d,\infty)$ are determined explicitly, there are very good asymptotic bounds on these quantities. According to Reis and Rothvoss \cite{RR23} we have $\flt(d,0)\leq  O(d\log(d)^3)$, whereas an application of the best-known bounds on the so-called Mahler volume quickly yields
\begin{equation}
\label{eq:makai_weak}
\flt(d,\infty) \leq \tfrac 8 \pi (d!)^{1/d} \leq \left( \tfrac{8}{e\pi} + o (1)\right)d.
\end{equation}
We refer to \cite[page 3]{HenkXue} for details. As a consequence, we have
\begin{corollary}
	\label{cor:k_flatness_bound}
	Let $d,k\in\NN$ and $k\geq 1$. Then,
	\begin{equation}
	\label{eq:fltdk_asymp}
	\flt(d,k) 
	\leq \mathcal{O}(d(\log(d)^3 + k^{1/d})).
	\end{equation}	
\end{corollary}

\begin{remark} \label{wdt:0:remark}
	As a byproduct of Equation \eqref{eq:makai_weak}, we see that $\wdt(K) = 0$ when $K \in \K^d$ is not full-dimensional: 
Since width and volume are homogeneous functionals, we can read \eqref{eq:makai_weak} as
\begin{equation}
\label{eq:makai_weak2}
\wdt(K)\leq (c_d \vol(K))^{1/d},
\end{equation}
where $c_d = (\tfrac 8\pi)^d \,d!$, provided that $\vol(K) >0$. For $\vol(K)=0$ we have
\[
\wdt(K) \leq \wdt(K+\varepsilon B_d) \leq (c_d \vol(K+\varepsilon B_d))^{1/d}, 
\]
where $B_d$ denotes the Euclidean unit ball and $\varepsilon>0$ is arbitrary. This shows that \eqref{eq:makai_weak2} holds for all convex bodies $K\in\K^d$.

This insight is not entirely trivial, as one may see by considering $K$ being the irrational segment $K = \conv \{ (0,0), (\sqrt{2},1) \}$. With this choice of $K$, the infimum in the definition of $\wdt(K)$ is not attained, which suggests that there is a need for justification why the infimum is in fact $0$. Diophantine approximation could be one line of justification for this particular $K$. This suggests that the equality $\wdt(K) = 0$ for non-full-dimensional $K$ is a disguised qualitative version of the Diophantine approximation.
\end{remark}

From the asymptotic perspective, it is natural to ask for the ratio of $\flt(d,1)/\flt(d,0)$ as $d$ grows large. 
From \eqref{eq:fltdk_ub} we obtain an upper bound:
\begin{equation}
\label{eq:flt_ratio}
1 \leq \limsup_{d\to\infty} \frac{\flt(d,1)}{\flt(d,0)} \leq 1 + \lim\sup_{d\to \infty}\frac{\flt(d,\infty)}{\flt(d,0)} \leq 1 + \tfrac 4{e\pi}\approx 1.468,
\end{equation}
where we used a result of Mayrhofer, Schade and Weltge stating that $\flt(d,0)\geq 2d-o(1)$ \cite{MSW} in combination with \eqref{eq:makai_weak} to obtain the last inequality.

Most likely, the above estimate of $\flt(d,1)$ through $\flt(d,0)$ is not sharp asymptotically, even if Makai's conjecture was confirmed (in which case the upper bound would improve to $1+\tfrac 1 e$). 

Restricting ourselves to simplices, the following proposition shows that the ratio converges to $1$ as $d \to \infty$. Let $\flt_s(d,k)$ the maximum width of $k$-point simplices in dimension $d$. Then we have 
\begin{proposition} 
	$\flt_s(d,1) \le (1+ \frac{1}{d} ) \flt_s(d,0)$. In particular, 
	\[\lim_{d \to \infty} \frac{\flt_s(d,1) }{ \flt_s(d,0) } = 1.
	\] 
\end{proposition}
\begin{proof} 
	In dimension $d$, consider the $1$-point simplex $S$ of maximum lattice width. Let $\vc$ be the unique interior lattice point of $S$ and let $\vv_0,\ldots,\vv_d$ be the vertices of $S$. Then $\vc =\sum_{i=0}^d \lambda_i \vv_i$ with $\lambda_i \ge 0$ and $\sum_{i=0}^d \lambda_i =1$. Assume that $\lambda_0$ is the smallest barycentric coordinate. Then we have $\lambda_0 \le \frac{1}{d+1}$. After translation we can also assume that $\vv_0=0$. It follows that $ \frac{d}{d+1}  S$ is a hollow $d$-dimensional simplex. Consequently, $\flt_s(d,0) \ge \frac{d}{d+1} \flt_s(d,1)$ and the assertion follows. 
\end{proof} 

It would be interesting to understand the relation of $\flt_s(d,k)$ and $\flt(d,k)$. It seems plausible that $\flt_s(d,0) = \flt(d,0)$. For larger $k$, it might be that $\flt_s(d,k)$ and $\flt(d,k)$ are close to each other or even coinciding. Indeed, Theorems \ref{thm:hurkens}, \ref{thm:flt21} and \ref{thm:flt22} as well as Equation \eqref{eq:makai} seem to  support this conjecture.

\subsection{$\flt(d,k)$ is attained}
\label{ssec:realizers}

We clarify why the supremum in the definition of $\flt(d,k)$ is a maximum attained on a convex body. 

A part of our attainability argument is topological, as we show that $\flt(d,k)$ is the problem of  maximization of a continuous functional on a closed set. 
Let $B_d$ be the unit Euclidean ball in $\RR^d$ centered in $\vo$. The \emph{Hausdorff distance} of convex bodies  $K,L \in \K^d$ is the minimum $\rho \ge 0$ such that $K \subseteq L + \rho B_d$ and $L \subseteq K + \rho B_d$.
 The Hausdorff distance is a metric on $\K^d$ so that it is also called the Hausdorff metric. It is well-known that the family $\K^d$ of convex bodies endowed with the Hausdorff metric is a complete metric space. 
For more details, see \cite{schneider}. 

 
In order to keep the paper self-contained, we want to give a rigorous proof of the continuity of the lattice width.

\begin{lemma} 
	The lattice width functional on $\K^d$ is continuous with respect to the Hausdorff metric. 
\end{lemma}
\begin{proof} 
	Consider a sequence $(K_t)_{t \in \ZZ_{>0}}$ of convex bodies in $\K^d$ that converges to a convex body $K' \in \K^d$ in the Hausdorff metric. It is known that the volume is continuous in the Hausdorff metric. Thus, the limit of $\vol(K_t)$, as $t \to \infty$, is $\vol(K')$. If $\vol(K') = 0$, then by Equation \eqref{eq:makai_weak2}, $\wdt(K_t)$ is upper bounded by $( c_d \vol(K_t) )^{1/d}$, a value converging to $0$. So, $\wdt(K_t) \to 0$. Again by Equation \eqref{eq:makai_weak2}, $\wdt(K')=0$, as $K'$ is not full-dimensional. Consequently, $\wdt(K_t) \to \wdt(K')$ if $K'$ is not full-dimensional. 
	
	Consider the case of a full-dimensional limit $K'$. Pick a ball $\rho B_d + \vc$ of a radius $\rho>0$, which is a subset of $K'$. When $t$ is large enough, we have $K' \subseteq K_t + \frac{\rho}{2} B_d$ and $K_t \subseteq K' + \frac{\rho}{2} B_d$. Taking into account $\rho B_d + \vc \subseteq K'$, we obtain $\rho B_d + \vc \subseteq K_t + \frac{\rho}{2} B_d$. Applying the cancellation law for convex bodies, we obtain $\frac{\rho}{2} B_d + \vc \subseteq K_t$ for all sufficiently large $t$. The latter implies that $\wdt(K,\vu) \ge \rho \| \vu\|_2$ for $\vu \in \RR^d$, where $\| \vu \|_2$ is the Euclidean length of $\vu$. On the other hand, in view of $K_t \subseteq K' + \frac{\rho}{2} B_d$, one has $\wdt(K_t) \le \wdt(K'+ \frac{\rho}{2} B_d) =: \gamma$. Hence, if $\|\vu\|_2 > \frac{\gamma}{\rho}$, then $\wdt(K,u) > \wdt(K)$. This allows us to determine $\wdt(K_t)$ as a minimum $\wdt(K_t) = \min_{\vu \in U} \wdt(K_t,\vu)$ on the finite set $U$ of all $\vu \in \ZZ^d \setminus \{0\}$ that satisfy $\|\vu\|_2 \le \frac{\gamma}{\rho}$. As the width function is continuous in the Hausdorff metric for each fixed choice of direction $\vu$, we conclude that $\wdt(K_t) \to \wdt(K')$, as $t \to \infty$. 
 \end{proof}  

In the appendix of \cite{codenottifreyer} it is shown that $\flt(d,0)$ is attained on some convex body. Our aim is to provide a somewhat more general statement for generic flatness-type problems. 
\newcommand{\cX}{\mathcal{X}}
\begin{lemma} \label{gen:flt:attained}
	Let $\cX$ be a non-empty subfamily of $\K^d$, with the following properties: 
	\begin{enumerate} 
		\item All convex bodies in $\cX$ are full-dimensional. 
		\item $\cX$ is closed with respect to the Hausdorff metric. 
		\item Every convex body unimodularly equivalent to a convex body in $\cX$ belongs to $\cX$. 
		\item $\epsilon(\cX):=\inf \{\lambda_1(K-K) \colon K \in \cX \} >0$. 
	\end{enumerate} 
	Then, whenever $\flt(\cX):=\sup \{ \wdt(K) \colon K \in \cX \}$ is finite, it is attained on some $K' \in \cX$. 
\end{lemma} 
\begin{proof}
	By  Equation \eqref{eq:banaszczyk}, 
	\(
		\lambda_d((K-K)^\ast) \lambda_1(K-K) \le c\,d\log(d),
	\)
	where $c>0$ is an absolute constant. 
	This yields 
	\[
		\lambda_d((K-K)^\ast) \le \frac{c \,d\log(d) }{\epsilon(\cX)}
	\]
	for $K \in \cX$. 
	It is known that if $C$ is a $\vo$-symmetric convex body in $\RR^d$, then $d \cdot \lambda_d(C) \cdot C$ contains a basis of the lattice $\ZZ^d$, see \cite[equation (11) on page~68]{geometryofnumbers} . Thus, for
	\[
		\gamma_d(\cX) := \frac{c \, d^2 \log(d)}{\epsilon(\cX)}
	\]
	we have that 
	 every convex body $K \in \cX$ has the property that $\gamma_d(\cX) \cdot (K-K)^\ast$ contains a basis of the lattice $\ZZ^d$. Since  $(K-K)^\ast = \{ \vu \in \RR^d \colon \wdt(K,\vu) \le 1\}$,
	 this means that there exists a basis $\vb_1,\ldots,\vb_d$ of $\ZZ^d$ such that $\wdt(K,\vb_i) \le \gamma_d(\cX)$ for every $i  \in \{1,\ldots,d\}$. Applying a unimodular transformation to $K$, we thus may assume that a translation of $K$ is contained  in the box $[0,\gamma_d(\cX)]^d$. But then, using translations by lattice vectors, we can assume that $K$ is in the slightly expanded box $[0,\gamma_d(\cX)+1]^d$.   Consequently, $\flt(\cX)$ is the limit 
	 $\flt(\cX) = \lim_{t \to \infty} \wdt(K_t)$ 
	 for a sequence 
	 $(K_t)_{t \in \ZZ_{>0}}$ 
	 of convex bodies 
	 $K_t \in \cX$ satisfying 
	 $K_t \subseteq [0,\gamma_d(\cX)+1]^d$.  By  Blaschke's selection theorem (see \cite{schneider}), we can pass to a subsequence convergent in the Hausdorff metric. Let $K'$ be a compact convex set being the limit of $K_t$ in the Hausdorff metric, as $t \to \infty$. As $\cX$ is closed in the Hausdorff metric, we have $K' \in \cX$. Thus, since the lattice-width functional on $\K^d$ is continuous in the Hausdorff metric, we have $\wdt(K' ) = \lim_{t \to \infty} \wdt(K_t) = \flt(\cX)$. 
\end{proof} 

In order to apply Lemma~\ref{gen:flt:attained} to show that $\flt(d,k)$ is attained, we need to suggest a family $\cX$ of convex bodies with at most $k$ interior lattice points that satisfies $\flt(d,k) = \flt(\cX)$ and the assumptions of Lemma~\ref{gen:flt:attained}. 

\begin{lemma} \label{lambda1:bound:k}
	For every $d$-dimensional convex body $K \in \K^d$ with at most $k$ interior lattice points, one has 
	\[
		 \lambda_1(K-K) \ge \min \left\{ \frac{1}{\wdt(K)} \, , \,  \frac{1}{k+1} \left( 1 - \frac{\flt(d-1,0)}{\wdt(K)}  \right) \right\}. 
	\]
\end{lemma} 
\begin{proof} 
	We assume that $\lambda_1(K-K) < \frac{1}{\wdt(K)}$ and $1 - \frac{\flt(d-1,0)}{\wdt(K)} > 0$, as otherwise the assertion is satisfied. Assume, without loss of generality,  that $\lambda_1(K-K)$ is attained in direction $\ve_d$, so  that $\lambda \cdot (\vp - \vq) = \ve_d$ for some $\vp,\vq \in K$. As $\wdt(K) < \frac{1}{\lambda_1(K-K)}$, the width of $K$ is necessarily attained in a direction orthogonal to $\ve_d$, because the points $\vp$ and $\vq$ ensure that the width function of $K$ in directions not orthogonal to $\ve_d$ is at least  $\frac{1}{\lambda_1(K-K)}$. 
	
	We consider the coordinate projection $\pi(x_1,\ldots,x_d) := (x_1,\ldots,x_{d-1})$. Let $\vr = \pi(\vp) = \pi(\vq)$. We scale $K$ down to a smaller convex body $L = (1-\alpha) K + \alpha \vp$, where $0 < \alpha < 1$. If $\wdt(L) > \flt(d-1,0)$, then $\pi(L)$ is not hollow and so it contains an interior lattice point, which can be given as  $\vz = \pi((1-\alpha) \vv + \alpha \vp)$ with $\vv \in \inter(K)$. It follows that $I:=(1-\alpha) \vv + \alpha \cdot \conv \{\vp,\vq\}$ is a segment whose relative interior is a subset of $\inter(K)$ of length $\frac{\alpha}{\lambda_1(K-K)}$. The segment $I$ gets projected by $\pi$ onto the  lattice $\vz$. As $K$ has at most $k$ interior lattice points, the length of segment $I$ is at most $k+1$. We conclude that $\frac{\alpha}{\lambda_1(K-K)} \le k+1$, whenever $\wdt(L) > \flt(d-1,0)$. But
	$\wdt(L) = (1-\alpha) \wdt(\pi(K)) = (1- \alpha) \wdt(K)$. 
	So, we see that $(1-\alpha) \wdt(K) > \flt(d-1,0)$ implies $\alpha \le (k+1) \lambda_1(K-K)$. Taking $\alpha = 1 - \flt(d-1,0)/ \wdt(K) - \epsilon$, where $\epsilon>0$ is small, we have 
	\[
		1 - \flt(d-1,0) / \wdt(K)  - \epsilon  \le (k+1) \lambda_1(K-K). 
	\]
	Taking the limit as $\epsilon \to 0$, we obtain $1 - \flt(d-1,0) / \wdt(K) \le (k+1) \lambda_1(K-K)$. 
\end{proof}

\begin{theorem} 
For every $d \in \ZZ_{>0}$ and every $k \in \ZZ_{\ge 0}$, 
the generalized flatness constant $\flt(d,k)$ is attained on a compact convex set $K' \in \K^d$ with $\LEinter{K'} \le k$. Furthermore, every such $K'$ satisfies $\lambda_1(K'-K') \ge \frac{1}{(1+k) \flt(d,k) }$. 
\end{theorem}
\begin{proof}
	It is known that $\flt(d,k)$ is a finite positive number and from the definition of $\flt(d,k)$ it can be seen that $\flt(d,k) \ge \flt(d,0)$. So, let $\cX$ be the family of all convex bodies $K$ with at most $k$ interior lattice points that satisfy $\flt(d,0) \le \wdt(K)$. We will use Lemma~\ref{gen:flt:attained} to verify the assertion. For our choice of $\cX$, we have $\flt(\cX) = \flt(d,k)$. 
	Lemma~\ref{lambda1:bound:k} yields 
	\[
		\lambda_1(K-K) \ge \min
		\left \{ 
			\frac{1}{\flt(d,k)} , 
			\frac{1}{k+1} \left( 
				 1-  \frac{\flt(d-1,0)}{\flt(d,0)}
			\right) 
		\right \} 
	\]
	for every $K \in \cX$. 
	By Lemma~\ref{flt:monotonicity}, $\flt(d-1,0) \le \flt(d,0) -1$, which implies 
	\[
\lambda_1(K-K) \ge \min
\left \{ 
\frac{1}{\flt(d,k)} , 
\frac{1}{(k+1) \flt(d,0)} 
\right \}. 
	\]
	In view of  Proposition~\ref{change:k:ineq}, $\flt(d,k) \le \left\lceil (k+1)^{1/d} \right\rceil \flt(d,0) \le (k+1) \flt(d,0)$.

	This implies that $\inf \{ \lambda_1(K-K) \colon K \in \cX \} \ge \frac{1}{(k+1) \flt(d,k)}$ for every $K \in \cX$. Thus, the assumption $\epsilon(\cX)>0$ in Lemma~\ref{gen:flt:attained} is satisfied. 
	
	So, by  Lemma~\ref{gen:flt:attained}, $\flt(\cX) = \flt(d,k)$ is attained on some $K' \in \cX$ and that every such $K'$ satisfies $\lambda_1(K' - K') \ge \frac{1}{(k+1) \flt(d,k)}$. 
\end{proof} 

We can also derive an analogous statement for $\flt(d,\infty)$. 

\begin{theorem} 
 Let $d \in \ZZ_{>0}$. 
 Then	$\flt(d,\infty) = \sup \{\wdt(K) \colon K \in \K^d, \ \vol(K) \le 1\}$ is attained on some $K' \in \{ K \in \K^d \colon \vol(K) \le 1\}$. 
\end{theorem} 
\begin{proof} 
	We recall that $\flt(d,\infty)$ is the supremum of the lattice width functional over convex bodies $K \in \K^d$ that satisfy $\vol(K) \le 1$. 
When calculating $\flt(d,\infty)$ we can restrict attention to convex bodies with $\vol(K) =1$, because the lattice width of convex bodies $K$ with $\vol(K) = 0$ is zero and each convex body with the volume $0 < \vol(K) < 1$ can be re-scaled to a larger convex body with the volume $1$. Such a re-scaling makes the lattice width larger. As the lattice width of the unit cube $[0,1]^d$ is $1$ and its volume is also $1$, we can restrict attention to convex bodies of the lattice width at least $1$.

We thus can define $\flt(d,\infty)$ as  $\flt(d,\infty) = \flt(\cX)$, where $\cX = \{ K \in \K^d \colon \vol(K)=1, \wdt(K) \ge 1 \}$. This family satisfies the assumptions of Lemma~\ref{gen:flt:attained}. The only assumption whose verification is not straightforward is $\epsilon(\cX)>0$. 
Using Banaszczyk's transference theorem \eqref{eq:banaszczyk} as well as the lower bound in Minkowski's second theorem on successive minima
\[
\frac{2^d}{d!}\leq \vol(K-K)\lambda_1(K-K)\cdots\lambda_d(K-K),
\] 
the Rogers-Shephard inequality for the difference body
\[
\vol(K-K) \leq \binom{2d}{d} \vol(K),
\]
and our assumption $\vol(K) =1$, we eventually obtain  $\lambda_1(K-K) \ge \frac{2^d}{d! \cdot \binom{2d}{d} (c d\log(d))^{d-1} }$. Thus, one has $\epsilon(\cX) > 0$. 
Theorem~\ref{gen:flt:attained} implies that $\cX$ is attained on some $K' \in \K^d$. 
\end{proof} 

\subsection{$\flt(d,k)$ is attained on polytopes}
\label{ssec:max}
It turns out that $\flt(d,k)$ is attained on polyhedra and, for $k<\infty$, there are two classes of polyhedra that can be shown to contain optimizers of $\flt(d,k)$: lattice reduced bodies and $k$-maximal bodies. Lattice reduced bodies were introduced in \cite{codenottifreyer}. 
\begin{definition} 
	A $d$-dimensional convex body $K \in \K^d$ is called lattice reduced if for every convex body $L$  that is a proper subset of $K$ one has $\wdt(L) < \wdt(K)$. 
\end{definition} 

In~\cite{codenottifreyer}, the following has been shown. 

\begin{theorem} \label{reduced:thm} 
		Every $d$-dimensional convex body $K \in \K^d$ contains a reduced convex body of the same lattice width. Every $d$-dimensional reduced convex body is a polytope with at most $2^{d+1}-2$ vertices. 
\end{theorem} 

As an immediate corollary of this result, we obtain. 

\begin{theorem} 
	Let $d \in \ZZ_{>0}$. Then 
	every optimizer of $\flt(d,\infty)$ is a polytope with at most $2^{d+1}-2$ vertices and, for every $k \in \ZZ_{>0}$, there exists an optimizer of $\flt(d,k)$ that is a polytope with at most $2^{d+1} - 2$ vertices. 
\end{theorem} 
\begin{proof} 
	Consider an optimizer $K'$ of $\flt(d,\infty)$. If $K'$ is not reduced, by Theorem~\ref{reduced:thm}, there would exist a reduced convex body $L'$ properly contained in $K'$ with the same lattice width. But then $\frac{1}{\vol(L')^{1/d}} L'$ is a convex body of volume one and with a larger lattice width than $K'$, which contradicts the choice of $K'$. So, $K'$ is necessarily reduced and by Theorem~\ref{reduced:thm}, $K'$ is a polytope with at most $2^{d+1} -2$ vertices. 
	
	For $\flt(d,k)$ with $k \in \ZZ_{>0}$ the argument is similar. Consider an optimizer $K'$ of $\flt(d,k)$. Then a reduced convex body $L'$ being a subset of $K'$ and having the same lattice width is also an optimizer of $\flt(d,k)$ and, by Theorem~\ref{reduced:thm} , $L'$ is a polytope with at most $2^{d+1} - 2$ vertices. 
\end{proof} 

While lattice reduced optimizers of $\flt(d,k)$ arise from the incentive to make an optimizer of $\flt(d,k)$ inclusion-minimal, we may have the opposite incentive of making a given optimizer of $\flt(d,k)$ inclusion-maximal. For such inclusion-maximal sets, it will be convenient to distinguish the exact numbers of lattice points, which is why we use the following definition.

\begin{definition} \label{def:k-maximal}
Let $K\subset\RR^d$ be  closed convex set and let $k \in \ZZ_{\ge 0}$. We say that $K$ is $k$-maximal, if $\LEinter{K}=k$ and for any closed convex set $L\supseteq K$ with $\LEinter{L}=k$, we have $L=K$.
\end{definition}

\begin{proposition} \label{inclusion:maximization}
	Let $d \in \ZZ_{>0}$ and $k \in \ZZ_{\ge 0}$. Then every closed convex set $K$ with $k$ interior lattice points is a subset of a $k$-maximal closed convex set. 
\end{proposition} 
\begin{proof} 
	Let $Z = K \cap \ZZ^d$ and consider the partially ordered set consisting of all open convex set $U$ satisfying 
	$U \subseteq \inter(K)$. Zorn's lemma guarantees that this poset has a maximal element $U'$. The topological closure of $U'$ is a $k$-maximal closed convex set. 
\end{proof} 

A 0-maximal closed convex set is  also known as an ``inclusion maximal hollow'' body in the literature. Lov\'asz gave a characterization of 0-maximal closed convex sets (see also \cite{Av} for a short proof). 

\begin{theorem}[Lov\'asz, \cite{lovasz_2d}] 
	\label{lov:characterization}
	 The following statements are true. 
\begin{enumerate} 
\item A subset $K$ of $\RR^d$ is $0$-maximal if and only $K$ is a polyhedron with no interior lattice points and such that the relative interior of each facet of $K$ contains a lattice point. 
\item A unbounded subset $K$ of $\RR^d$ is $0$-maximal if and only if it is unimodularly equivalent to $\RR^m \times K'$, where $m \in \{1,\ldots,d-1\}$ and $K$ is a bounded $0$-maximal set in $\RR^{d-m}$. 
\item Every $0$-maximal subset of $\RR^d$ is a polyhedron with at most $2^d$ facets. Furthermore, there exist polytopes with $2^d$ facets that are $0$-maximal. 
\end{enumerate} 
\end{theorem} 

For $k=0$, we thus have the following. 
\begin{proposition} 
	For each $d \in \ZZ_{>0}$, the flatness constant $\flt(d,0)$ is attained on a $0$-maximal polytope with at most $2^d$ facets. 
\end{proposition} 
\begin{proof} 
	Let $K'$ be an arbitrary $0$-point convex body with $\wdt(K') = \flt(d,0)$ and a $0$-maximal closed convex set $L'$ that contains $K'$ as a subset. Such a set exists by Proposition~\ref{inclusion:maximization}. By Theorem~\ref{lov:characterization}, $L'$ is a polyhedron. This polyhedron is bounded. In fact, if $L'$ were unbounded, then up to unimodular transformation it would have the form $L' = \RR^m \times M'$, where $m \in \{1,\ldots,d-1\}$ and $M'$ is a bounded $0$-maximal set in dimension $d-m$. But in view of  Theorem~\ref{flt:monotonicity}, $\wdt(K') = \wdt(L') \le \wdt(M') \le \flt(d-m,0) < \flt(d,0)$, which is a contradiction to the choice of $K'$. Hence, $L'$ is a bounded $0$-maximal set on which $\flt(d,0)$ is attained. Theorem~\ref{lov:characterization} gives bound $2^d$ on the number of facets of $L'$.  
\end{proof} 

In the case $k>0$, $k$-maximal sets are necessarily bounded, but similarly to the case $k=0$, we can characterize $k$-maximal bodies as polytopes with exactly $k$ interior points and at least one lattice point inside each facet.

\begin{theorem}[{see \cite{dhelly}}]
\label{theorem:k-maximality}
Let $d, k \in \ZZ_{>0}$. Then a subset of $\RR^d$ is $k$-maximal if and only if it is a polytope with $k$ interior lattice points and the relative interior of each facet containing a lattice point. Furthermore, there exists a finite upper bound $h(d,k)$ on the number of facets of $k$-maximal polytopes in $\RR^d$. 
\end{theorem}

We refer to the lattice points in the relative interior of facets of $K$ as \emph{blocking points}.

As a direct consequence of Proposition~\ref{inclusion:maximization} and Theorem~\ref{theorem:k-maximality}, we obtain 

\begin{proposition} 
	For all $d \in \ZZ_{>0}$ and $k \in \ZZ_{\ge 0}$, $\flt(d,k)$ is attained on an $\ell$-maximal polytope with $\ell \in \{0,\ldots,k\}$. Furthermore, if $\ell < k$, then $\flt(d,\ell) = \flt(d, \ell+1) = \cdots = \flt(d,k)$. 
\end{proposition} 

If we knew that $\flt(d,k)$ is strictly increasing in $k$, we could strengthen the assertion of the above proposition by setting $\ell=k$. The following proposition implies that we can define $\flt(d,k)$ as $\sup \{ \wdt(K) \colon K \in \K^d, \LEinter{K}=k\}$, using convex bodies satisfying $\LEinter{K} =k$ 
rather than the bodies satisfying $\LEinter{K} \le k$, but in such a definition it is not clear whether the supremum is attained.

We use it to prove that the $\leq$ sign in the definition of $\flt(d,k)$ can (almost) be replaced by an $=$ sign.

\begin{proposition}\label{prop:monotonicity_k}
For any $d \in \ZZ_{>0}$ and $k,\ell\in\ZZ_{\geq 0}$ with $k\leq \ell$ we have
\[
\sup\{\wdt(K) \colon K\in \K^d,~\LEinter{K}=k\} \leq \sup\{\wdt(K) \colon K\in \K^d,~\LEinter{K}=\ell\}.
\]
\end{proposition}
\begin{proof}
	It suffices to prove the assertion for $\ell = k+1$. If $k=0$, fix a $0$-maximal convex body $K$ that attains $\flt(d,0)$. As we have observed, such a body $K$ is a polytope. In the case $k=1$, take $K$ to be an $k$-maximal polytope.  
	Consider a facet $F\subset K$ and let $Z_F = \relint(F) \cap\ZZ^d$. Moreover, let $\vz$ be a vertex of $\conv( Z_F )$. We choose a $(d-2)$-plane $A\subset \aff(F)$ that separates $\vz$ from $\conv( Z_F \setminus\{\vz\})$. Rotating $F$ around $A$ slightly yields a new polytope $K'$ with $\inter(K')\cap\ZZ^d = (\inter(K)\cap\ZZ^d) \cup \{\vz\}$ whose width is arbitrarily close to the width of $K$. This shows that
\[
\sup\{\wdt(K) \colon K\in \K^d,~\LEinter{K}=k\} \leq \sup\{\wdt(K) \colon K\in \K^d,~\LEinter{K}= k+1\}.\qedhere
\]
\end{proof}

As a consequence, we see that
\begin{equation}
\label{eq:fltdk_eq}
\begin{split}
\flt(d,k) &= \sup\{ \wdt(K) : K\in\K^d,~\LEinter{K} = k\}.
\end{split}
\end{equation}
Unlike the maximum in \eqref{eq:fltdk} it is not clear that the supremum in \eqref{eq:fltdk_eq} is attained. The main difference is that the set of convex bodies that satisfies $\LEinter{K}=k$ is not closed for $k>0$, which is why the usual compactness arguments don't apply. 

\subsection{$\flt(d,k)$ is an algebraic number} We show that $\flt(d,k)$ is algebraic for all $d, k$. Previously, it was not known if even $\flt(d, 0)$ is always an algebraic number. 


\begin{theorem} \label{alg:number}
	Let $d \in \ZZ_{\ge 1}$ and $k \in \ZZ_{\ge 0}$. Then $\flt(d,k)$ is an algebraic number. 
\end{theorem} 
\begin{proof}
	Above we have shown that $\flt(d,k)$ is attained on a polytope. So let $P$ be a polytope on which 
	$\flt(d,k)$ is attained, let 
	$Z = \inter(P) \cap \ZZ^d$ and let $m$ be the number of vertices of $P$. Let us sandwich $P$ between two rational polytopes $A$ and $B$ with the property that $A \subseteq \inter(P) \subseteq P \subseteq \inter(B)$ and 
	$A \cap \ZZ^d = Z$. Let $W = B \cap \ZZ^d$. 
	Consider polytopes of the form $Q = \conv(\vq_1,\ldots,\vq_m)$ with the property that $\vq_1,\ldots,\vq_m \in B$, $Q \supseteq A$ and $Q \cap \ZZ^d = Z$. 
	We  show that the condition 
	\[
		C(t,Q) :=  \quad \bigl(t = \wdt(Q)\bigr) \wedge \bigl(A \subseteq Q \subseteq B\bigr) \wedge \bigl(Q \cap \ZZ^d = Z\bigr)
	\]
	can be formulated as a first-order formula over reals. We note that $Q$ is given by $\vq_1,\ldots,\vq_m \in \RR^d$, so that this condition is a constraint involving in $1 + dm$ variables $(t,q_1,\ldots,q_m)$. 
	For the condition $t = \wdt(Q)$ to be phrased as a first-order formula, we observe that the fact $A \subseteq Q \subseteq B$ allows us to formulate $\wdt(Q)$ as a finite mini-max $\wdt(Q) = \min_{\vu \in U} \max_{i,j= 1,\ldots,m}  \langle \vq_i - \vq_j, \vu \rangle $, where $U \subseteq \ZZ^d \setminus \{\vo\}$  is the set of all vectors $\vu \in \ZZ^d \setminus \{\vo \}$ satisfying $\wdt(A,\vu) \le \wdt(B)$. This shows that $t = \wdt(Q)$ can be formulated via a first-order formula over reals. The condition $A \subseteq Q$ means that for every vertex $\va$ of $A$ one has $\va \in Q$, where $\va \in Q$ can be phrased as the existence of a representation as a convex combination $\va = \sum_{i=1}^m \lambda_i \vq_i$, with $\lambda_i \ge 0$ and $\sum_{i=1}^m \lambda_i = 1$. This can be described via first order formula, too. The condition $Q \subseteq B$ can be phrased as $\vq_i \in B$ for each $i  \in \{1,\ldots,m\}$ and so it can be formulated as a system of linear inequalities with rational coefficients using the inequality description of $B$. Finally, the condition $Q \cap \ZZ^d = Z$ can be formulated as the fact that among points of $W$, the points of $Z$ belong to $Q$, while the points of $W \setminus Q$ do  not belong to $Q$. Again, this can be phrased as the existence of a convex combination $\vz = \sum_{i=1}^m \lambda_i \vq_i$ for the points in $\vz \in Z$ and as a negation of the existence of a convex combination $\vw = \sum_{i=1}^m \lambda_i \vq_i$ for the points $\vw \in W \setminus Z$. 
	
	It follows that the condition $D(t)$ describing that $t$ is the lattice width of some polytope $Q$ of the form $Q = \conv(\vq_1,\ldots,\vq_m)$ that satisfies 
	$A 
	\subseteq Q \subseteq B$ and 
	$Q \cap \ZZ^d = Z$ is a first-order formula over reals that is written using coefficients in the field $\QQ$. 
	
	By quantifier elimination algorithm from real algebra, we can re-formulate the formula $D(t)$ using a quantifier-free formula with the coefficients in $\QQ$. Such a quantifier-free reformulation of $D(t)$ is a  Boolean combination 
	$\Phi (g_1 (t) \ge 0,\ldots, g_N(t) \ge 0)$ of polynomial inequalities with coefficients in $\QQ$. That is, $\Phi$ is a Boolean formula and $g_1,\ldots,g_N$ are polynomials in $\QQ[t]$. The formula $D(t)$ describes the set 
	$T:=\{ t \in \RR \colon D(t) \ \text{is satisfied} \}$. By construction, the every element of $T$ is the lattice-width of some polytope $Q$ that contains at most $k$ lattice points. Hence, $\sup T \le \flt(d,k)$. 
	On the other hand, by taking $\vq_1,\ldots,\vq_m$ arbitrarily close to the $m$ vertices of $P$, the polytope $Q = \conv(q_1,\ldots,q_m)$ gets arbitrarily close to the maximizer $P$ of $\flt(d,k)$. Even more specifically, we can pick $Q \subseteq P$ that is a slightly smaller homothetical copy of $P$. 
	Consequently, $\sup T = \flt(d,k)$ and $\sup T - \epsilon  \in T$ for every sufficiently small $\epsilon>0$. This shows that in a small neighborhood of $\sup T$ as $t$  increases and passes through $\sup T$, the truth value of the Boolean formula $\Phi (g_1 (t) \ge  0,\ldots, g_N(t) \ge 0)$ gets flipped from true to false. This is only possible if some of the polynomials $g_1,\ldots,g_N$ is zero at $\sup T$. In fact, if none of these polynomials were zero at $\sup T$, their sign would not change in a small neighborhood of $\sup T$, so that the truth value of the formula for $t$ would remain the same, which is a contradiction. It follows that $\sup T = \flt(d,k)$ is a zero of a polynomial from $\QQ[t]$. Hence, $\flt(d,k)$ is an algebraic number. 
 \end{proof} 

\begin{remark}
	For background in  real algebraic geometry, including Tarski's theorem and  quantifier elimination, we refer to \cite{bochnak2013real} and \cite{scheiderer2024course}. 
\end{remark} 

\begin{remark} \label{rem:alg-degree}
		The argument of Theorem~\ref{alg:number} is local as it only involves polytopes that are sandwiched between $A$ and $B$, where $A$ and $B$ can be chosen arbitrarily close to the optimizer $P$ of $\flt(d,k)$. 
		Hence, the same argument also shows that if a polytope $P$ is a local optimizer of $\flt(d,k)$, then its width $\wdt(P)$ is still an algebraic number. That is, all local optima of the $\flt(d,k)$ problem are algebraic numbers. For example, for the problem of determination of $\flt(2,0)$, apart from the global optimum $1 + \frac{2}{\sqrt{3}}$, one also has a local optimum $2$ attained on the quadrilateral $\conv \{ (3/2,1/2), (1/2,3/2), (-1/2,1/2), (1/2,-1/2)\}$. 
\end{remark} 

\begin{remark} 
	One may wonder if Theorem~\ref{alg:number} can be further improved by narrowing the underlying field. The flatness constant $\flt(2,0) = 1 + \frac{2}{\sqrt{3}}$ is an algebraic number of degree $2$, whereas the flatness constants $\flt(2,1) = 3$ (determined in this paper) and $\flt(2,2) = 10 / 3$ are rational (in other words, they algebraic numbers of degree $1$). Is it possible that $\flt(d,k)$ is a rational number when $k> 0$? Can we say something more specific about the degree of the algebraic numbers $\flt(d,0)$? For the local optimizer $P$ of $\flt(3,0)$ determined  in \cite{acms_local_3d_2021}, the degree of the algebraic number $\wdt(P) = 2 + \sqrt{2}$ is $2$. For local optimizers $P$ of $\flt(d,0)$ that have been found  \cite{MSW} for $d=4$ and $d=5$, the degree of the algebraic number $\wdt(P)$ is $4$ and $2$, respectively. If these local optimizers are actually global, then it looks as if the degree does not grow. Is it possible that there is a uniform bound on the degree of the algebraic numbers $\flt(d,k)$? Is it possible that the degree of the algebraic number $\flt(d,0)$ is $2$ when $d \ge 3$ is odd? Resolving these questions would require some further ideas. 
\end{remark} 

\begin{remark} 
	We remark that $\flt(d,\infty)$, too, is an algebraic number, as follows from the reformulation of $\flt(d,\infty)$ in algebraic terms. This can be seen via connecting $\flt(d,\infty)$ to the reverse isodiametric problem, which has a completely algebraic formulation. See, for example, \cite{Aliev}.  
\end{remark}

\section{Planar convex bodies with at most one interior point}
\label{sec:flt21}


In this section we prove Theorem \ref{thm:flt21}, that is, we show that $\flt(2,1) = 3$.

\subsection{1-maximal polygons and their vertices} \label{sec:shards}

In \cite{hurkens}, Hurkens handled the case of convex bodies with no interior lattice points, i.e., he computes $\flt(2,0)$, so we now have to treat convex bodies with exactly one interior lattice point. We can further reduce to treating only $1$-maximal convex bodies (cf.\ Definition \ref{def:k-maximal}). That is, it suffices to show the following statement.

%

\begin{proposition} \label{prop:1-maximal}
  Let $K$ be a $1$-maximal convex body. Then $\wdt(K) \leq 3$ and equality holds if and only if $K$ is unimodularly equivalent to $3\Tst$.
\end{proposition}


%
%
%

The notion of blocking points (cf.\ Section \ref{ssec:max}) allows us to associate to any $1$-maximal polygon a lattice polygon that we call \emph{blocking polygon}.
\begin{definition} \label{def:blocking-polytope}
  Let $P\subset\RR^2$ be a $1$-maximal polygon. Its \emph{blocking polygon} $\BP$ is the convex hull of the blocking points of $P$, that is,
\[
  \BP = \conv \{ \vec b \in \ZZ^2 :  \vec b \in \relint(e)\text{ for some edge }e\subset P \}.
  \]
We sometimes refer to $P$ as a \textit{circumscriber} of $\BP$. 
\end{definition}

Our first goal is to show that the vertices of the circumscriber $P$ in the above definition correspond in a weak sense to the edges of its blocking polygon $\BP$ and that they necessarily lie in specific regions that depend on the lattice point structure of $\BP$.

We let $B=\{\vx\in\RR^2 : \langle \va_i,\vx\rangle\leq b_i,~0\leq i \leq m\}$ be a lattice polygon. For an edge $e = \{\vx \in B : \langle \va_j,\vx\rangle = b_j\}$, $0\leq j \leq m$, of $B$ we define the polyhedron $S_e$ obtained from $B$ by reversing the inequality constraint responsible for $e$, i.e., 
\begin{equation}
\label{eq:shard_definition_formal}
S_e = \{\vx\in\RR^2 : \langle \va_j,\vx\rangle\geq b_j,~\langle \va_i,\vx\rangle \leq b_i,~i\neq j\}.
\end{equation}
Moreover, for $\vx_0\in\inter (B)$, we consider the face cone 
\[
F(\vx_0,e) = \pos(e-\vx_0) + \vx_0,
\]
where $\pos(X) = \{\lambda \vx : \vx\in X,~\lambda\geq 0\}$.
These polyhedra are related as follows.

\begin{lemma}
\label{lemma:fcone}
In the above setting, we have
\[
S_e = \bigcap_{\vx_0\in B\setminus e} F(\vx_0,e).
\]
\end{lemma}

\begin{proof}
Let us write $B = \{ \vx\in\RR^2 : \langle \va_i,\vx\rangle \leq b_i,~0\leq i \leq m\}$, where $\va_i\in\RR^2$, $b_i\in\RR$, $m\in\NN$. We may assume that the inequality description is irredundant. Without loss of generality, let $e$ be the edge defined by the outer normal vector $\va_0$. Then,
\begin{equation}
\label{eq:be_eq}
S_e = \{ \vx\in\RR^2 : \langle \va_0,\vx\rangle \geq b_0,~ \langle \va_i,\vx\rangle \leq b_i,~1\leq i \leq m\},
\end{equation}
and we define $\overline S_e = B\cup S_e =  \{ \vx\in\RR^2 : \langle \va_i,\vx\rangle \leq b_i,~1\leq i \leq m\}$.

First, let $\vy\in S_e$ and $\vx_0\in B\setminus e$. The segment $[\vx_0,\vy]$ intersects $S_e \cap \{\langle \va_0,\vx\rangle = b_0\} = e$ in a point $\vx_1\neq \vx_0$. It follows that there exists some $\lambda > 0$ such that $\vy = \vx_0 +\lambda (\vx_1 - \vx_0)\in F(\vx_0,e)$.

For the reverse inclusion, consider $\vy\in \RR^d \setminus S_e$. Then, $\vy$ violates one of the inequalities defining $S_e$ (cf.\ \eqref{eq:be_eq}). First, we assume that $\langle \va_0,\vy \rangle < b_0$. There exists $\vx_0$ in $B\setminus e$ such that $\langle \va_0,\vy\rangle < \langle \va_0,\vx_0 \rangle$ and, thus, $\langle \va_0,\vx \rangle > \langle \va_0,\vy\rangle$ holds for all $\vx\in F(\vx_0,e)$.

Finally, assume that $\langle \va_i, \vy\rangle > b_i$ for some $1\leq i \leq m$. There exists $\vx_0 \in B\setminus e$ with $\langle \va_i,\vx_0\rangle = b_i$. Since $e\subset B \subset \{\langle\va_i,\vx\rangle \leq b_i\}$ we have $\langle \va_i,\vx\rangle \leq b_i$ for all $\vx\in F(\vx_0,e)$. Thus, in both cases, there exists $\vx_0\in B\setminus e$ with $\vy\not\in F(\vx_0,e)$.
\end{proof}


\begin{lemma}
\label{lemma:initial_shards}
Let $B$ be a lattice polygon with edge set $E(B)$ and let $P$ be a 1-maximal polygon with $\BP=B$. For $e\in E(B)$ let $S_e$ be the region described in Equation \eqref{eq:shard_definition_formal}. Then for each $e\in E(B)$ there exists a point $\vp_e\in S_e$ such that $P = \conv\{\vp_e : e\in E(B)\}$.
\end{lemma} 

\begin{proof}
Fix an edge $e\in E(B)$ and $\vx_0\in B\setminus e$. Since the cones $(F(\vx_0,e))_{e\in E(P)}$ form a fan, each vertex $\vp$ of $P$ is contained in some $F(\vx_0,e)$. Within $F(\vx_0,e)$ there can be at most one vertex of $P$. Otherwise, there exists an edge of $P$ that is separated from $B$ by the affine hull of $e$. This would contradict $\BP=B$. 

If $\inter(F(\vx_0,e))$ does contain a vertex of $P$ we 
denote it by $\vp_e(\vx_0)$. Note that $\vp_e(\vx_0) \notin B$, since it is a vertex of $P$ and $P$ circumscribes $B$. Moreover, $\vp_e(\vx_0)$ is contained in two edges of $P$ that support $B$ at the endpoints of $e$. Let $\vx_1$ be a point in $B\setminus e \subset P$.
 Then the segment $[\vp_e(\vx_0), \vx_1]$ is contained in $P$ and, thus, crosses $e$. Hence, $\vp_e(\vx_0)\in F(\vx_1,e)$ and we may set $\vp_e = \vp_e(\vx_0)$.
If there exists no $\vx_0\in B\setminus e$ such that $F(\vx_0,e)$ contains a vertex of $P$, we define $\vp_e$ to be some point in the relative interior of $e$. In either case, we have $\vp_e \in S_e$ (cf.\ Lemma \ref{lemma:fcone}).

By the fan property of $(F(\vx_0,e))_{e\in E(B)}$, the vertices of $P$ are covered by the points $\vp_e$.
\end{proof}

We observe that the interiors of the regions $S_e$ are pairwise disjoint. Moreover, we can further restrict the regions where the vertices of $P$ can lie by taking into account the lattice point structure of $P$.

\begin{lemma}
\label{lemma:forbidden_cones}
Let $P$ be a 1-maximal polygon with interior lattice point $\vp$. Then, for any $\vq\in\ZZ^2\setminus\{ \vp\}$ we have
\[
P \cap (- \inter((\pos(\BP-\vq)) + \vq) = \emptyset.
\]
\end{lemma}

In what follows we write $\sigma_\vq = - \inter(\pos(\BP-\vq) )+ \vq$.

\begin{proof}
Since $\BP$ is 2-dimensional, we have $ \sigma_\vq = -\pos(\inter(B)-\vq) + \vq$. Towards a contradiction, let $\vx\in P \cap \sigma_\vq$. By definition of $\sigma_\vq$, for some $\vy \in \inter(B)$ we have $\vq \in [\vx, \vy]$. Since $B\subset P$, we have $\vy\in\inter(P)$ and, thus, $\vq\in\inter(P)$. This contradicts the fact that $\vp$ is the unique interior lattice point of $P$.
\end{proof}

In summary, Lemmas \ref{lemma:initial_shards} and \ref{lemma:forbidden_cones} give us a good idea of the vertices of a 1-maximal polygon $P$ with $B_P = P$; First, we obtain from Lemma \ref{lemma:initial_shards} that we have (at most) one vertex $\vp_e$ per edge $e$ and that it is contained in the polyhedron $S_e$. In some cases, all points in the interior of a region $S_e$ are possible as vertices of $P$, for instance if $B$ is a hexagon (cf.\ Section \ref{sec:hexagon}). In other cases, we can further reduce the set $S_e$ with the help of Lemma \ref{lemma:forbidden_cones} which gives
\begin{equation}
\label{eq:final_shards}
\vp_e \in S_e \setminus \bigcup_{\vq\in\ZZ^2\setminus\{\vp\}} \sigma_q.
\end{equation}
This yields a proper subset of $S_e$ if, for instance, $B$ is a symmetric quadrilateral (cf.\ Section \ref{subsub:sym_quad}).

\subsection{Classifying the blocking polygons}
\label{ssec:class}


In this section we list (up to unimodular equivalence) the possible blocking polygons that we can see for a 1-maximal polygon.

\begin{lemma} \label{lemma:disjoint-casework-on-blocking-polytopes}
  Let $P$ be a $1$-maximal polygon. Then its blocking polytope \hypertarget{def:BP}{$B_P$} is a 2-dimensional lattice polygon. Moreover exactly one of the following holds:

  \begin{mycases}
    \item \label{case:empty} $\BP$ is an empty polygon (either the standard triangle or the unit square),
    \item \label{case:hollow} $\BP$ is a hollow polygon with lattice points in the relative interior of some of its edges,
    \item \label{case:reflexive} $\BP$ is a polygon with exactly one interior lattice point. 
  \end{mycases}
\end{lemma}

\begin{proof}
  We first show that $\BP$ is 2-dimensional. By Theorem \ref{theorem:k-maximality}, each edge of $P$ contains at least one blocking point. We need to show that three blocking points in distinct edges cannot be colinear. Indeed, take two such blocking points. Since they each lie in the relative interior of a different edge, the open segment connecting them is in the interior of $P$, and therefore cannot contain any blocking point.

  Since an interior lattice point of $\BP$ is also an interior lattice point of $P$, its clear that $\BP$ cannot have more than one interior lattice point, otherwise $P$ would not be $1$-maximal. 
\end{proof}


We require the following lemma in order to reduce the sets of blocking polygons to a finite list.

\begin{lemma} \label{lemma:three-colinear-boundary-points}
  Let $P$ be a $1$-maximal polygon, with three colinear lattice points on an edge of $P$. Then $\wdt(P) \leq 3$ with equality attained exactly if $P$ is unimodularly equivalent to $3 \Delta_2$. 
\end{lemma}

\begin{proof}
Suppose that $P$ contains an edge $f_0$ with three colinear points. After a unimodular transformation, we can assume that
\begin{equation}
\label{eq:pointsine}
(-1,-1),(0,-1),(1,-1) \in f_0.
\end{equation}
and that $P\subset \{y \geq -1\}$. Let $\vp$ be the interior point of $P$. We have $p_y = 0$; Otherwise, the triangle $\conv\{(0,-1),(1,-1),\vp\}$ contains an integer point which is not one of its vertices. Such a point would be an additional interior point of $P$.
After applying a shear, we can, thus, assume that $\vp = (0,0)$.

Let $\ell_0$ denote the length of $f_0$ and let $w_y = \wdt(P,\ve_2)$. We can assume that $w_y \geq 3$.  If $\ell_0 > 3$ it follows from the convexity of $P$ that the length of $P \cap \{y=1\}$ is strictly greater than 1. 
Indeed, let $g(t) = \vol_1(P\cap\{y=t\})$ be the function that records the length of a horizontal section of $P$. It is concave on its support. Thus,
\begin{equation}
\label{eq:brunn}
g(1) \geq \tfrac 23 g(2) + \tfrac 13 g(-1) > 1.
\end{equation}
Hence, there exists an additional lattice point in the interior of $P$, a contradiction. We deduce that $\ell_0\leq 3$. Now it suffices to prove $\ell_0 = \wdt(P,\ve_1)$ in order to show the claimed inequality.

  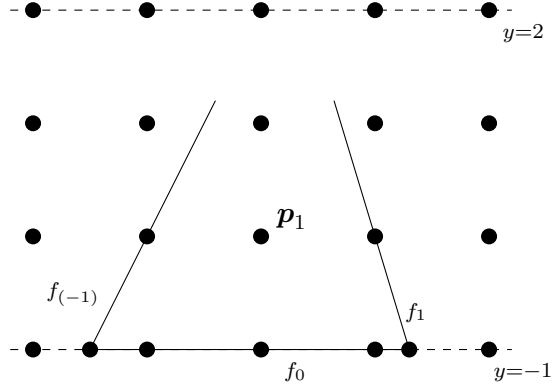
\begin{figure}[h]
  \begin{tikzpicture}[scale=1.5]

  \fill (-1.5, 0) circle (2pt);
  \fill (1.3, 0) circle (2pt);

  \draw (1.3, 0) -- (-1.5,0); 

  \draw (1.3, 0) -- (1-.3*12/10, 1+ 12/10);
  \draw (-1.5, 0) -- (-1 + 12/10*.5, 1+12/10);

  \draw[dashed] (-2.2, 3) -- (2.2, 3); 
  \node at (2.3, 2.8) {${\scriptstyle y = 2}$};
  \draw[dashed] (-2.2, 0) -- (2.2, 0);
  \node at (2.3, -.2) {${\scriptstyle y = -1}$};
  \node at (.27, 1.15) {$\vp_1$}; 
  \node at (.3, -.2) {${\scriptstyle f_0}$};
  \node at (1.36, .33) {${\scriptstyle f_1}$};
  \node at (-1.65, .49) {${\scriptstyle f_{(-1)}}$};
    \foreach \x in {-2, -1, 0,...,2}{
    \foreach \y in {0,..., 3}{
      \fill (\x,\y) circle (2pt);
    }
  }
  \end{tikzpicture}
  \caption{The edges $f_{(-1)}$ and $f_1$, emanating from $f_0$. As they separate the points $(-1, 0), (1,0)$ from $P$, respectively, we see that $w(P, \vec e_1) = \ell_0$.}
  \label{fig:3-points-lemma}
\end{figure}

To this end, let $f_{(-1)}$ and $f_1$ be the adjacent edges of $f_0$ in $P$ ordered from left to right, as in Figure \ref{fig:3-points-lemma}. Since $P$ is 1-maximal, $f_{1}$ contains an integer point in its relative interior. It follows that the vertex $\vp_{1}$ of $f_1$ which is not contained in $f_0$ has $p_{1,y} > 0$. Since $(0,0)$ is the unique interior point of $P$ we now see that $(1,0)$ is separated from $\inter(P)$ by $f_1$. In the same way, we see that $(-1,0)$ is separated from $\inter(P)$ by $f_{(-1)}$. It follows that the orthogonal projection of $P$ on $\{y=-1\}$ is $f_0$, which proves $\wdt(P,\ve_1) = \ell_0$ and, thus, $\wdt(P) \leq 3$.

Suppose that equality holds. Then we have $\ell_0 = w_y = 3$; if either of the parameters was strictly larger than 3 while the other is equal to 3, there would be an interior point on $P\cap \{y=1\}$ (cf.\ \eqref{eq:brunn}). Now it follows that $P\cap \{y=0\}$ is a segment of length at least 2 that contains $(0,0)$ as its unique interior point. Consequently, we have $P\cap \{y=0\} = [(-1,0),(1,0)]$. Similarly, we see that $P\cap\{y=1\}$ is a unit segment between two lattice points. From this we see that $P = \conv\{(-2,-1),(1,-1),(1,2)\}$ or $P =\conv\{(-1,-1),(2,-1),(-1,2)\}$.
\end{proof}

To prove Theorem \ref{thm:flt21} we want to show, for each possible lattice polytope $B$ in the classification of Proposition \ref{lemma:disjoint-casework-on-blocking-polytopes}, that all polytopes $P$ with blocking polytope $\BP=B$ satisfy $\wdt(P)\leq 3$. We will now see that Lemma \ref{lemma:three-colinear-boundary-points} allows us to take care of all but finitely many cases, collected in Proposition~\ref{prop:remaining-casework} and handled in the subsequent sections.

First consider $\BP$ from Case~\ref{case:empty} of Proposition~\ref{lemma:disjoint-casework-on-blocking-polytopes}, that is, either a unimodular triangle or unit square.
We can eliminate $\BP$ being the unit square. Indeed, if $\BP$ is a square then the interior lattice point of $P$ has the same parity as one of its blocking points. The midpoint between the interior point and that blocking point would be an additional interior lattice point, a contradiction.

Thus, up to unimodular equivalence, the only option for $\BP$ in Case~\ref{case:empty} is
\[
\conv\{(-1,-1),(-1,0),(0,-1)\}.
\]
The unique interior points $\vp$ of $P$ must be such that its convex hull with $\BP$ contains no additional points beyond $\vp$ and the lattice point of $\BP$. One sees that this amounts to $\vp\in\{(-2,0),(0,-2),(0,0)\}$. The three possible choices of $\vp$ are unimodularly equivalent, so we can assume that $\vp = (0,0)$. 

Next, consider $\BP$ as in Case~\ref{case:hollow}, that is, a hollow lattice polygon with a  lattice point $\vp$ in the relative interior of one of its edges $e$. This lattice point is either in the interior of $P$ or it is a lattice point in the relative interior of an edge of $P$; see Figure \ref{fig:3-colinear-points}. 

If $\vp$ is in the boundary of $P$, then so is $e$. In particular, $P$ has an edge that contains 3 lattice points and Theorem \ref{thm:flt21} is obtained from Lemma \ref{lemma:three-colinear-boundary-points}. So we may assume that all edges of $\BP$ that are contained in the boundary of $P$ have lattice length 1.
In that case $e$ passes through the interior of $P$ and it is the only edge of $P$ to do so. Since $P$ contains only one interior point, the lattice length of $e$ is 2 and $\vp$ is the unique interior point of $P$.


\begin{figure}[h]
	\begin{tikzpicture}
  \fill[gray!50] (-3, 0) -- (-3, .6) -- (-1, .6) -- (-1, 0) -- cycle;
  \fill[gray!50] (1, 0) -- (3, 0) -- (3, .6) -- (1, .6);

  \draw[thick] (-3, 0) -- (-3, .6);
  \draw[thick] (-1, 0) -- (-1, .6);

    \draw[thick] (3, 0) -- (3, .6);
  \draw[thick] (1, 0) -- (1, .6);

	\draw[very thick, blue!40] (-2.3, -.6) -- (-3.4 , 4/7 * .6);
	\draw[very thick, blue!40] (-2.3, -.6) -- (.1, 11/13 * .6);
	\draw[very thick, blue!40, opacity = .8] (.4, 0) -- (3.3, 0);
	\draw[thick] (-3, 0) -- (-1, 0);
	\draw[thick] (1, 0) -- (3,0);
	\foreach \x in {-4, -3, -2, -1, ...,3, 4}{
		\foreach \y in {-1, 0, 1}{
			\fill (\x,\y) circle (2pt);
		}
	}
	\fill[red] (-3, 0) circle (2pt);
	\fill[red] (-2, 0) circle (2pt);
	\fill[red] (-1, 0) circle (2pt);
	\fill[red] (1, 0) circle (2pt);
	\fill[red] (2, 0) circle (2pt);
	\fill[red] (3, 0) circle (2pt);

  \node at (-2, .3) {$\BP$};
  \node at (2, .3) {$\BP$};
	\end{tikzpicture}
	\cprotect\caption{When $\BP$ has three colinear boundary points (visualized in red), either the inside point is an interior lattice point of $P$, as on the left, or all three points lie on an edge of $P$, as on the right. In this situation Lemma \ref{lemma:three-colinear-boundary-points} applies.}
	\label{fig:3-colinear-points}
\end{figure}
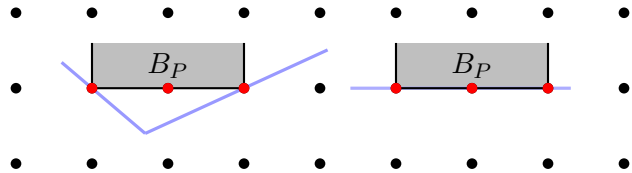

A classification of hollow lattice polygons is given in \cite{Rabinowitz89}: a hollow polygon is (up to unimodular equivalence) either $2 \Delta_2$ or the convex hull of two lattice segments, embedded in $\{y=0\}, \{y=1\}$. 
Consider then $\BP$ as in Case~\ref{case:hollow} with exactly one edge with exactly one interior lattice point. Combined with the classification, $\BP$ is thus one of the following two polygons, pictured in Figure \ref{fig:two-hollow-cases}:

$$\conv\{(-1, 0), (0, 1), (1, 0)\}, \quad \conv\{(-1, 0), (-1, 1), (0, 1), (1, 0)\}.$$

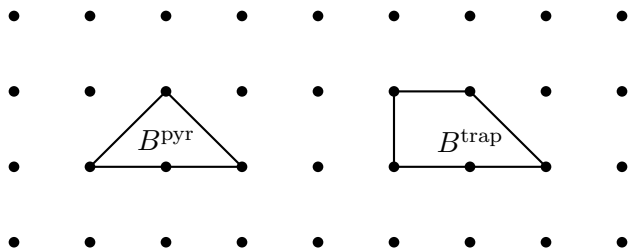
\begin{figure}[h]
\begin{tikzpicture}
\draw[thick] (-3, 0) -- (-1, 0) -- (-2, 1) -- cycle;
\draw[thick] (1, 0) -- (3, 0) -- (2, 1) -- (1, 1) -- cycle;
  \foreach \x in {-4, -3, -2, -1, ...,3, 4}{
    \foreach \y in {-1, 0,1, 2}{
      \fill (\x,\y) circle (2pt);
    }
  }
  \node at (2,.35) {$\Btrap$};
  \node at (-2,.35) {$\Bpyr$};
\end{tikzpicture}
\cprotect\caption{The two remaining polygons of Case~\ref{case:hollow} of Proposition \ref{lemma:disjoint-casework-on-blocking-polytopes}. We refer to them as $\Bpyr$ and $\Btrap$, respectively.}
\label{fig:two-hollow-cases}
\end{figure}

Now consider $\BP$ as in Case~\ref{case:reflexive}. 
There are 16 lattice polygons (up to unimodular transformation) with one interior lattice point, see \cite{Rabinowitz89}. 
By Lemma \ref{lemma:three-colinear-boundary-points}, we can assume that $\BP$ has no edge of lattice length greater than 1.  Filtering the list of the 16 so accordingly yields the following 5 polygons, seen in Figure \ref{fig:case-3-remaining}:

\begin{align*}
&\conv \{(-1, -1), (1, 0), (0, 1)\},~ \conv \{(-1, -1), (0, -1), (1, 1), (-1, 0)\}, \\
&\conv \{(-1, 0), (0, -1), (1, 0), (0, 1)\},~ \conv \{(-1, -1), (0, -1), (1, 0), (0, 1), (-1, 0) \},\\ &\conv \{(-1, -1), (0, -1), (1, 0), (1, 1), (0, 1), (-1, 0)\}.
\end{align*}

\begin{figure}[h] 
  \begin{tikzpicture}
  \draw[thick] (-4, 0) -- (-2, 1) -- (-3, 2) -- cycle;
  \draw[thick] (-1, 0) -- (0, 0) -- (1, 2) -- (-1, 1) -- cycle;
  \draw[thick] (2, 1) -- (3, 0) -- (4, 1) -- (3, 2) -- cycle;
  \draw[thick] (-4, -3) -- (-3, -3) -- (-2, -2) -- (-3, -1) -- (-4, -2) -- cycle;
  \draw[thick] (-1, -3) -- (0, -3) -- (1, -2) -- (1, -1) -- (0, -1) -- (-1, -2) -- cycle; 
    \foreach \x in {-4, -3, -2, -1, ...,3, 4}{
      \foreach \y in {-3, -2, -1, 0,1, 2}{
        \fill (\x,\y) circle (2pt);
      }
    }
    \node at (-2.8,1.25) {$\Bterm$};
    \node at (-.52, .6) {$\Bkite$};
    \node at (3,.65) {$\Bcross$};
    \node at (-3.5, -2.4) {$\Bpent$};
    \node at (-.52, -2.4) {$\Bhex$};
  \end{tikzpicture}
\cprotect\caption{The 5 remaining polygons of Case~\ref{case:reflexive} of Proposition \ref{lemma:disjoint-casework-on-blocking-polytopes}. We refer to them as $\Bterm, \Bkite, \Bcross, \Bpent,$ and $\Bhex$, respectively. Compare with Figure 6 of \cite{HaaseSchicho09}.}
\label{fig:case-3-remaining}
\end{figure}
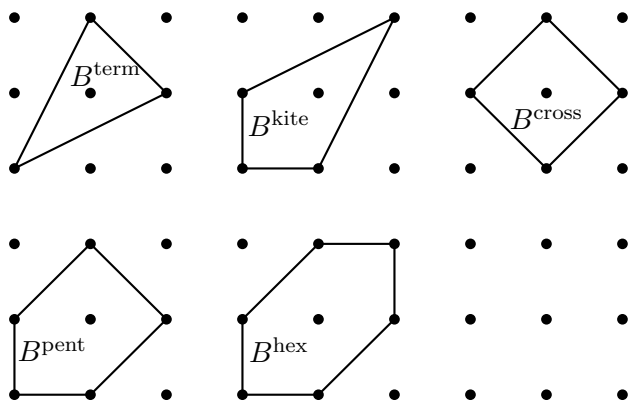

We have thus reduced the proof of Theorem \ref{thm:flt21} to proving the following proposition, which we prove case by case in the following subsections. 
\begin{proposition} \label{prop:remaining-casework}
    Let $P$ be a $1$-maximal polygon such that $\BP$ is one of the following
  \begin{mycases}
    \item $\BP$ is $\conv\{(-1,-1),(-1,0),(0,-1)\}$, and the interior lattice point of $P$ is $(0,0)$.
    \item $\BP$ is $\conv\{(-1, 0), (0, 1), (1, 0)\}$ or $\conv\{(-1, 0), (-1, 1), (0, 1), (1, 0)\}$, and the interior lattice point of $P$ is $(0, 0)$; see Figure \ref{fig:two-hollow-cases}.
    \item $\BP$ is one of the 5 polygons as in Figure \ref{fig:case-3-remaining}.
  \end{mycases}
  Then $\wdt(P) \leq 3$, and equality holds only for $P=3\Tst$. 
\end{proposition}

\subsection{Casework on the blocking polygons}

In the following we consider the finitely many blocking polygons $B$ that we identified in the previous section. We show that $\wdt(P) \leq 3$ for each of those blocking polygons with equality if and only if $P$ is equivalent to $3\Tst$, in which case $\BP$ is a hexagon.

Although the hexagon is the unique blocking polytope of $3\Tst$, it is important to note that for any blocking polytope $B$ of our list, $3\Tst$ can be approximated by 1-maximal polygons $P$ with $\BP = B$. Hence, the (strict) upper bound of 3 for the width of a 1-maximal polygon with $\BP=P$ cannot be improved in any of the upcoming cases. 

Before we begin with the case distinction, we want to give an overview of the similarities and differences of the cases. First, by a nice choice of unimodular embedding of our blocking polytopes, it will turn out that in each cases it actually suffices to consider (a subset of) the set of directions
\[
\mathcal A = \{\pm \ve_1,~\pm\ve_2,~\pm(\ve_2 -\ve_1)\}
\]
in order to measure the lattice width. Indeed, one generally has $\wdt(P)\leq \wdt(P;\mathcal A)$ and due to the positioning of the blocking polytopes, it turns out that this estimate is sharp.

In the first cases (hexagon, pentagon, long triangle and crosspolygon) it is possible to bound $\wdt(\mathcal A)$ by deriving certain linear inequalities. Next, in case of the kite, the trapezoid, and the standard triangle, to contend with the piecewise polynomial nature of the optimization problem, we use a variational argument and split our domain into lienar regions. The most exceptional case seemed to be the terminal triangle for which we furthermore employ Lagrange multipliers and the computation of elimination ideals.

\subsubsection{The hexagon} \label{sec:hexagon}
We consider polygons $P$ such that 
\hypertarget{target:Bhex}{%
$$\BP = B^{\mathrm{hex}} = \conv \{(-1, -1), (0, -1), (1, 0), (1, 1), (0, 1), (-1, 0)\}.$$
}
\begin{lemma} \label{lem:hexagon-shards}
  Let $P$ be a $1$-maximal polygon with $\BP = \Bhex$. Then each vertex of $P$ lies in one of the following regions, and each region contains (at most) one vertex:

  \begin{align*}
    S_1 &= \conv \{(-1, -1), (-1, -2), (0, -1)\}, \\
    S_2 &=\conv \{(0, -1), (1, -1), (1, 0)  \},  \\
    S_3 &=\conv \{(1, 0), (2, 1), (1, 1) \},  \\
    S_4 &=\conv \{(1, 1), (1, 2), (0, 1)  \},  \\
    S_5 &=\conv \{(0, 1), (-1, 1), (-1, 0)\},  \\
    S_6 &=\conv \{(-1, 0), (-2, -1), (-1, -1) \}.
  \end{align*}
\end{lemma}

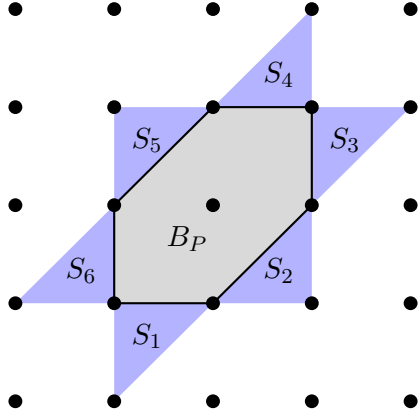
\begin{figure}[h]
\begin{tikzpicture}[scale=1.3]
    \fill[blue!30] (-1, -1) -- (-1, -2) -- (0, -1) -- cycle;
    \fill[blue!30] (0, -1) -- (1, -1) -- (1, 0) --cycle; 
    \fill[blue!30] (1, 0) -- (2, 1) -- (1, 1) --cycle;
    \fill[blue!30] (1, 1) -- (1, 2) -- (0, 1)  -- cycle; 
    \fill[blue!30] (0, 1) -- (-1, 1) -- (-1, 0) -- cycle;
    \fill[blue!30] (-1, 0) -- (-2, -1) -- (-1, -1) -- cycle;
 	\fill[gray!30] (-1, -1) -- (0, -1) -- (1, 0) -- (1, 1) -- (0, 1) -- (-1, 0) -- cycle;

  \draw[thick] (-1, -1) -- (0, -1) -- (1, 0) -- (1, 1) -- (0, 1) -- (-1, 0) -- cycle;

    \foreach \x in {-2, -1, ...,2}{
    \foreach \y in {-2,-1, ...,2}{
      \fill (\x,\y) circle (2pt);
    }
  }
  \node at (-.26, -.34) {$\BP$};
  \node at (-.6666, -1.3333) {$S_1$};
  \node at (.6666, -.6666) {$S_2$};
  \node at (1.3333, .6666) {$S_3$};
  \node at (.66666, 1.33333) {$S_4$};
  \node at (-.66666, .6666) {$S_5$};
  \node at (-1.333, -.6666) {$S_6$};
  \end{tikzpicture}
  \caption{The configuration as in Lemma \ref{lem:hexagon-shards}. The regions $S_1, ..., S_6$ are as labeled, in blue.}

\end{figure}

\begin{proof}
The regions $S_i$ are the $S_e$ from Lemma 	\ref{lemma:initial_shards}.
\end{proof}

\begin{lemma} \label{lem:hexagon-width}
  Let $P$ be a $1$-maximal polygon such that $\BP = \Bhex$. Then $\wdt(P) \leq 3$, with equality if and only if $P$ is
unimodularly equivalent to $3 \Tst$.
\end{lemma}

\begin{proof}
  For $1 \leq i \leq 6$, let $\vec p_i$ be (possibly degenerate, see Lemma \ref{lemma:initial_shards}) vertices of $P$ in the regions $S_i$ of Lemma \ref{lem:hexagon-shards}, respectively. Then the width of $P$ in directions $\ve_1, \ve_2$ is, respectively, 
  \begin{align*}
    \wdt(P, \ve_1)  = p_{3,x} - p_{6,x}, \quad \wdt(P, \ve_2)  = p_{4,y} - p_{1,y}. 
  \end{align*}


  Observe that the points $\vec p_1, \vec p_6$ must lie between the lines $\{y-x = -1\}$ and $\{y-x = 1\}$, and that $\vec p_1, \vec p_6$ and $(-1, -1)$ are colinear: that is, there exists $\lambda \geq 0$ such that $\vec p_1, \vec p_6$ lie on the line $\{y+1=-\lambda (x+1)\}$. We then see that 
  $p_{1,y} + p_{6,x}$ is minimized if $\vp_1$, $\vp_6$ lie respectively on lines $\{y-x=-1 \}$, $\{ y-x = 1\}$, in which case we have 
  $$p_{1,y} + p_{6,x} = -\frac{2+\lambda}{1+\lambda} + -\frac{2\lambda+1}{1+\lambda} = -3.$$
  
  Similarly, we obtain
  $$p_{3,x} + p_{4,y}  \leq 3,$$ 
  with equality if $\vp_3, \vp_4$ lie respectively on lines $\{y-x=-1\},~ \{y-x = 1\}$. 

  Altogether, we have
  \begin{align*}
  \wdt(P, \ve_1) + \wdt(P, \ve_2)  = p_{3,x} + p_{4,y} - p_{6,x} - p_{1,y} \leq 6.
  \end{align*}
  Thus $\min \{\wdt(P, \ve_1), \wdt(P, \ve_2)\}\leq 3$. 

  Suppose now that $\wdt(P) = 3$, thus in particular both $\wdt(P, \ve_1), \wdt(P, \ve_2)$ are $3$. Then it necessarily holds that $\vec p_1, \vec p_4\in\{y = x-1\}$ and $\vec p_4, \vec p_6\in \{y = x+1\}$. 

  Examining the width in direction $\ve_1 - \ve_2$, we see it is $2$ except when $\vec p_1, \vec p_3$ are not vertices but lie in the interior of edges of $P$, or when the same holds for $\vp_4$ and $\vp_6$. The first scenario can happen only if $\vp_2=(1,-1)$ and thus $\vp_4 =(1,2)$ and $\vp_6=(-2,-1)$, while the second forces $\vp_5=(-1, 1)$, and thus $\vp_1= (-1,-2), \vp_3=(2,1)$, and in both cases we have that $P\cong3\Tst$.
\end{proof}

\subsubsection{The pentagon} \label{sec:pentagon}
We consider $P$ such that 
\hypertarget{target:Bpent}{%
$$\BP = B^{\mathrm{pent}} = \conv \{ (-1, 1), \pm \ve_1, \pm \ve_2 \}.$$
}

\begin{lemma} \label{lem:pentagon-shards}
  Let $P$ be $1$-maximal with $\BP= \Bpent$. Then each vertex of $P$ lies in one of the following regions, and each region contains (at most) one vertex:
  \begin{align*}
    S_1 &= \conv \{(-1, -1), (-1, -2), (0, -1)\},  \\
    S_2 &= \conv \{(0, -1), (2, -1), (1, 0)\},  \\
    S_3 &= \conv \{(1, 0), (2,1), (0, 1)\} \cup \conv \{(1, 0), (1, 2), (0, 1)\},  \\
    S_4 &= \conv \{(0, 1), (-1, 2), (-1, 0)\},  \\
    S_5 &= \conv \{(-1, 0), (-2, -1), (-1, -1)\}. \\
  \end{align*}
\end{lemma}
\begin{figure}[h]
  \begin{tikzpicture}[scale=1.3]
  \fill[blue!30] (-2.2, -2.2) -- (-2.2, 2.2) -- (2.2, 2.2) -- (2.2, -2.2) -- cycle;

  \fill[purple!20] (-1, 0) -- (-1, 2.2) -- (-2.2, 2.2) -- (-2.2, -1.2) -- cycle;
  \fill[purple!20] (-1, -1) -- (-2.2, -1) -- (-2.2, -2.2) -- (-1, -2.2) -- cycle;
  \fill[purple!20] (0, -1) -- (-1.2, -2.2) -- (2.2, -2.2) -- (2.2, -1) -- cycle;

  \fill[purple!20] (1, 0) -- (2.2, -1.2) -- (2.2, 1.2) -- cycle;
  \fill[purple!20] (1,1) -- (2.2, 1) -- (2.2, 2.2) -- (1, 2.2) -- cycle;
  \fill[purple!20] (0, 1) -- (1.2, 2.2) -- (-1.2, 2.2) -- cycle;

  \draw[purple!50, dashed] (-1, 0) -- (-1, 2.2);
  \draw[purple!50, dashed] (-1, 0) -- (-2.2, -1.2);
  \draw[purple!50, dashed] (-1, -1) -- (-2.2, -1);
  \draw[purple!50, dashed] (-1, -1) -- (-1, -2.2);
  \draw[purple!50, dashed] (0, -1) -- (-1.2, -2.2);
  \draw[purple!50, dashed] (0, -1) -- (2.2, -1);
  \draw[purple!50, dashed] (1, 0) -- (2.2, -1.2);
  \draw[purple!50, dashed] (1, 0) -- (2.2, 1.2);
  \draw[purple!50, dashed] (1, 1) -- (2.2, 1);
  \draw[purple!50, dashed] (1, 1) -- (1, 2.2);
  \draw[purple!50, dashed] (0, 1) -- (1.2, 2.2);
  \draw[purple!50, dashed] (0, 1) -- (-1.2, 2.2);

  \fill[gray!40] (-1, 0) -- (-1, -1) -- (0, -1) -- (1, 0) -- (0, 1) -- cycle;

  \draw[thick] (-1, 0) -- (-1, -1) -- (0, -1) -- (1, 0) -- (0, 1) -- cycle;

    \foreach \x in {-2, -2, -1, ...,2}{
    \foreach \y in {-2, -2,-1, ...,2}{
      \fill (\x,\y) circle (2pt);
    }
  }
  \node at (-.26, -.34) {$\BP$};
  \node at (-.666, -1.333) {$S_1$};
  \node at (1, -.666) {$S_2$}; 
  \node at (.666, .666) {$S_3$};
  \node at (-.666, 1) {$S_4$};
  \node at (-1.3333, -.6666) {$S_5$};
  \end{tikzpicture}
\caption{The regions as in Lemma \ref{lem:pentagon-shards}. The regions $S_1, ..., S_5$ are as labeled, in blue. }
\end{figure}
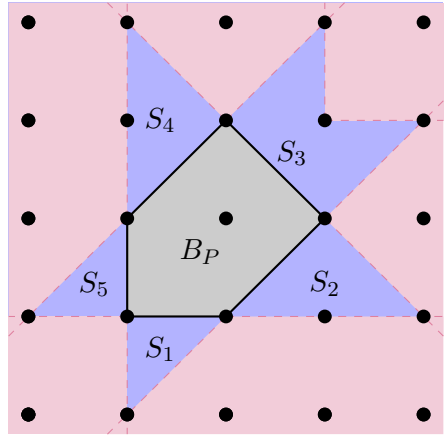

\begin{proof}
	We apply Lemma \ref{lemma:initial_shards}, and further cut out $\sigma_{(1,1)}$ to obtain $S_3$. 
\end{proof}
 
\begin{lemma}\label{lem:pent}
Let $P$ be a $1$-maximal polygon such that $\BP = \Bpent$. Then $\wdt(P) < 3$.
\end{lemma}

\begin{proof}
Let $\vp_i$, $1\leq i \leq 5$, be (possibly degenerate) vertices of $P$ in the regions $S_i$ of Lemma \ref{lem:pentagon-shards}, respectively. 
Up to reflection of $P$ over $\{ y = x \}$, we can assume that $\vp _3$ lies in $\{ x \leq 1 \}$.
The colinearity of $\vp _3, (1, 0), \vp _2$ then forces $\vp_2 \in S'_2 = \conv \{(1, -1), (2, -1), (1, 0)\}$.
We now consider separate cases depending on $p_{3,y}$.

\textbf{Case 1: } $p_{3,y} \leq 1$.
Here the colinearity $\vp _3, (0, 1), \vp _4$ forces $\vp_4 \in S'_4 = \conv \{ (0, 1), (-1, 2), (-1, 1) \}$, as visualized in Figure \ref{fig:pentagon-shards-case-1}.

\begin{figure}[h]
  \begin{tikzpicture}[scale=1.3]
  \fill[blue!30] (-2, -1) -- (-1, -1) -- (-1, 0) -- cycle;

  \fill[blue!30] (-1, -1) -- (-1, -2) -- (0, -1) -- cycle;
  \fill[blue!30] (1, 0) -- (1, -1) -- (2, -1) -- cycle;
  \fill[blue!30] (1, 0) -- (1, 1) -- (0, 1) -- cycle;
  \fill[blue!30] (0, 1) -- (-1, 2) -- (-1, 1) -- cycle;

  \fill[gray!40] (-1, 0) -- (-1, -1) -- (0, -1) -- (1, 0) -- (0, 1) -- cycle;

  \draw[thick] (-1, 0) -- (-1, -1) -- (0, -1) -- (1, 0) -- (0, 1) -- cycle;

    \foreach \x in {-2, -1, ...,2}{
    \foreach \y in {-2,-1, ...,2}{
      \fill (\x,\y) circle (2pt);
    }
  }
  \node at (-.666, -1.333) {$S_1$};
  \node at (1.333, -.666) {$S'_2$}; 
  \node at (.666, .666) {$S'_3$};
  \node at (-.666, 1.333) {$S'_4$};
  \node at (-1.3333, -.6666) {$S_5$};
  \node at (-.26, -.34) {$\BP$};
  \end{tikzpicture}
  \hspace{2cm}
  \begin{tikzpicture}[scale=1.7]

  \fill[blue!30] (-2, -1) -- (-1, -1) -- (-1, 0) -- cycle;

  \fill[blue!30] (0, 1) -- (-1, 2) -- (-1, 1) -- cycle;

  \fill[gray!40] (-1, 0) -- (0,1) -- (.333, .666) -- (.333, -.666) -- (0, -1) -- (-1, -1) -- cycle;
  \draw[thick] (-1, 0) -- (0,1) -- (.333, .666);
  \draw[thick] (.333, -.666) -- (0, -1) -- (-1, -1) -- (-1,0);

  \draw[thick] (-1.3333333, -.6666666) -- (-.4, -.666666+1.8666666);
  \draw[thick, dashed] (-1.333333,-.6666666) -- (-1.5, -1);
  \fill (-1.3333333, -.6666666) circle (2pt);
  \fill (-.4, -.66666666+1.8666666666) circle (2pt); 

  \fill[red] (-1.5, -1) circle (2pt);
  \fill[red] (-.5, 1) circle (2pt);

  \foreach \x in {-2,-1,0}{
    \foreach \y in {-1,0, ...,2}{
      \fill (\x,\y) circle (2pt);
    }
  }

  \node at (-.4, .9) {$\vec A$}; 
  \node at (-1.6333, -.8666) {$\vec B$};
  \node at (-1.15, -.66666) {$\vp_5$}; 
  \node at (-.2, 1.3) {$\vp_4$};
  \node at (-.26, -.34) {$\BP$};
  \end{tikzpicture}
\caption{The regions as in {Case 1} of Lemma \ref{lem:pent}. On the right we visualize the points $\vec A, \vec B$ which help relate $p_{4,y}$ and $p_{5,x}$.}
\label{fig:pentagon-shards-case-1}
\end{figure}
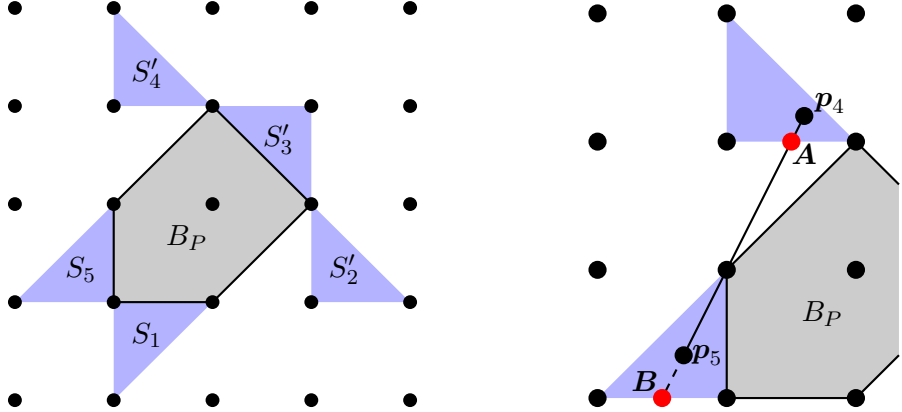

We consider now the width of $P$ in directions $\ve_1, \ve_2$. We have
\begin{align*}
  \wdt(P, \ve_1)  = p_{2,x} - p_{5,x}, \quad \wdt(P, \ve_2)  = p_{4,y} - p_{1,y}.
\end{align*}

Let $L$ denote the line given by $\vp _5, (-1, 0), \vp_4$, which are colinear by construction, and let $\vec A = L \cap \{y = 1\}$, $\vec B = L \cap \{y = -1\}$, as in Figure \ref{fig:pentagon-shards-case-1}. We have
$$A_x = -2 - B_x \geq -2 - p_{5,x}.$$

Since $\vp_4$ lies below the line $\{ y=-x+1 \}$, we have $p_{4,y} \leq -p_{4,x} +1 \leq -  A _x + 1$, hence
$$p_{4,y} -p_{5,x} \leq  3.$$

A symmetric argument yields
$$p_{2,x} - p_{1,y} \leq 3.$$

Hence the sum of the widths is bounded 
$\wdt(P,\ve_1) + \wdt(P,\ve_2) \leq 6, $
and in particular one must be $\leq 3$. Let us examine when $\wdt(P)=3$ can hold. For this to be the case, equality has to hold in all the inequalities above, which forces $\vp_5$ to lie on the horizontal edge of $S_5$ and $\vp_4$ on the diagonal of $S'_4$, and symmetrically, $\vp_1$ lies on the vertical edge of $S_1$ and $\vp_2$ on the diagonal of $S'_2$. The only way to have the required colinearity of $\vp_1, (-1,-1)$ and $\vp_5$ in this situation is if $\vp_1, \vp_5$ are both equal to $(-1, -1)$. This forces $P$ to be $2\Delta_3$, contradicting $\BP=B^{\text{pent}}$. 

\textbf{Case 2: } $p_{3,y} \geq 1$.

In this case, we restrict $S_2, S_4$ to obtain the regions $S'_2, S'_4$ as in Figure \ref{fig:pentagon-shards-case-2}. 

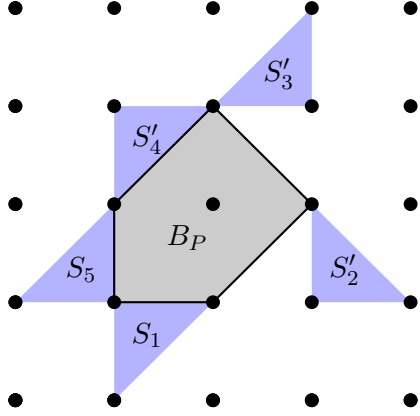
\begin{figure}[h]
  \begin{tikzpicture}[scale=1.3]
  \fill[blue!30] (-2, -1) -- (-1, -1) -- (-1, 0) -- cycle;

  \fill[blue!30] (-1, -1) -- (-1, -2) -- (0, -1) -- cycle;
  \fill[blue!30] (1, 0) -- (1, -1) -- (2, -1) -- cycle;
  \fill[blue!30] (1, 2) -- (1, 1) -- (0, 1) -- cycle;
  \fill[blue!30] (0, 1) -- (-1, 0) -- (-1, 1) -- cycle;

  \fill[gray!40] (-1, 0) -- (-1, -1) -- (0, -1) -- (1, 0) -- (0, 1) -- cycle;

  \draw[thick] (-1, 0) -- (-1, -1) -- (0, -1) -- (1, 0) -- (0, 1) -- cycle;

  \foreach \x in {-2, -2, -1, ...,2}{
    \foreach \y in {-2, -2,-1, ...,2}{
      \fill (\x,\y) circle (2pt);
    }
  }
  \node at (-.26, -.34) {$\BP$};
  \node at (-.666, -1.333) {$S_1$};
  \node at (1.333, -.666) {$S'_2$}; 
  \node at (.666, 1.333) {$S'_3$};
  \node at (-.666, .666) {$S'_4$};
  \node at (-1.3333, -.6666) {$S_5$};
  \end{tikzpicture}
\caption{The regions as in {Case 2} of Lemma \ref{lem:pent}.}
 \label{fig:pentagon-shards-case-2}
\end{figure}

Arguing analogously to case 1, colinearity of $\vp_1, (-1, -1)$ and $\vp_5$ yields
$$p_{1,y} + p_{5,x} \geq -3,$$

while colinearity $\vp_2, (1,0)$ and $\vp_3$ yields
$$p_{2,x} + p_{3,y} \leq 3.$$

Putting these together we obtain
\begin{align*}
  \wdt(P, \ve_1) + \wdt(P,\ve_2)  = p_{2,x} - p_{5,x} + p_{3,y} - p_{1,y}  \leq 6.
\end{align*}

Just as in case 1, equality $\wdt(P) = 3$ can occur only if $\vp_3, \vp_5$ lie on $\{y = x + 1\}$, $\vp_2$ on $\{ y = 2 \}$, and $\vp_1$ on $\{y=x-1\}$. The only way this can occur is if $\vp_2=(1,-1)$, which forces $P=3\Tst,$ contradicting the assumption $\BP=B^{\text{pent}}$.
\end{proof}

\subsubsection{The crosspolygon}
\label{subsub:sym_quad}
We consider $P$ such that 
\hypertarget{target:Bcross}{%
$$\BP = B^{\mathrm{cross}} = \conv \{ \pm \ve_1, \pm \ve_2 \}.$$
}

\begin{lemma} \label{lem:symm-quad-shards}
  Let $P$ be $1$-maximal such that $\BP = \Bcross$. Then each vertex of $P$ lies in one of the following regions, and each region contains (at most) one vertex:

  \begin{align*}
    S_1 &= \conv\{(1, 0), (2, 1), (0, 1)\} \cup \conv \{(1, 0), (1, 2), (0, 1)\}, \\
    S_2 &= \conv \{(0,1),(-1,2),(-1,0)\} \cup \conv \{(0,1),(-2,1),(-1,0)\}, \\
    S_3 &= \conv \{(-1,0),(-2,-1),(0,-1)\} \cup \conv \{(-1,0),(-1,-2),(0,-1)\}, \\
    S_4 &= \conv \{(0,-1),(1,-2),(1,0)\} \cup \conv \{(0,-1),(2,-1),(1,0)\}.
  \end{align*}

\end{lemma}
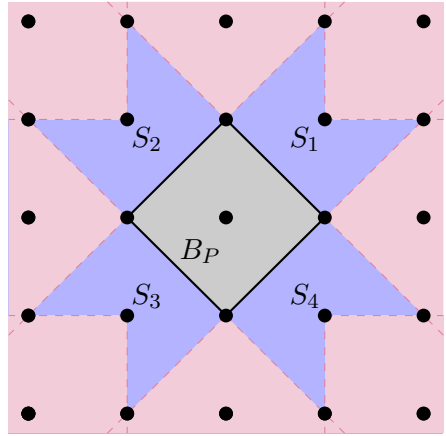
\begin{figure}[h]
  \begin{tikzpicture}[scale=1.3]
  \fill[blue!30] (-2.2, -2.2) -- (-2.2, 2.2) -- (2.2, 2.2) -- (2.2, -2.2) -- cycle;

  \fill[purple!20] (-1, 0) -- (-2.2, 1.2) -- (-2.2, -1.2) -- cycle;
  \fill[purple!20] (-1, -1) -- (-2.2, -1) -- (-2.2, -2.2) -- (-1, -2.2) -- cycle;
  \fill[purple!20] (0, -1) -- (-1.2, -2.2) -- (1.2, -2.2) -- cycle;
  \fill[purple!20] (1, -1) -- (1, -2.2) -- (2.2, -2.2) -- (2.2, -1) -- cycle;
  \fill[purple!20] (1, 0) -- (2.2, -1.2) -- (2.2, 1.2) -- cycle;
  \fill[purple!20] (1,1) -- (2.2, 1) -- (2.2, 2.2) -- (1, 2.2) -- cycle;
  \fill[purple!20] (0, 1) -- (1.2, 2.2) -- (-1.2, 2.2) -- cycle;
  \fill[purple!20] (-1, 1) -- (-1, 2.2) -- (-2.2, 2.2) -- (-2.2, 1) -- cycle;

  \draw[purple!50, dashed] (-1, 0) -- (-2.2, 1.2);
  \draw[purple!50, dashed] (-1, 0) -- (-2.2, -1.2);
  \draw[purple!50, dashed] (-1, -1) -- (-2.2, -1);
  \draw[purple!50, dashed] (-1, -1) -- (-1, -2.2);
  \draw[purple!50, dashed] (0, -1) -- (-1.2, -2.2);
  \draw[purple!50, dashed] (0, -1) -- (1.2, -2.2);
  \draw[purple!50, dashed] (1, -1) -- (1, -2.2);
  \draw[purple!50, dashed] (1, -1) -- (2.2, -1);
  \draw[purple!50, dashed] (1, 0) -- (2.2, -1.2);
  \draw[purple!50, dashed] (1, 0) -- (2.2, 1.2);
  \draw[purple!50, dashed] (1, 1) -- (2.2, 1);
  \draw[purple!50, dashed] (1, 1) -- (1, 2.2);
  \draw[purple!50, dashed] (0, 1) -- (1.2, 2.2);
  \draw[purple!50, dashed] (0, 1) -- (-1.2, 2.2);
  \draw[purple!50, dashed] (-1, 1) -- (-1, 2.2);
  \draw[purple!50, dashed] (-1, 1) -- (-2.2, 1);

  \fill[gray!40] (-1, 0) -- (0, -1) -- (1, 0) -- (0, 1) -- cycle;

  \draw[thick] (-1, 0) -- (0, -1) -- (1, 0) -- (0, 1) -- cycle;

    \foreach \x in {-2, -2, -1, ...,2}{
    \foreach \y in {-2, -2,-1, ...,2}{
      \fill (\x,\y) circle (2pt);
    }
  }

  \node at (-.26, -.34) {$\BP$};
  \node at (.8, .8) {$S_1$};  
  \node at (-.8, .8) {$S_2$};
  \node at (-.8, -.8) {$S_3$};
  \node at (.8, -.8) {$S_4$};
  \end{tikzpicture}
\caption{The regions as in Lemma \ref{lem:symm-quad-shards}.}
\label{fig:cross-polytope-shards}
\end{figure}
\begin{proof}
  These regions arise by combining Lemma \ref{lemma:initial_shards} and Lemma \ref{lemma:forbidden_cones} applied to the points $(1,1), (-1,1), (1,-1)$, and $(-1,-1)$.
\end{proof}

\begin{lemma}
  Let $P$ be a $1$-maximal polygon such that $\BP=\Bcross$. 

  Then $\wdt(P) < 3$. 
\end{lemma}

\begin{proof}

  Let $\vp_1, \vp_2, \vp_3, \vp_4$ be (possibly degenerate) vertices of $P$ in $S_1, S_2, S_3, S_4$, respectively. 
  We first handle some degenerate cases. 

\begin{claim*} If any $\vp_i$ is on the boundary of the square $[-1, 1]^2$, then $\wdt(P) < 3$. \end{claim*}

 Due to the symmetry of the cross-polytope, we can assume without loss of generality that $p_{1,y} = 1$. Then colinearity of $\vp_1, (0, 1), \vp_2$ yields that $\vp_2$ also lies on $\{ y = 1 \}$. Combined with the regions as in Lemma \ref{lem:symm-quad-shards}, we see that $\wdt(P,\ve_2) \leq 3$. Equality occurs if $\vp_3 = (-1, -2)$ or $\vp_4 = (1, -2)$, but then $P$ is a subpolygon of $3 \Tst$; reducedness of $3 \Tst$ yields that $\wdt(P) = 3$ if and only if $P = 3 \Tst$, contradicting $\BP= \Bcross$.  

  Now suppose that one of the $\vp_i$ lies in the open square $(-1, 1)^2$, without loss of generality $\vp_1$. Then the the colinearities $\vp_1, (0,1), \vp_2$ and $\vp_1, (1, 0), \vp_4$ yield that $p_{2,y} \geq 1$ and $p_{4,x} \geq 1$, respectively. But then the colinearity of $\vp_2, (-1,0), \vp_3$ forces $\vp_3$ to lie in $\conv\{(-1,0), (-1,-1), (-2,-1)\}$, the left triangle of $S_3$, while the colinearity of $\vp_3, (0,-1), \vp_4$ forces $p_3$ to be in the lower triangle of $S_3$. The only way both of these conditions can hold is if $\vp_3 = (-1, -1)$, so we are in the situation treated by the claim above. 

  We can therefore assume that none of the $\vp_i$ lie in the square $[-1,1]^2$.
  Up to reflection over $\{ y = x \}$, we can assume that $\vp_1 \in \conv \{(1, 1), (1, 2), (0, 1)\} \setminus [(0, 1), (1, 1)]$. Considering the colinearities yields the regions $S_1',\dots, S_4'$ in Figure \ref{fig:cross-polytope-final-shards} as possible locations for the vertices.

  \begin{figure}[h]
  \begin{tikzpicture}[scale=1.3]
  \fill[blue!30] (1, 1) -- (1, 2) -- (0, 1) -- cycle;
  \fill[blue!30] (-1, 1) -- (-2, 1) -- (-1, 0) -- cycle;
  \fill[blue!30] (-1, -1) -- (0, -1) -- (-1, -2) -- cycle;
  \fill[blue!30] (1, 0) -- (2, -1) -- (1, -1) -- cycle;

  \fill[gray!40] (-1, 0) -- (0, -1) -- (1, 0) -- (0, 1) -- cycle;

  \draw[thick] (-1, 0) -- (0, -1) -- (1, 0) -- (0, 1) -- cycle;

    \foreach \x in {-2, -2, -1, ...,2}{
    \foreach \y in {-2, -2,-1, ...,2}{
      \fill (\x,\y) circle (2pt);
    }
  }

  \node at (-.26, -.34) {$\BP$};
  \node at (.6, 1.2) {$S_1'$};  
  \node at (-1.4, .7) {$S_2'$};
  \node at (-.7, -1.3) {$S_3'$};
  \node at (1.3, -.7) {$S_4'$};
  \end{tikzpicture}
  \hspace {1cm}
    \begin{tikzpicture}[scale=1.7]
  \fill[blue!30] (1, 1) -- (1, 2) -- (0, 1) -- cycle;
  \fill[blue!30] (1, 0) -- (2, -1) -- (1, -1) -- cycle;

  \fill[gray!40] (-.33333, -.66666) -- (0, -1) -- (1, 0) -- (0, 1) -- (-.3333333,.666666) -- cycle;

  \draw[thick] (-.333333, -.66666) -- (0, -1) -- (1, 0) -- (0, 1) -- (-.333333,.666666);
  \fill (1.333, -.6666) circle (2pt);
  \fill (.4, -.6666+1.86666) circle (2pt);

  \draw[thick] (1.333, -.6666) -- (.4, -.6666+1.86666);
  \draw[thick, dashed] (1.3333,-.6666) -- (1.5, -1); 

  \fill[red] (.5, 1) circle (2pt);
  \fill[red] (1.5, -1) circle (2pt); 

    \foreach \x in {0, 1 ,2}{
    \foreach \y in {-1, ...,2}{
      \fill (\x,\y) circle (2pt);
    }
  }
  
  \node at (-.15, -.43) {$\BP$};
  \node at (1.63, -.86) {$\vec B$};
  \node at (.4, .85) {$\vec A$};
  \node at (.26, 1.35) {$\vp_1$}; 
  \node at (1.08, -.6666) {$\vp_4$};
  \end{tikzpicture}
  \caption{The regions that the vertices of $P$ can lie in. On the right we visualize the points $\vec A, \vec B$.}
  \label{fig:cross-polytope-final-shards}
  \end{figure}
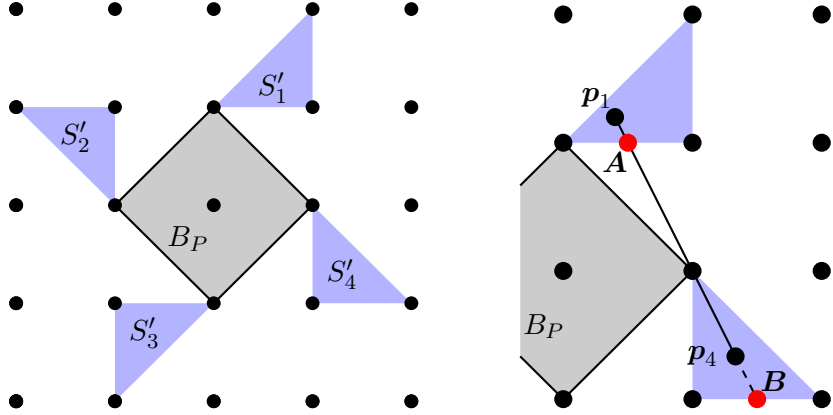

  Thus, we can express $\wdt(P,  \ve_1), \wdt(P, \ve_2)$ as 
  \begin{align*}
    \wdt(P, \ve_1) = p_{4,x} - p_{2,x}, \quad \wdt(P, \ve_2) = p_{1,y} - p_{3,y}. 
  \end{align*}
  We use a similar technique as in Section \ref{sec:pentagon} to show $\wdt(P) \leq \min\{\wdt(P, \ve_1), \wdt(P, \ve_2)\} < 3$.

  Let $L$ be the line through the colinear points $\vp_4, (1,0), \vp_1$, and let $\vec A$ be the intersection point $L \cap \{y=1\}$ and $\vec B$ be the intersection point $L \cap \{y=-1\}$; see Figure \ref{fig:cross-polytope-final-shards}, on the right.

  By symmetry of $\vec A, \vec B$ around $(1,0)$, we have that 
  $$A_x = 2 - B_x \leq 2- p_{4,x}.$$

  Since $\vp_1 \in \{ y \leq x + 1 \}$, it holds
  that 
  \begin{equation}\label{eq:cross_pointA}
  p_{1,y} \leq 1 + p_{1,x} \leq 1 + A_x.
  \end{equation}

  Combining the inequalities, we have
  $p_{1,y} +  p_{4,x} \leq 3.$

  A symmetric argument yields  
  $p_{2,x} + p_{3,y} \geq -3.$
  Together these inequalities give
  \begin{align*} 
    \wdt(P, \ve_1) + \wdt(P, \ve_2) = p_{4,x} + p_{1,y} - p_{2,x} - p_{3,y} \leq 6.
  \end{align*}
  Hence  $ \min\{\wdt(P, \ve_1), \wdt(P, \ve_2)\}\leq 3$. 

  For equality to hold, all inequalities above should be equalities. In particular, both inequalities of \eqref{eq:cross_pointA} must be equalities, which implies $\vp_1 = (0,1)$. Then, by the claim, equality occurs if and only if $P$ is unimodularly equivalent to $3 \Tst$, which contradicts our assumption that $\BP=B^{\text{cross}}$.
\end{proof}

\subsubsection{The long triangle} \label{sec:long-triangle}
In this section, we consider polygons $P$ such that 
\hypertarget{target:Bpyr}{%
	$$\BP = B^{\mathrm{pyr}} = \conv \{ (-1, 0), (1, 0), (0, 1) \} $$
}
with $\vec 0$ as the interior point of $P$.

\begin{lemma}
	Let $P$ be a $1$-maximal polygon such that $\BP = \Bpyr$, and $\vec 0$ is the interior point of $P$. Then $\wdt(P) < 3$.
\end{lemma}
\begin{proof}
	Lemma \ref{lemma:initial_shards} and Lemma \ref{lemma:forbidden_cones} applied to the points $q=(1,1), (-1,1)$, as well as symmetry and colinearity conditions allow us to assume that the vertices $\vp_1, \vp_2$ and $\vp_3$ of $P$ lie respectively in the regions (pictured in Figure \ref{fig:long-triangle-shards})
	\begin{align*}
	S_1&=\conv \{(0, 1), (1, 1), (1, 2)\},\\
	S_2&=\{-1 \leq y \leq 1,~ y \leq \tfrac{-1}{2}x - \tfrac{1}{2}\}, \\
	S_3&=\conv \{(1, 0), (3, -2), (1, -1)\}.
	\end{align*}

	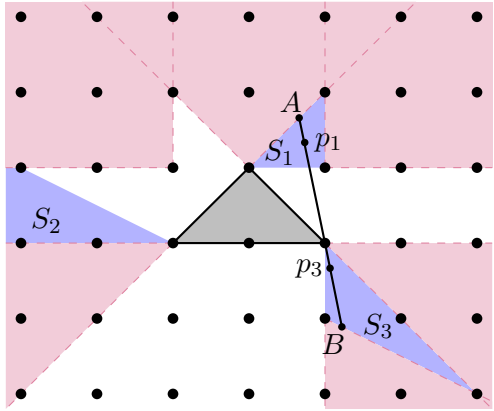
\begin{figure}[h]
		\begin{tikzpicture}
		
		\fill[blue!30] (0, 1) -- (1, 1) -- (1, 2) -- cycle;
		\fill[blue!30] (1, 0) -- (3, -2) -- (1, -1) -- cycle;
		\fill[blue!30] (-3, 1) -- (-3.2, 1) -- (-3.2, 0) -- (-1, 0) -- cycle;
		
		\fill[gray!50] (-1, 0) -- (1, 0) -- (0, 1) -- cycle;
		\fill[purple!20] (0, 1) -- (2.2, 3.2) -- (-2.2, 3.2) -- cycle;
		\fill[purple!20] (1,1) -- (3.2, 1) -- (3.2, 3.2) -- (1, 3.2) -- cycle;
		\fill[purple!20] (1, 0) -- (3.2, 0) -- (3.2, -2.2) -- cycle;
		\fill[purple!20] (1, -1) -- (1, -2.2) -- (3.2, -2.2) -- (3.2, -2.1) -- cycle;
		\fill[purple!20] (-1, 0) -- (-3.2, 0) -- (-3.2, -2.2) -- cycle;
		\fill[purple!20] (-1, 1) -- (-1, 3.2) -- (-3.2, 3.2) -- (-3.2, 1) -- cycle;
		
		\draw[purple!50, dashed] (0, 1) -- (2.2, 3.2); 
		\draw[purple!50, dashed] (0, 1) -- (-2.2, 3.2);

		\draw[purple!50, dashed] (1,1) -- (3.2, 1); 
		\draw [purple!50, dashed] (1,1) -- (1, 3.2); 
		
		\draw[purple!50, dashed] (1, 0) -- (3.2, 0);
		\draw[purple!50, dashed] (1, 0) -- (3.2, -2.2);
		
		\draw[purple!50, dashed] (1, -1) -- (3.2, -2.1);
		\draw[purple!50, dashed] (1, -1) -- (1, -2.2);
		
		\draw[purple!50, dashed] (-1, 0) -- (-3.2, 0);
		\draw[purple!50, dashed] (-1, 0) -- (-3.2, -2.2);
		\draw[purple!50, dashed] (-1, 1) -- (-1, 3.2);
		\draw[purple!50, dashed] (-1, 1) -- (-3.2, 1);
		
		\draw[thick] (-1, 0) -- (1, 0) -- (0, 1) -- cycle;
		\draw[thick] (.66, 1.66) -- (1, 0) -- (11/9,-10/9);
		
		\fill (.66, 1.66) circle (1.4pt);
		\fill (11/15, 4/3) circle (1.4pt);
		\fill (16/15,-1/3) circle (1.4pt);
		\fill (11/9,-10/9) circle (1.4pt);
		
		\node at (.4, 1.2) {$S_1$};
		\node at (1.7, -1.1) {$S_3$};
		\node at (-2.666, .3) {$S_2$};
		\node at (.55, 1.87) {$A$};
		\node at (10/9, -12/9) {$B$};
		\node at (1.05, 4/3) {$p_1$};
		\node at (.8,-1/3) {$p_3$};  
		
		\foreach \x in {-3, -2, -1, ...,2, 3}{
			\foreach \y in {-2,-1, ...,3}{
				\fill (\x,\y) circle (2pt);
			}
		}
		\end{tikzpicture}
		\caption{Constraints on $P$ circumscribing $\BP$ where $\BP$ is the long triangle. The shaded purple regions are open, pointed cones $\sigma_{\vq}$ for various lattice points $\vq$, see Lemma \ref{lemma:forbidden_cones}. By our discussion, we can assume that the vertices of $P$ lie in the shaded blue regions.}
		\label{fig:long-triangle-shards}
	\end{figure}
	
	We see that the width of $P$ in direction $\ve_2$ is realized by $\vp_1, \vp_3$. Now let $A$ be the intersection point of the line $L$ through $\vp_1$ and $\vp_3$ and the line $\{ y=x+1 \}$ and $B$ the intersection point of $L$ with $\{ y=-\frac{1}{2}x-\frac12 \}$, see Figure \ref{fig:long-triangle-shards}. We have that 
	
	\[
	\wdt(P,\ve_2) = p_{1,y}-p_{3,y} \leq A_x -B_x.
	\]
	
	
	If $A=(a,a+1)$, with $0 \leq a \leq 1$, we have $B=(\frac{a+3}{3a+1}, \frac{-2(a+1)}{3a+1})$. Thus we can estimate
	
	\[\wdt(P,\ve_2) \leq 3\frac{(a+1)^2}{3a+1}.\]	
	
	Let us call the right hand side of the above inequality $f(a)$. We easily see that $f'(a)$ is negative for $a<\frac13$, $0$ at $\frac13$ and positive after. Thus the maximum of $f$ on the interval $[0,1]$ is achieved exactly at $a=0$ and $a=1$, where we have $f(0)=f(1)=3$. We see that for $a=1$ we have $P=3\Tst$, while for $a=0$ the width in direction $\ve_1+\ve_2$ is $2$. 
\end{proof}

\subsubsection{The non-symmetric quadrilateral}
Consider $P$ such that 
\hypertarget{target:Bkite}{%
$$\BP = B^{\mathrm{kite}} = \conv \{ (-1, -1), (0, -1), (1, 1), (-1, 0) \}.$$
}

\begin{lemma} \label{lem:non-symm-quad-shards}
  Let $P$ be $1$-maximal with $\BP = \Bkite$. Then each vertex of $P$ lies in one of the following regions, and each region contains (at most) one vertex:

  \begin{align*}
    S_1 &= \conv \{(-1, -1),(-1, -3),(0,-1)\},\\
    S_2 &= \conv \{(0,-1),(1,-1),(1,1)\} \cup \conv\{(0,-1),(3,2),(1,1)\},\\
    S_3 &= \conv \{(1,1),(2,3),(-1,0)\} \cup \conv\{(1,1),(-1,1),(-1,0)\},\\
    S_4 &= \conv \{(-1,0),(-3,-1),(-1,-1)\}.
  \end{align*}

\end{lemma}

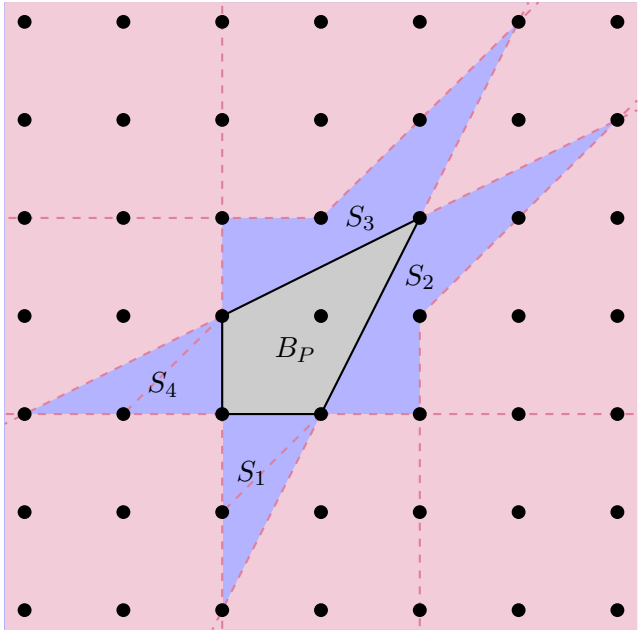
\begin{figure}[h]
\begin{tikzpicture}[scale=1.3]
  \fill[blue!30] (-3.2, -3.2) -- (3.2, -3.2) -- (3.2, 3.2) -- (-3.2, 3.2) -- cycle;

    \fill[purple!20] (-1, 0) -- (-1, 3.2) -- (-3.2, 3.2) -- (-3.2, -1.1) -- cycle;
      \fill[purple!20] (-1, -1) -- (-1, -3.2) -- (-3.2, -3.2) -- (-3.2, -1) -- cycle;
    \fill[purple!20] (0, 1) -- (2.2, 3.2) -- (-3.2, 3.2) -- (-3.2, 1) -- cycle;
    \fill[purple!20] (1, 1) -- (2.1, 3.2) -- (3.2, 3.2) -- (3.2, 2.1) -- cycle;

    \fill[purple!20] (1, 0) -- (3.2, 2.2) -- (3.2, -3.2) -- (1, -3.2) -- cycle;

    \fill[purple!20] (0, -1) -- (3.2, -1) -- (3.2, -3.2) -- (-1.1, -3.2) -- cycle;

  \draw[dashed, thick, purple!50] (-1, -1) -- (-3.2, -1);
  \draw[dashed, thick, purple!50] (-1, -1) -- (-1, -3.2);

  \draw[dashed, thick, purple!50] (-1, 0) -- (-1, 3.2);
  \draw[dashed, thick, purple!50] (-1, 0) -- (-3.2, -1.1); 

  \draw[dashed, thick, purple!50] (0, 1) -- (2.2, 3.2); 
  \draw[dashed, thick, purple!50] (0, 1) -- (-3.2, 1); 

  \draw[dashed, thick, purple!50] (1, 1) -- (3.2, 2.1);
  \draw[dashed, thick, purple!50] (1, 1) -- (2.1, 3.2); 

  \draw[dashed, thick, purple!50] (1, 0) -- (3.2, 2.2);
  \draw[dashed, thick, purple!50] (1, 0) -- (1, -3.2);

  \draw[dashed, thick, purple!50] (0, -1) -- (3.2, -1);
  \draw[dashed, thick, purple!50] (0, -1) -- (-1.1, -3.2);
  \draw[dashed, thick, purple!50] (0, -1) -- (-1, -2);
  \draw[dashed, thick, purple!50] (-1, 0) -- (-2, -1); 

  \fill[gray!40] (-1, -1) -- (0, -1) -- (1, 1) -- (-1, 0) -- cycle;

  \draw[thick] (-1, -1) -- (0, -1) -- (1, 1) -- (-1, 0) -- cycle;

    \foreach \x in {-3, -2, -1, ...,3}{
    \foreach \y in {-3, -2,-1, ...,3}{
      \fill (\x,\y) circle (2pt);
    }
  }

  \node at (-.26, -.34) {$\BP$};
  \node at (-.7, -1.6) {$S_1$};
  \node at (1, .4) {$S_2$};
  \node at (.4, 1) {$S_3$};
  \node at (-1.6, -.7) {$S_4$};

\end{tikzpicture}
\caption{The regions as described in Lemma \ref{lem:non-symm-quad-shards}.} 
\label{fig:non-symm-quad-shards}
\end{figure}

\begin{proof}
  Together with $\Bkite$, the regions $S_1, ..., S_4$ are the complements of the six open cones $\sigma_{\vq}$ as in Lemma \ref{lemma:forbidden_cones}, for $\vq = (-1, -1), (0, -1), (1, 0), (1, 1), (0, 1), (-1, 0)$. 
\end{proof}

We divide our analysis of $P$ according to where the vertices of $P$ $\vp_1 \in S_1$ and $\vp_4 \in S_4$ lie. We do so by splitting the regions $S_1$ and $S_4$ into two unimodular triangles: respectively $S_1^A$ and $S_1^B$, $S_4^A$ and $S_4^B$, divided by dashed red line in Figure \ref{fig:non-symm-quad-shards}.
\begin{equation}\label{eq:shard_decomp_ab}
\begin{split}
S_1 =& S_1^A \cup S_1^B = \conv \{(-1, -1), (-1, -2), (0, -1) \} \cup  \conv \{(0, -1), (-1, -3), (-1, -2)\},\\
S_4 = & S_4^A \cup S_4^B = \conv \{(-1, -1), (-1, 0), (-2, -1) \} \cup \conv \{(-1, 0), (-3, -1), (-2, -1)\}.
\end{split}
\end{equation}

\begin{lemma} \label{lem:non-symm-quad-both-close}
  Let $P$ such that $\BP = \Bkite$. Then we cannot have $\vp_1 \in S_1^A\setminus [(-1, 0), (-2, -1)]$ and $\vp_4 \in S_4^A\setminus [(0, -1), (-1, -2)]$.
\end{lemma}

\begin{proof}
Towards a contradiction, let $\vp_1, \vp_4$ be as described and let $\vp_2, \vp_3$ be the vertices of $P$ in $S_2, S_3$, respectively. Then by colinearity of $\vp_1, (0, -1),$ and $\vp_2$ and considering the region $S_2$, we see that 
  $$\vp_2 \in \conv \{(0, -1), (1, -1), (1, 0)  \}.$$
  A symmetric argument yields that 
  $$\vp_3 \in \conv \{(-1, 0), (0, 1), (-1, 1) \}.$$

  These restrictions make it impossible for $\vp_2$, $\vp_3$ and $(1,1)$ to be colinear, which is a contradiction.
\end{proof}

We can also handle the case when $\vp_1, \vp_4$ are both in the farther triangles from $\BP$. 

\begin{lemma} \label{lem:non-symm-quad-both-far}
  Let $P$ be a $1$-maximal polygon such that $\BP = \Bkite$ and suppose that the  vertices $\vp_1$ and $\vp_4$ of $P$ (cf.\ Lemma \ref{lem:non-symm-quad-shards}) such that 
    $\vp_1 \in S_1^B$ and
   $ \vp_4 \in S_4^B$ (cf.\ \eqref{eq:shard_decomp_ab}).
 Then $\wdt(P) < 3$.  
\end{lemma}

\begin{figure}[h]
\begin{tikzpicture}[scale=1.3]
  \fill[blue!30] (-3, -1) -- (-1, 0) -- (-2, -1) -- cycle;
  \fill[blue!30] (0, -1) -- (-1, -3) -- (-1, -2) -- cycle;
  \fill[blue!30] (-1, 0) -- (1, 1) -- (2, 3) -- cycle;
  \fill[blue!30] (0, -1) -- (1, 1) -- (3, 2) -- cycle;

  \fill[gray!30] (-1, -1) -- (0, -1) -- (1, 1) -- (-1, 0) -- cycle;

  \draw[thick] (-1, -1) -- (0, -1) -- (1, 1) -- (-1, 0) -- cycle;

    \foreach \x in {-3, -2, -1, ...,3}{
    \foreach \y in {-3, -2,-1, ...,3}{
      \fill (\x,\y) circle (2pt);
    }
  }
  \node at (-.26, -.34) {$\BP$};
  \node at (-2, -.7) {$S_4^B$};
  \node at (-.7, -2) {$S_1^B$};
  \end{tikzpicture}
  \caption{The configuration as in Lemma \ref{lem:non-symm-quad-both-far}. The vertices $\vp_2, \vp_3$ are further constrained (compare with $S_2, S_3$ of Figure \ref{fig:non-symm-quad-shards}), and thus the width in direction $\ve_1 - \ve_2$ is realized by $\vp_1, \vp_4$. }
\end{figure}
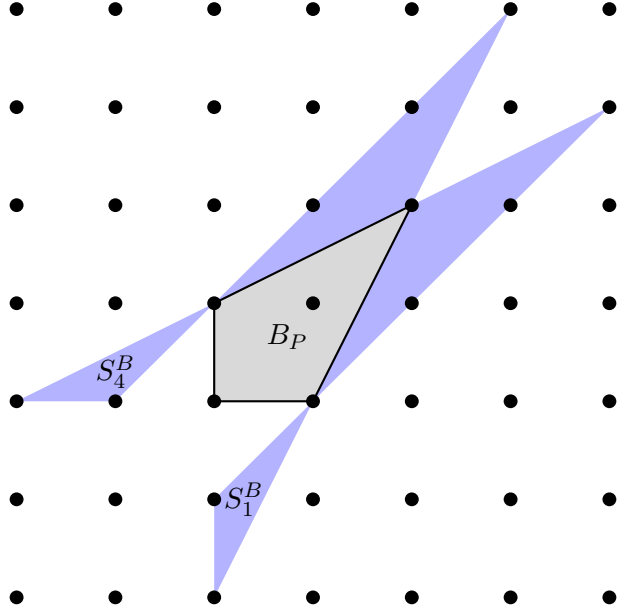

\begin{proof}
  Let $\vp_2, \vp_3$ be the vertices of $P$ in the regions $S_2, S_3$ respectively. Then the colinearity of $\vp_1, (0, -1), \vp_2$ and the colinearity of $\vp_4, ~(-1, 0),~ \vp_3$ combined with the further constraints on $S_2$ and $S_3$ from colinearity yields that the width of $P$ in direction $\ve_1 - \ve_2$ is realized by the vertices $\vp_1, \vp_4$. Hence, 
  $$\wdt(P, \ve_1 - \ve_2) = p_{1,x} - p_{4,x} - p_{1,y} + p_{4,y}.$$

  Consider the line $L$ passing through $\vp_1,~ (-1, -1),~ \vp_4$. We define new points $\vp_1', \vp_4'$ lying on the intersection of $L$ with $\{y = \frac{1}{2} x + \frac{1}{2}\}$ and $\{y = 2x - 1\}$, respectively. 

  Since the segment $[\vp_1, \vp_4]$ is contained in the segment $[\vp_1', \vp_4']$, we can upper-bound the width in direction $\ve_1 - \ve_2$.
  \begin{equation} 
    \wdt(P, \ve_1 - \ve_2) = p_{1,x} - p_{4,x} - p_{1,y} + p_{4,y} \leq p'_{1,x} - p'_{4,x} - p'_{1,y} + p'_{4,y}.
  \end{equation}
  Note that $\vp_1 ', \vp_4 '$ are fully determined by $p'_{1,x}$, so the above expression can be bounded $\leq 3$ via calculus in one variable. Furthermore, this bound is attained exactly when $\vp_1 = (0, -1)$ or $\vp_1 = (-1, -3)$, contradicting that $\BP = \Bkite$. 
\end{proof}

By symmetry across the $\{ y = x \}$ line, the remaining case to consider is when $\vp_1$ is far from $\BP$ and $\vp_4$ is close to $\BP$. Then, using colinearity to restrict the regions $S_2, S_3$, we are left with proving the following Lemma.

\begin{lemma} \label{lem:non-symm-quad-shards-final}
  Let $P$ such that $\BP = \Bkite$ and let the vertices of $P$ lie in the regions
  \begin{align*}
  	S_1&= \conv \{ (1, 1), (2, 1), (2,2) \}, \\
  	S_2&= \conv \{ (-1, 0), (0, 1), (-1, 1) \},\\
    S_3&= \conv \{ (-1, -1), (-1, 0), (-2, -1) \},\\
    S_4&= \conv \{ (0, -1), (-1, -3), (-1, -2) \},\\
  \end{align*}

  Then $\wdt(P) < 3$.
\end{lemma}

\begin{figure}[h]
\begin{tikzpicture}[scale=1.3]
  \fill[blue!30] (-2, -1) -- (-1, 0) -- (-1, -1)-- cycle;
  \fill[blue!30] (0, -1) -- (-1, -3) -- (-1, -2) -- cycle;
  \fill[blue!30] (-1, 0) -- (-1, 1) -- (0, 1) -- cycle;
  \fill[blue!30] (2, 1) -- (1, 1) -- (3, 2) -- cycle;

  \fill[gray!30] (-1, -1) -- (0, -1) -- (1, 1) -- (-1, 0) -- cycle;

  \draw[thick] (-1, -1) -- (0, -1) -- (1, 1) -- (-1, 0) -- cycle;

    \foreach \x in {-2, -1, ...,3}{
    \foreach \y in {-3, -2,-1, ...,2}{
      \fill (\x,\y) circle (2pt);
    }
  }
  \node at (-.24, -.36) {$\BP$};
  \node at (-.73, -2) {$S_1$};
  \node at (1.8, 1.2) {$S_2$};
  \node at (-.666, .666) {$S_3$};
  \node at (-1.333, -.666) {$S_4$};
  \end{tikzpicture}

\caption{The regions as in Lemma \ref{lem:non-symm-quad-shards-final}. For the width directions $\ve_1, \ve_2, \ve_1 - \ve_2$, they are realized by vertices of fixed regions. }
\end{figure}
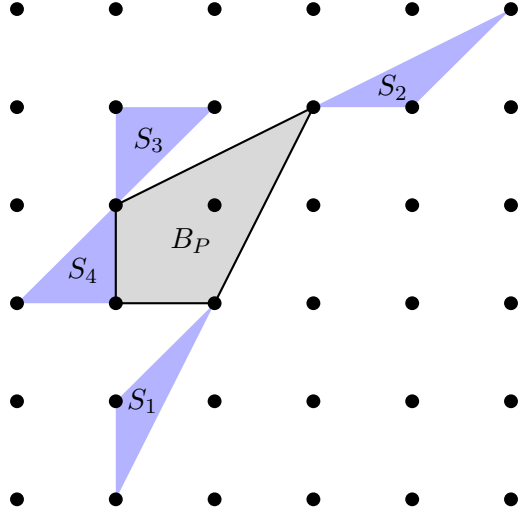

\begin{proof}
We consider the set of lattice directions $\hex = \{\pm \ve_1, \pm \ve_2, \pm (\ve_2 - \ve_1)\}$ and we show that $\wdt(P;\hex) \leq 3$ for any 1-maximal polygon $P$ with $\BP = \Bkite$, where we denote the blocking quadrilateral by $\Bkite$, and whose vertices lie in the regions $S_i$, $1\leq i \leq 4$.
Let $\vp_i$ denote the vertex of $P$ from the region $S_i$. One sees that
\begin{equation}
\label{eq:nsquad_wdt_realization}
\begin{split}
\wdt(P,\ve_1) &= \langle \ve_1, \vp_2 - \vp_4\rangle,\\
\wdt(P,\ve_2) &= \langle \ve_2, \vp_2 - \vp_1\rangle,\\
\wdt(P, \ve_2 - \ve_1) &= \langle \ve_2 - \ve_1, \vp_3 - \vp_1\rangle.
\end{split}
\end{equation}
In the following, we write $\vp_2 = (a,b)$.

\begin{claim}
\label{claim:geq5}
If $a+3b \geq 5$, then $\wdt(P,\ve_2-\ve_1)  < 3$.
\end{claim}

Suppose first that $a+3b = 5$. The line $\{x+3y = 5\}$ intersects the boundary of the region $S_2$ in the points $(\tfrac 75,\tfrac 65)$ and $(2,1)$. For $\lambda \in [0,1]$, let $\vp_2(\lambda) = \lambda(2,1) + (1-\lambda)(\tfrac 75,\tfrac 65)$.
 Moreover, let $\widetilde\vp_3(\lambda)\in S_3$ be the intersection point of the line $\{x=-1\}$ with the line generated by $\vp_2(\lambda)$ and $(1,1)$. Likewise, let $\widetilde\vp_1(\lambda)$ be the intersection point of the line $\{x=-1\}$ with the line generated by $\widetilde \vp_2(\lambda)$ and $(0,-1)$.
One checks that $\wdt(P',\ve_2 - \ve_1)\leq \langle \ve_2-\ve_1,\widetilde\vp_3(\lambda) - \widetilde\vp_1(\lambda) \rangle$.
The right hand side is a single variable function rational function $f(\lambda)$, which can be computed explicitly by simple planar geometry. It turns out that $f$ attains its unique maximum at $\lambda = 1$. This corresponds to $\vp_2 = (2,1)$. This implies $P = \conv\{(-1,1),(2,1),(-1,-2)\}$, i.e., $f(\lambda)\leq 3$ with equality only if $\lambda = 1$. However, $\lambda = 1$ corresponds to a polytope $P$ with $B_P \neq \Bkite$ and we deduce that $\wdt(P,\ve_2-\ve_1) < f(1) = 3$.

Next, let $a+3b > 5$. Observe that the width in direction $\ve_2 - \ve_1$ strictly increases, if we replace the point $\vp_2$ by the intersection point $(a',b')$ of the line generated by $\vp_2$ and $(0,-1)$ with the line  $\{x+3y=5\}$. In doing so, we keep $\vp_4$ and $\vp_1$ fixed and determine $\vp_3$ from the colinearity constraints. We obtain a new quadrilateral $P'$ with $\wdt(P,\ve_2-\ve_1) < \wdt(P',\ve_2 - \ve_1) < 3$. \\
\\
From now on, we will assume that $a+3b < 5$ and so, in particular, $a < 2$.
Moreover, we assume that $\wdt(P,\ve_1) \geq 3$. Suppose that this inequality was strict. Then, one can move $\vp_4$ along the segment $[\vp_4,\vp_1]$ towards $\vp_1$ and obtain a new point $\vp_4'$. Let $\vp_3'$ be the intersection of the line generated by $\vp_4'$ and $(0,-1)$ with the line generated by $\vp_2$ and $(1,1)$. Then, the quadrilateral $P' = \conv \{\vp_1,\vp_2,\vp_3',\vp_4'\}$ has $B_{P'} = \Bkite$ and its vertices lie in the respective regions. At this point we specify $\vp_4'$ such that $\wdt(P',\ve_1) = 3$ (recall that $a < 2$). By the intersection pattern of $[\vp_3',\vp_4']$ and $[\vp_3,\vp_4]$ at the blocking point $(-1,0)$, and since the vertices in $S_1$ and $S_4$ are unchanged, we find
\begin{equation}
\label{eq:wdt_comparison_PPprime}
\begin{split}
\wdt(P,\ve_2) &=\wdt(P',\ve_2).\\
\end{split}
\end{equation}
Next, we note that an arbitrary 1-maximal polygon $P'$ with $\wdt(P',\ve_1) =3$, $B_{P'} = \Bkite$ and vertices in the specified regions is determined by three parameters. Writing $\vp_2 = (a,b)$, $a,b\in\RR$, it follows that $\vp_4 = (a-3,c)$ for some $c$.
The points $\vp_1$ and $\vp_3$ are uniquely determined by the colinearity constraints.
For the values of the three parameters, we have the following inequalities, which follow from our assumptions and the containment in the regions, respectively.
\begin{equation}
\label{eq:nsquad_params}
a+3b < 5,\quad -a+2b\leq 1,\quad b\geq 1,\quad -1\leq c \leq a-2.
\end{equation}
In order to conclude the proof, by \eqref{eq:wdt_comparison_PPprime}, it suffices to prove that the polytope $P'$, parametrized by $a,b,c$ as in \eqref{eq:nsquad_params}, has $\wdt(P',\ve_2) < 3.$

The vertical width of $P'$ is $b-d$, where $d$ denotes the second coordinate of $\vp_1$ (cf.\ \eqref{eq:nsquad_wdt_realization}). Since $\vp_1$ is determined by the intersection of the line generated by $\vp_3$ and $(-1,-1)$ with the line generated by $\vp_1$ and $(0,-1)$, we can solve the corresponding linear system and compute $d$ in terms of $a,b,c$. This way, we obtain
\[
\begin{split}
\wdt(P',\ve_2) &= b-\left(-1 - \frac{(b+1)(c+1)}{a(c+1) - (a-2)(b+1)}\right).
\end{split}
\]
We observe that the value of $\langle \ve_2, \vp_1\rangle$ is increasing in $c$. By \eqref{eq:nsquad_params}, we estimate
\begin{equation}
\label{eq:D_func_est}
\begin{split}
\wdt(P',\ve_2) - 3 &\leq b-\left(-1  - \frac{(b+1)(a-1)}{a(a-1) - (a-2)(b+1)}\right) -3\\
&= \frac{(b-2)D - (b+1)(a-1)}{D},
\end{split}
\end{equation}
where $D= a(a-1) - (a-2)(b+1)$. Using the inequalities \eqref{eq:nsquad_params} we can estimate it as follows: 
\[
\begin{split}
D &= (a-1)a - (a-2)(b+1) = (a-1)(a-b-1) + b+1\\
&\geq (a-1)(a-b-1) + 3 = (a-1)(a-2b - 1 + b) + 3\\
&\geq (a-1)b+3\geq 3 > 0.
\end{split}
\]
It remains to show that the numerator $E$ in \eqref{eq:D_func_est} is negative. We substitute $x = a+3b$ and $y=b$ so that $E$ becomes
\[\begin{split}
E &= x^2y - 7xy^2 + 12y^3 - 2x^2 + 13xy - 19y^2 + 5x - 18y - 5\\
&= x^2  (y-2) + x(-7y^2 + 13y + 5) + 12y^3 - 19y^2 -18y -5.
\end{split}\]
Let $E'$ be the derivative of $E$ with respect to $x$. Since by \eqref{eq:nsquad_params} we have $y\leq \tfrac 32$ it follows that $y-2 < 0 $ and  since by the assumption of the claim we have $x < 5$, we deduce
\[
E' = 2x(y-2) - 7y^2 + 13y +5 > 10(y-2) - 7y^2 + 13y +5,
\]
where the right hand side is positive for $1\leq y\leq \tfrac 32$. Hence, $E$ is strictly increasing in $x$ and since $x< 5$, we have
\[
E < E(x=5) = 12(y-\tfrac 52)(y-1)^2 \leq 0.
\]
Thus we obtain from \eqref{eq:D_func_est} that $\wdt(P',\ve_2) < 3$ as desired.
\end{proof}

By combining Lemmas \ref{lem:non-symm-quad-both-close}, \ref{lem:non-symm-quad-both-far} and \ref{lem:non-symm-quad-shards-final} we can conclude the case of $\BP=\Bkite$.

\begin{lemma}
	Let $P$ such that $\BP = \Bkite$. Then $\wdt(P) < 3$.
\end{lemma}

\subsubsection{The trapezoid} \label{sec:trapezoid}
We consider now consider those polygons $P$ whose blocking polytope is
\hypertarget{target:Btrap}{
$$B_P=B^{\mathrm{trap}}=\conv \{ (-1, 0), (1, 0), (0, 1), (-1, 1)\}.$$
}

\begin{lemma} \label{lem:trapezoid-shards}
  Let $P$ be a $1$-maximal polygon such that $\BP = \Btrap$.
  Then each vertex of $P$ lies in one of the following regions, and each region contains at most one vertex::
  \begin{align*}
    S_1 & = \conv \{ (-1, 0), (-1, -2), (1, 0) \} \cup \conv \{ (-1, 0), (1, -2), (1, 0) \} \cup \conv \{ (-1, 0), (4, -2), (1, 0) \}, \\
    S_2 & = \{ 0 \leq y \leq 1, y \geq -x + 2 \},\\
    S_3 & = \conv \{ (-1, 1), (-1, 2), (0, 1) \}, \\
    S_4 & = \{x \leq 0, 0 \leq y \leq 1 \}.
  \end{align*}
\end{lemma}
These regions are visualized in Figure \ref{fig:trapezoid-shards}.
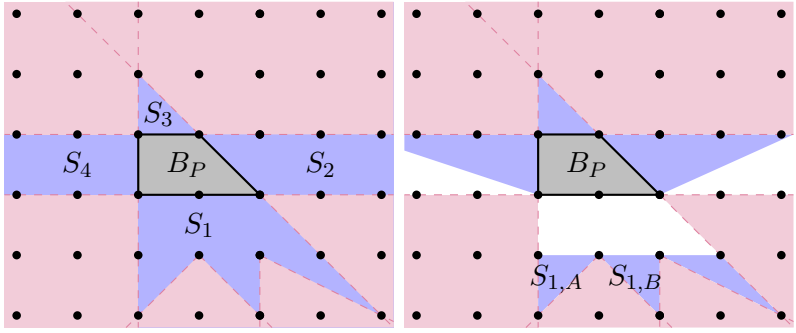
\begin{figure}[h]
\begin{tikzpicture}[scale=.8]

  \fill[blue!30] (3.2, 3.2) -- (3.2, -2.2) -- (-3.2, -2.2) -- (-3.2, 3.2) -- cycle;

  \fill[gray!50] (-1, 0) -- (1, 0) -- (0, 1) -- (-1, 1) -- cycle;
  \fill[purple!20] (0, 1) -- (3.2, 1) -- (3.2, 3.2) -- (-2.2, 3.2) -- cycle;

  \fill[purple!20] (1, 0) -- (3.2, 0) -- (3.2, -2.2) -- cycle;
  \fill[purple!20] (1, -1) -- (1, -2.2) -- (3.2, -2.2) -- (3.2, -2.1) -- cycle;
  \fill[purple!20] (-1, 0) -- (-3.2, 0) -- (-3.2, -2.2) -- (-1, -2.2) -- cycle;
  \fill[purple!20] (-1, 1) -- (-1, 3.2) -- (-3.2, 3.2) -- (-3.2, 1) -- cycle;
  \fill[purple!20] (0, -1) -- (1.2, -2.2) -- (-1.2, -2.2) -- cycle;

  \draw[purple!50, dashed] (0, 1) -- (3.2, 1); 
  \draw[purple!50, dashed] (0, 1) -- (-2.2, 3.2);

  \draw[purple!50, dashed] (1, 0) -- (3.2, 0);
  \draw[purple!50, dashed] (1, 0) -- (3.2, -2.2);

  \draw[purple!50, dashed] (1, -1) -- (3.2, -2.1);
  \draw[purple!50, dashed] (1, -1) -- (1, -2.2);

  \draw[purple!50, dashed] (-1, 0) -- (-3.2, 0);
  \draw[purple!50, dashed] (-1, 0) -- (-1, -2.2);
  \draw[purple!50, dashed] (-1, 1) -- (-1, 3.2);
  \draw[purple!50, dashed] (-1, 1) -- (-3.2, 1);

  \draw[purple!50, dashed] (0, -1) -- (1.2, -2.2);
  \draw[purple!50, dashed] (0, -1) -- (-1.2, -2.2);

  \draw[thick] (-1, 0) -- (1, 0) -- (0, 1) -- (-1, 1) -- cycle;

    \foreach \x in {-3, , -2, -1, ...,2, 3}{
    \foreach \y in {-2,-1, ...,3}{
      \fill (\x,\y) circle (2pt);
    }
  }
  \node at (-.2, .5) {$\BP$};
  \node at (0, -.5) {$S_1$};
  \node at (2, .5) {$S_2$};
  \node at (-.666, 1.333) {$S_3$};
  \node at (-2, .5) {$S_4$};

\end{tikzpicture}
\begin{tikzpicture}[scale=.8]
  \fill[blue!30] (0, -1) -- (-1, -2) -- (-1, -1) -- cycle;
  \fill[blue!30] (0, -1) -- (1, -1) -- (1, -2) -- cycle;
  \fill[blue!30] (1, -1) -- (2, -1) -- (3, -2) -- cycle;

  \fill[blue!30] (1, 0) -- (3.2, 1) -- (0, 1) -- cycle;

  \fill[blue!30] (0, 1) -- (-1, 2) -- (-1, 1) -- cycle;

  \fill[blue!30] (-1, 0) -- (-3.2, 2.2/3) -- (-3.2, 1) -- (-1, 1) -- cycle;

  \fill[gray!50] (-1, 0) -- (1, 0) -- (0, 1) -- (-1, 1) -- cycle;
  \fill[purple!20] (0, 1) -- (3.2, 1) -- (3.2, 3.2) -- (-2.2, 3.2) -- cycle;

  \fill[purple!20] (1, 0) -- (3.2, 0) -- (3.2, -2.2) -- cycle;
  \fill[purple!20] (1, -1) -- (1, -2.2) -- (3.2, -2.2) -- (3.2, -2.1) -- cycle;
  \fill[purple!20] (-1, 0) -- (-3.2, 0) -- (-3.2, -2.2) -- (-1, -2.2) -- cycle;
  \fill[purple!20] (-1, 1) -- (-1, 3.2) -- (-3.2, 3.2) -- (-3.2, 1) -- cycle;
  \fill[purple!20] (0, -1) -- (1.2, -2.2) -- (-1.2, -2.2) -- cycle;

  \draw[purple!50, dashed] (0, 1) -- (3.2, 1); 
  \draw[purple!50, dashed] (0, 1) -- (-2.2, 3.2);

  \draw[purple!50, dashed] (1, 0) -- (3.2, 0);
  \draw[purple!50, dashed] (1, 0) -- (3.2, -2.2);

  \draw[purple!50, dashed] (1, -1) -- (3.2, -2.1);
  \draw[purple!50, dashed] (1, -1) -- (1, -2.2);

  \draw[purple!50, dashed] (-1, 0) -- (-3.2, 0);
  \draw[purple!50, dashed] (-1, 0) -- (-1, -2.2);
  \draw[purple!50, dashed] (-1, 1) -- (-1, 3.2);
  \draw[purple!50, dashed] (-1, 1) -- (-3.2, 1);

  \draw[purple!50, dashed] (0, -1) -- (1.2, -2.2);
  \draw[purple!50, dashed] (0, -1) --  (-1.2, -2.2);

  \draw[thick] (-1, 0) -- (1, 0) -- (0, 1) -- (-1, 1) -- cycle;

    \foreach \x in {-3, , -2, -1, ...,2, 3}{
    \foreach \y in {-2,-1, ...,3}{
      \fill (\x,\y) circle (2pt);
    }
  }
  \node at (-.2, .5) {$\BP$};
  \node at (.6, -1.4) {$S_{1,B}$};
  \node at (-.7, -1.4) {$S_{1,A}$};
\end{tikzpicture}
\caption{Left: The regions $S_i$ as in Lemma \ref{lem:trapezoid-shards}. Right: The regions can be further constrained by consider the width in direction $\ve_2$.}
\label{fig:trapezoid-shards}
\end{figure}
\begin{lemma} \label{lem:trapezoid-thm}
  Let $P$ be a $1$-maximal polygon such that $\BP = \Btrap$. Then $\wdt(P) < 3$.
\end{lemma}
\begin{proof}
	We denote by $\vp_1, \vp_2, \vp_3, \vp_4$ the possibly degenerate vertices in the regions $S_1, S_2, S_3, S_4$. 
	Observe that, once we choose $\vp_1$ and $\vp_3$ inside their respective regions, we have completely determined the polygon $P$ thanks to the colinearity conditions between vertices of $P$ and of $\BP$. For ease of notation, we denote the coordinates of $\vp_3$ as $(a,b) = (p_{3,x}, p_{3,y})$, for some $(a,b) \in S_3$. We consider how the width of $P$ varies as we allow $\vp_1$ to vary within $S_1$. 
	
	If $p_{1, y} > -1$, then we have $\wdt(K, \ve_2) <3$ and we are done. Thus we can further restrict $S_1$ as pictured on the right hand side of Figure \ref{fig:trapezoid-shards}.
	We further observe that the unimodular transformation defined by the matrix $\begin{bmatrix}
	-1 & -1 \\ 0 & 1
	\end{bmatrix}$, which maps the left triangle of $S_1$ to the right one while preserving $\Btrap$, ensures that we can treat the case where $\vp_1$ lies in either of these triangles as the same case. 
	Thus we distinguish two cases: 
	
	\textbf{Case 1}: $\vp_1 \in S_{1,B} = \conv((1,-2),(1,-1),(0,-1))$.
	
	\begin{figure}
		
	\begin{tikzpicture}[scale=.8]
	\fill[blue!30] (0, -1) -- (1, -1) -- (1, -2) -- cycle;
	
	
	\fill[blue!30] (0, 1) -- (-1, 2) -- (-1, 1) -- cycle;
	
	
	\fill[gray!50] (-1, 0) -- (1, 0) -- (0, 1) -- (-1, 1) -- cycle;
	
	

	
	
	
	\draw[purple!50] (-.6, 1.4)--(-1, 1) --  (-2.25, -.5);
	\draw[purple!50] (-.6, 1.4)--(0, 1) --  (2.25,-.5);
	
	\draw[green!50] (.6, -1.6)--(1, 0) --  (1.25, 1);
	\draw[green!50] (1, -1.6)--(-1, 0) --  (-2.25, 1);
	
	\draw[thick] (-1, 0) -- (1, 0) -- (0, 1) -- (-1, 1) -- cycle;
	
	\draw[thick] (.6, -1.6) -- (1, -1.6);
	
	\foreach \x in {-3, , -2, -1, ...,2, 3}{
		\foreach \y in {-2,-1, ...,2}{
			\fill (\x,\y) circle (2pt);
		}
	}
	\fill (-.6, 1.4) circle (1.5pt);
	\fill (.75,-1.6) circle (1.5pt);
    \node at (-1, 1.4) {\Small $\vp_3$};
    \node at (.4, -1.6) {\Small $\vp_1$};
    \node at (-2.5, -.5) {\Small $L_3$};
	\node at (2.5, -.5) {\Small $R_3$};
	\node at (-2.4, 1.1) {\Small $L'_1$};
	\node at (1.5, 1.1) {\Small $R'_1$};
	\end{tikzpicture}
	
\end{figure}
	
	If we have $p_{1,y}> b-3$, then $\wdt(P, \ve_2)=b-p_{1,3}<3$ and we are done. 
	We will now show that in all other cases, we have $\wdt(P, \ve_1) <3$.
	If $p_{1,y}<b-3$, we observe that $\wdt(P, \ve_1)<\wdt(P', \ve_1),$ where $P'$ is the polygon with $B_{P'}=\Btrap$ with vertices $\vp'_3 = \vp_3$ and $\vp'_1 = (p_{1,x}, b-3)$. 

	It is therefore enough to consider the situation $\vp_3=(a,b),~ \vp_1=(c, b-3)$. We can write $\wdt(P,\ve_1)$ in terms of $a,b,c$, with restrictions on $a,b,c$ given by the regions $S_{1,B}$ and $S_3$, and show that this value is always less than $3$.
	Observe that $\wdt(P, \ve_1)= p_{2,x} - p_{4,x}$. We have $\vp_2 = R_1 \cap R_3$ and $\vp_4 = L_1 \cap L_3$, where the lines are defined as
	\begin{align*}
	L_1 &= \aff(\vp_1, (-1,0)),\quad R_1 = \aff(\vp_1, (1,0)),\\
	L_3 &= \aff(\vp_3, (-1,1))= \{(x,y): y= \tfrac{b-1}{a+1}x+\tfrac{a+b}{a+1}\},\\
	R_3 &= \aff(\vp_3, (0,1)) = \{(x,y): y= \tfrac{b-1}{a}x+1\}.	
	\end{align*}
	
	The first two lines have equations depending on $c$. To remove $c$ as a parameter, we consider instead the lines $L'_1$ and $R'_1$ as follows. 

	\begin{align*}
		L'_1 &= \aff((1,b-3), (0,-1)) =  \{(x,y): y= \tfrac{b-3}{2}(x+1)\},\\
		R'_1 &= \aff((2-b,b-3), (0,1)) =  \{(x,y): y= \tfrac{b-3}{1-b}(x-1)\}.
	\end{align*}
	
	The points $(2-b, b-3)$ and $(1,b-3)$ represent the extremes of the segment at height $b-3$ within $S_{1,b}$, so the extreme values of $x$-coordinates possible for $\vp_1$. This guarantees that the points $\vq_4 = L'_1 \cap L_3$ and
	$\vq_2 = R'_1 \cap R_3$ satisfy $q_{2,x} \geq p_{2,x}$ and $q_{4,x} \leq p_{4,x}$, and hence \[\wdt(P,\ve_1) \leq q_{2,x} - q_{4,x}.\] We now have the advantage that the right hand side now only depends on $a$ and $b$. We calculate
	\[
	q_{2,x} = -\frac{2a}{ab-3a+(b-1)^2}, \quad
	q_{4,x} = \frac{-ab +5a +b+3}{ab-3a-b-1},
	\]
	and thus
	\begin{equation}\label{eq:final_ineq_trapezoid}
	\wdt(P,\ve_1) \leq -\frac{2a}{ab-3a+(b-1)^2} -\frac{-ab +5a +b+3}{ab-3a-b-1}.
	\end{equation}
	
	Let \[f(a,b) = -\frac{2(b-1)p(a,b)}
	{D_1\,D_2},\] where $p(a,b)=a^{2}b-3a^{2}+ab^{2}-4ab+3a-b^{2}+b$, $D_1=ab-3a-b-1$ and $D_2=ab-3a+b^{2}-2b+1$. A simple calculation shows that $f(a,b)+3$ is equal to the right hand side of \eqref{eq:final_ineq_trapezoid}. We thus want to show that $f(a,b)<0$ for $(a,b) \in \inter(S_3)$. In this region we have $b-1>0$, and observing that $0 <a(b-3) < 2$ and $-b-1<-2$, we have that $D_1 = a(b-3)-b-1<0$, $D_2=a(b-3)+(b-1)^2>0$. Therefore it is enough to show that $p(a,b)>0$. To see this, consider $p(a,b)$ as a quadratic polynomial in $a$ with parameter $b$, that is, $p(a,b)=q_b(a)=(b-3)a^2+(b^2-4b+3)a-b^2+b$. Since $(b-3)$ is negative, the maximum of $q_b(a)$ is achieved at the point where its derivative $q_b'(a)=2(b-3)a+(b-3)(b-1)$ is equal to $0$. This happens at $a=\frac{1-b}{2}$, and thus we have $p(a,b)\leq p(\frac{1-b}{2},b) = -\frac{(b-1)(b^2+4)}{4}\leq 0$, since $1<b<2$. 
	
	We are left with analyzing the value of $\wdt(P,\ve_1)$ on the boundary of $S_3$, which involves maximizing a one-parameter expression. We have  $\wdt(P,\ve_1)<3$ on $S_3$ except on the edge $b=1$, and on the point $(-1,2)$ where it equals $3$. In these cases $P$ is a unimodular transformation of $3\Delta_3$.

	\textbf{Case 2}: $\vp_1 \in S_{1,A} = \conv((-1,-2),(-1,-1),(0,-1))$. 
	
	With the same arguments as the previous case, we can observe that 
	$\wdt(P, \ve_1) \leq s_{2,x}-s_{4,x},$
	where $\vs_2 = L''_1 \cap L_3$, $\vs_4 = R''_1 \cap R_3$ and the lines $L''_1$ and $R''_1$ go through the points $(b-2,b-3)$, $(-1,0)$ and $(-1,b-3)$, $(1,0)$ respectively. We then see, analogously to case 1, that
	\[\wdt(P, \ve_1) < \frac{ab-5a}{ab-3a+2b-2}-\frac{2a+b^2-2b+3}{ab-3a-b^2+3b-4}.\]
	
	Let $A:=ab-3a+2b-2$ as well as $B:=ab-3a-b^{2}+3b-4$
	and consider
	\[
	f(a,b):=\frac{ab-5a}{A}-\frac{2a+b^{2}-2b+3}{B}-3.
	\]
	We want to show $f(a,b)<0$ on $S_3$. A direct algebraic simplification gives
	\[
	f(a,b) = -\frac{(b-1)p(a,b)}{AB},
	\]
	where
	$p(x,y)=2a^{2}b-6a^{2}-a b^{2}+10ab-21a-4b^{2}+14b-18$.
	
	To show $f(a,b) <0$ on the interior $S_3$, we observe that  $b-1>0$, $A>0$ and $B<0$ hold for any $(a,b) \in S_3$. It is therefore enough to show 
	$p(a,b)<0$ on the interior of $S_3$. Consider now $p(a,b)$ as a quadratic $q_b(a)$ in the variable $a$ with $b$ as a parameter. The derivative is $q_b'(a)=(b-3)(4a -(b-7))$. This is a line with negative slope (since $b<2$), and therefore $q_b'(a)$ achieves its maximum value on the allowed interval $-1<a<0$ at $-1$. Thus $q_b'(a) < -(b-3)^2 <0$. Since the derivative of $q_b$ is negative on $[-1,0]$, this means that its maximum is achieved at $-1$, so $q_b(a) = p(a,b) \leq p(1,b)= -3(b-1)^2 <0$ for $1<b<2$. 
	
	We are left with analyzing the value of $\wdt(P,\ve_1)$ on the boundary of $S_3$, which reduces to a one parameter calculation and shows that we have  $\wdt(P,\ve_1)<3$ on $S_3$ except on the edge $b=1$, where it is equal to $3$. Indeed in this case the construction yields a unimodular transformation of $3\Delta_3$.
\end{proof}

\subsubsection{The standard triangle} \label{sec:standard-triangle}

We consider the case where $\BP$ is a translate of the standard triangle
\hypertarget{target:Bst}{%
$$\BP = B^{\mathrm{st}} = \conv \{ (-1, -1), (0, -1), (-1, 0) \}.$$
}
and the interior point of $P$ lies outside $\BP$. Up to an affine unimodular equivalence, we can assume this interior point is $(0, 0)$, by $1$-maximality of $P$ (cf.\ Section \ref{ssec:class}). 

\begin{lemma} \label{lem:standard-triangle-shards}

  Let $P$ be a $1$-maximal polygon such that $\BP = \Bst$
  and $(0, 0)$ the internal lattice point of $P$. 
  Up to symmetry across the $\{y=x\}$ line, each vertex of P lies in one of the following regions, and each region contains at most one vertex
  \begin{align*}
    S_1 & = \{-1 \leq x \leq 0, y \leq 2x \}, \\
    S_2 & =\conv \{(0, 0), (1, 1), (2, 3), (1, 2)\}, \\
    S_3 & = \{-1 \leq y \leq 0, y \geq x + 1\}. 
  \end{align*}

\end{lemma}

\begin{figure}[h]
    \begin{tikzpicture}
    
      \fill[gray!50] (-1, -1) -- (0, -1) -- (-1, 0) -- cycle;
      \fill[purple!20] (1, 1) -- (2.1, 3.2) -- (3.2, 3.2) -- (3.2, 2.1) -- cycle;
      \fill[purple!20] (0, 1) -- (0, 3.2) -- (2.2, 3.2) -- cycle;
      \fill[purple!20] (-1, -1) -- (-4.2, -1) -- (-4.2, -4.2) -- (-1, -4.2) -- cycle;

        \fill[blue!30] (0, 0) -- (1, 1) -- (2, 3) -- (0, 1) -- cycle;
        \fill[blue!30] (-1, 0) -- (-2, -1) -- (-4.2, -1) -- (-4.2, 0) -- cycle;
        \fill[blue!30] (0, -1) -- (0, -4.2) -- (-1, -4.2) -- (-1, -3);

      \draw[dashed, purple!50] (0, 1) -- (2.2, 3.2);
      \draw[dashed, purple!50] (0, 1) -- (0, 3.2);

      \draw[dashed, purple!50] (1, 1) -- (2.1, 3.2);
      \draw[dashed, purple!50] (1, 1) -- (3.2, 2.1);

      \draw[dashed, purple!50] (-1, -1) -- (-4.2, -1);
      \draw[dashed, purple!50] (-1, -1) -- (-1, -4.2);

        \coordinate (q1) at (-1, -1);
        \coordinate (q2) at (0, -1);
        \coordinate (q3) at (-1, 0); 

        \draw [thin, black] (q1) -- (q2) -- (q3) -- cycle;

        \node at (-.73, -.73) {$\BP$};
        \node at (-.5, -3) {$S1$}; 
        \node at (.5, 1) {$S2$};
        \node at (-2.5, -.5) {$S3$};
      \foreach \x in {-4, -3,-2, -1, 0, 1, 2, 3}{
        \foreach \y in {-4, -3, -2, -1,0, 1, 2, 3}{
          \fill (\x,\y) circle (2pt) {};
        }
      }
    \end{tikzpicture}
    \caption{The regions as described in Lemma \ref{lem:standard-triangle-shards}.}
    \label{fig:standard-triangle-shards}
\end{figure}
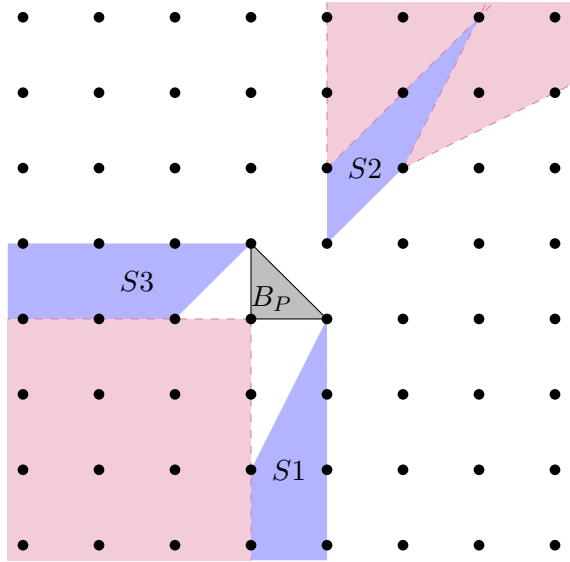

\begin{proof}
Since $\vec 0 \notin B_P$, there exists an edge $e$ of $B_P$ such that $\vec 0 \in \inter (S_e)$ (see Section \ref{sec:shards}).  Within $S_e$ there is a vertex $\vp_e$ of $P$ and we may assume that $\vp_e$ is above the line $\{y=x\}$. Since $\vec 0\in \inter P$, we have $\vec 0 \in \inter (\conv(e\cup\{\vp_e\}))$. This implies that $\vp_e \in \sigma_{\vec 0} = \RR^2_{>0}$ (cf.\ Section \ref{sec:shards}). Lemma \ref{lemma:forbidden_cones} applied to the open cones $\sigma_{(1, 1)}, \sigma_{(0, 1)}$ yields the region $S_2$ from Figure \ref{fig:standard-triangle-shards} as the feasible region for placing $\vp_e$ (from now on denoted by $\vp_2$). 
%
%
%

  Now consider the edges emanating from $\vp_2$. By definition of $\BP$, they necessarily pass through the vertices $(-1, 0)$ and $(0, -1)$, call the vertices of $P$ on the end of these edges $\vp_3, \vp_1$, respectively. Then the regions $S_3, S_1$ follow from these colinearities, combined with excluding the open cone $\sigma_{(-1, -1)}$. 
\end{proof}

We are now equipped to prove a width bound on $P$.

\begin{lemma} \label{lem:standard-triangle-width}
  Let $P$ be a $1$-maximal polygon such that $\BP = \Bst$
  and $(0, 0)$ the internal lattice point of $P$. 
  Then $\wdt(P) < 3$. 
\end{lemma}
 \begin{proof}
   Let $S_1, S_2, S_3$ be as in Lemma \ref{lem:standard-triangle-shards}, and let $\vp_i \in S_i$ be the vertices of $P$. 
  Let $\vp_3 = (a, b)$. Then we can subdivide $S_3$ into subregions $S_{3,i}, S_{3,ii}, S_{3,iii}$ by adding the respective constraints (see Figure \ref{fig:standrad-triangle-first-case}):
      \begin{enumerate}
        \item $b \leq -a - 2$ and $\frac{1}{2} a + \frac{1}{2} \leq b$,
        \item $b \geq -a-2$ and $\frac{1}{2} a + \frac{1}{2} \leq b$,
        \item $\frac{1}{2} a + \frac{1}{2} \geq b$. 
    \end{enumerate}

  \textbf{Case 1}: Colinearity of $\vp_3,\vp_2$ and $(-1,0)$ implies that the vertex $p_2$ lies in $S_2 \cap \{y \leq \frac{1}{2} x + \frac{1}{2} \}$, which implies $p_{2,y}\leq 1$. 
  Similarly, colinearity of $\vp_1,(-1, -1), \vp_3$, the vertex $\vp_1$ lies in $S_1 \cap \{y \geq -x - 2\}$, and thus $p_{1,y}\geq -2$. 

    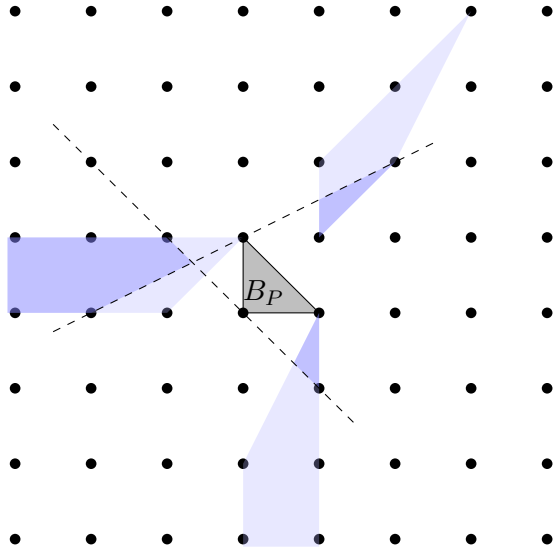
\begin{figure}[h]
    \begin{tikzpicture}

      \fill[gray!50] (-1, -1) -- (0, -1) -- (-1, 0) -- cycle;
      \foreach \x in {-4, -3,-2, -1, 0, 1, 2, 3}{
        \foreach \y in {-4, -3, -2, -1,0, 1, 2, 3}{
          \fill (\x,\y) circle (2pt) {};
        }
      }
        \coordinate (q1) at (-1, -1);
        \coordinate (q2) at (0, -1);
        \coordinate (q3) at (-1, 0); 

        \draw [thin, black] (q1) -- (q2) -- (q3) -- cycle;
        \fill[blue!30, opacity = .3] (0, 0) -- (1, 1) -- (2, 3) -- (0, 1) -- cycle;
        \fill[blue!30, opacity = .3] (-1, 0) -- (-2, -1) -- (-4.1, -1) -- (-4.1, 0) -- cycle;
        \fill[blue!30, opacity = .3] (0, -1) -- (0, -4.1) -- (-1, -4.1) -- (-1, -3) -- cycle;

        \fill[blue!60, opacity = .3] (-4.1, 0) -- (-2, 0) -- (-1.66666666, -.33333333) -- (-3, -1) -- (-4.1, -1) -- cycle;
        \fill[blue!60, opacity = .3] (0, 0) -- (1, 1) -- (0, .5) -- cycle;
        \fill[blue!60, opacity = .3] (0, -1) -- (0, -2) -- (-.333333, -1.6666666) -- cycle;

        \draw[thin, black, dashed] (-3.5, 1.5) -- (.5, -2.5); 
        \draw[thin, black, dashed] (-3.5, -1.25) -- (1.5, 1.25);
        \node at (-.73, -.73) {$\BP$};
    \end{tikzpicture}
    \cprotect\caption{The first case of Lemma \ref{lem:standard-triangle-width} is represented as the dark blue sub-region on the left.}
    \label{fig:standrad-triangle-first-case}
\end{figure}

    Thus $\wdt(P, \ve_2) \leq 3$. For equality to hold we would require $\vp_1 = (0, -1)$ and $\vp_2 = (1,1)$ (see Figure \ref{fig:standrad-triangle-first-case}), but as these two points are not colinear with $(0,-1)$, as required, this is impossible. 

    \textbf{Case 2}:
    The restrictions of $S_{3,ii}$ imply $p_{3,x} \leq -2$, and in a similar vein to the previous case, colinearities imply $p_{2,x}\leq 1$. This implies that $\wdt(P,\ve_1)\leq 3$, and since both cannot happen at the same time, the inequality is strict.
    
    \textbf{Case 3}: We will show that $$\min\{\wdt(P,\ve_1), \wdt(P,\ve_1-\ve_2)\}<3.$$
    
    The width in direction $\ve_1$ is achieved by $\vp_2$ and $\vp_1$, while that in direction $\ve_1-\ve_2$ by $\vp_3$ and $\vp_1$. Whenever $\wdt(P, \ve_1) > 3$, we can perturb $P$ 
    into a new polygon $P'$ with $B_{P'} = \Bst$, the same vertex $\vp_3=(a,b)$ as $P$ in region $S_{3, iii}$ and with  $\wdt(P', \ve_1) = 3$. See Figure \ref{fig:standard-triangle-perturbation}. The vertices $\vp'_2$ and $\vp'_1$ are then uniquely determined by colinearities, and we observe that 
        $$\wdt(P', \ve_1 - \ve_2) \geq \wdt(P, \ve_1 - \ve_2).$$
        
    It is thus enough to prove that $\wdt(P', \ve_1 - \ve_2) < 3$, which we do in the following.
    \begin{figure}[h]
    \centering
    \begin{tikzpicture}[scale = 1]
        \fill[blue!30] (0, 0) -- (1, 1) -- (2, 3) -- (0, 1) -- cycle;
        \fill[blue!30] (-1, 0) -- (-2, -1) -- (-4.1, -1) -- (-4.1, 0) -- cycle;
        \fill[blue!30] (0, -1) -- (0, -4.1) -- (-1, -4.1) -- (-1, -3) -- cycle;

        \fill[gray!50] (-1, -1) -- (-1, 0) -- (0, -1) -- cycle;
        \foreach \x in {-4, -3,-2, -1, 0, 1, 2, 3}{
            \foreach \y in {-4, -3, -2, -1,0, 1, 2, 3}{
                    \fill (\x,\y) circle (2pt) {};
            }
        }
        \coordinate (q1) at (-1, -1);
        \coordinate (q2) at (0, -1);
        \coordinate (q3) at (-1, 0); 

        \draw [thin, black] (q1) -- (q2) -- (q3) -- cycle;

        \node at (-.73, -.73) {$\BP$};
        \node at (-.5, -3) {$S1$}; 
        \node at (.6, 1.7) {$S2$};
        \node at (-3.5, -.5) {$S3$};
        \node[draw, circle, inner sep = .5pt, fill] at (-2.333, -3/4) {}; 
        \node[draw, circle, inner sep=.5pt, fill] at (.666666, .937392) {};
        \draw [thin, black] (-2.333, -3/4) -- (.666666, .937392);
        \node[draw, circle, inner sep =.5pt, fill] at (-2/33, -207/176) {};
        \draw[thin, black, dashed] (-2/33, -207/176) -- (.666666, .937392);
        \node[draw, circle, inner sep = .5pt, fill] at (1.03, 1.14216) {};
        \draw[thin, black] (1.03, 1.14216) -- (-2.333, -3/4);
        \draw[thin, black] (-2.333, -3/4) -- (-2/33, -207/176);
        \coordinate (p1) at (-.082, -1.172);
        \node[draw, circle, inner sep = .5pt, fill] at (p1) {};
        \draw[thin, black] (p1) -- (1.03, 1.14216); 
        \node at (-2.4, -.6) {$\vp_3$};
        \node at (.5, 1) {$\vec p_2'$};
        \node at (1, 1.3) {$\vp_2$};
    \end{tikzpicture}
    \cprotect\caption{When $\wdt(P, \ve_1) > 3$, we can slide $\vp_2$ to find a $P'$ (the dashed triangle) such that $\wdt(P', \ve_1 - \ve_2) \geq \wdt(P, \ve_1 - \ve_2)$}
    \label{fig:standard-triangle-perturbation}
    \end{figure}
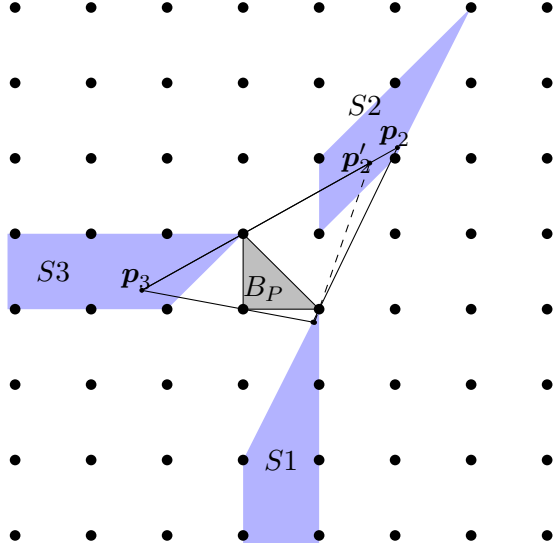
    The vertices $\vp_2', \vp_1'$ of $P'$ are completely determined by the vertex $\vp_3$, by $\wdt(P', \ve_1)=3$ and the colinearites given by $\Bst$ as the inscribed polygon. 
Their precise coordinates are 
    \begin{align*}
        \vp'_2 & = (a+3, b+ \frac{3b}{a+1}), \\
        \vp'_1 & = (\frac{(b+1)(a+3)}{b-2}, -1+\frac{(b+1)(ab+a+4b+1)}{(a+1)(b-2)}).
    \end{align*}
    Thus, we can write $\wdt(P', \ve_1-\ve_2)$ in terms of $a,b$ as follows.
    $$\wdt(P', \ve_1 - \ve_2) =  \langle \ve_1-\ve_2, \vp_3 - \vp_1\rangle = \frac{3(a-b)(a+b+1)}{(a+1)(b-2)}.$$
    To prove $\wdt(P', \ve_1 - \ve_2) < 3$, it suffices that for any $(a,b)\in S_{3,iii}$ we have 
    $$\frac{(a+1)(b-2) - (a-b)(a+b+1)}{(a+1)(b-2)} > 0.$$ 

    Note that $a+1$, $b-2$ are both negative on $S_3 \setminus \{(-1, 0)\}$, hence the denominator is always positive. So we are reduced to showing positivity of the numerator on $S_{3,iii}$, that is
    \begin{equation} \label{eq:hyperbola-eq}
      -a^2 +ab-3a+b^2+2b-2 > 0.
    \end{equation}

    To show this, we proceed as the end of Lemma \ref{lem:trapezoid-thm} for the trapezoid: we consider the left hand side $p_b(a)=-a^2 +ab-3a+b^2+2b-2 $ as a quadratic polynomial in $a$ with parameter $b$. Since its leading term is negative, it achieves its minimum value among $a \in [b-1, 2b-1]$ on the boundary. We calculate $p_b(b-1)=b^2$ and $p_b(2b-1)=-(b+1)b$. Since within $S_{3,iii}$ we have $-1 \leq b \leq 0$, we see that $p_b(a)>0$ everywhere except for $p_3=(-3,-1)$ and $p_3=(-1,0)$ and $p_3=(-1,0)$, where we see that the width is also strictly less than $3$.
  \end{proof}


\subsubsection{The terminal triangle}\label{sssec:terminal}
We consider $P$ such that 
\hypertarget{target:Bterm}{%
	$$\BP = B^{\mathrm{term}} = \conv \{\vq_1, \vq_2, \vq_3\},$$
}
where $\vq_1=(1, 0), \vq_2=(0, 1)$ and $\vq_3=(-1, -1)$.
We label the vertices of $P$ as in the right of Figure \ref{fig:terminal-triangle-one-shard}, with vertex $\vp_i$ opposite of $\vq_i$ for each $i=1,2,3$.

\begin{lemma} \label{lem:terminal-triangle-shards}
	Let $P$ be a $1$-maximal polytope such that $\BP = \Bterm$. 
Then, up to a reflection across $\{y = x\}$, each vertex of $P$ lies in one of the following regions, and each region contains at most one vertex: 
	\begin{align*}
	S_1 &= \conv \{(-2, -1), (-3, -2), (-1, -1)\}, \\
	S_2 &= \conv \{(1, -1), (2, -1), (1, 0)\}, \\
	S_3 &= \conv \{(1, 2), (1, 3), (0, 1)\}. 
	\end{align*}
	
\end{lemma}
\begin{proof}
	By symmetry over $\{ y = x \}$, we assume without loss of generality that $\vp_{3}$ lies above the  $\{ y = x \}$ line.
	Lemma \ref{lemma:forbidden_cones} applied to the points $(0, 1)$ and $(1, 1)$ yields that $\vp_{3}$ lies in $\conv \{(1, 0), (1, 3), (0, 1)\}$--the shaded blue triangle on the left of Figure \ref{fig:terminal-triangle-one-shard}.
	
	Then, by the colinearity of $\vp_3, \vq_1, \vp_2$, it follows that $p_{2,x} \geq 1$. Combined with the cones $\sigma_{(1, 0)}, \sigma_{(1, -1)}$ of Lemma \ref{lemma:forbidden_cones}, this yields that $\vp_2 \in S_2$. 
	Similarly, colinearity of $\vp_2, \vq_3, \vp_1$ yields one of the sides of $S_1$, and the cones $\sigma_{(-1,-1)}, \sigma_{(-2,-1)}$ the other two sides.  To conclude, we can further restrict the region of $\vp_3$ to $S_3$ by using the colinearity of $\vp_3, \vq_2, \vp_1$.
\end{proof}

The next statement now follows directly from considering the slopes of the edges $P$.

\begin{lemma} \label{lem:terminal-triangle-width-attainers}
	Let $P$ be a $1$-maximal polytope such that $\BP = \Bterm$. Let the vertices of $P$ be $\vp_i \in S_i$, $1\leq i \leq 3$ as in Lemma \ref{lem:terminal-triangle-shards}. Then,
	
	\begin{enumerate}
		\item $\wdt(P,\ve_1) = \langle \ve_1,\vp_2 - \vp_1\rangle$,
		\item $\wdt(P,\ve_2) = \langle \ve_2,\vp_3 -\vp_1\rangle$,
		\item $\wdt(P, \ve_1-\ve_2) = \langle \ve_1-\ve_2 , \vp_2 - \vp_3\rangle$.
	\end{enumerate}
	%
\end{lemma}

\begin{figure}[h]
	\centering
	\begin{tikzpicture}[scale = 1]
	\fill[blue!30] (0, 1) -- (1, 0) -- (1, 3) -- cycle; 
	
	\fill[purple!20] (1, 0) -- (2.2, 0) -- (2.2, 3.2) -- (1, 3.2) -- cycle;
	\fill[purple!20] (0, 1) --(1.1, 3.2) -- (-2.2, 3.2) -- cycle;
	
	\draw[purple!50,dashed] (1, 0) -- (2.2, 0);
	\draw[purple!50,dashed] (1, 0) -- (1, 3.2); 
	
	\draw[purple!50,dashed] (0, 1) -- (1.1, 3.2);
	\draw[purple!50,dashed] (0, 1) -- (-2.2, 3.2); 
	\fill[gray!50] (1, 0) -- (0, 1) -- (-1, -1) -- cycle;
	\foreach \x in {-3,-2, -1, 0, 1, 2}{
		\foreach \y in {-2, -1,0, 1, 2, 3}{
			\fill (\x,\y) circle (1pt) {};
		}
	}
	
	\coordinate (q1) at (-1, -1);
	\coordinate (q2) at (0, 1);
	\coordinate (q3) at (1, 0);
	
	\draw [thick, black] (q1) -- (q2);
	\draw [thick, black] (q1) -- (q3);
	\draw [thick, black] (q2) -- (q3);

	\node at (-.86, -1.28) {$\vq_3$};
	\node at (.25, .94) {$\vq_2$};
	\node at (.85, 0) {$\vq_1$};
	
	\coordinate (p23) at (.65, 1.85);
	\coordinate (p13) at (1.03, -.15);
	
	\node[draw, circle, inner sep=1pt, fill] at (p23) {};
	\node[draw, circle, inner sep=1pt, fill] at (p13) {};
	
	
	
	\draw [thick, gray, dashed] (p13) -- (p23);
	
	
	\node at (.65, 2.02) {$\vp_{3}$};
	\node at (1.23, -.15) {$\vp_{2}$}; 

	\end{tikzpicture}
	\hspace{1cm}
	\begin{tikzpicture}[scale = 1]

	\fill[blue!30] (0, 1) -- (1, 2) -- (1, 3) -- cycle; 
	\fill[blue!30] (1, -1) -- (2, -1) -- (1,0) -- cycle;
	\fill[blue!30] (-1, -1) -- (-2, -1) -- (-3, -2) -- cycle;
	\fill[gray!50] (1, 0) -- (0, 1) -- (-1, -1) -- cycle;

	\coordinate (q1) at (-1, -1);
	\coordinate (q2) at (0, 1);
	\coordinate (q3) at (1, 0);

	\draw [thick, black] (q1) -- (q2);
	\draw [thick, black] (q1) -- (q3);
	\draw [thick, black] (q2) -- (q3);

	\fill[blue!30] (0, 1) -- (1, 2) -- (1, 3) -- cycle; 
	\fill[blue!30] (1,0) -- (1, -1) -- (2, -1) -- cycle; 
	\fill[blue!30] (-1, -1) -- (-2, -1) -- (-3, -2) -- cycle; 
	
	\node at (-.86, -1.28) {$\vq_3$};
	\node at (.35, .94) {$\vq_2$};
	\node at (1.25,.1) {$\vq_1$};
	
	\coordinate (p12) at (-1.78, -1.33);
	\coordinate (p23) at (.65, 1.85);
	\coordinate (p13) at (1.03, -.15);

	\draw [thick, gray] (p12) -- (p23);
	
	\draw [thick, gray] (p12) -- (p13);
	
	\draw [thick, gray] (p13) -- (p23);
	
	\node  at (-2, -1.38) {$\vp_{1}$};
	\node at (.65, 2.02) {$\vp_{3}$};
	\node at (1.15, -.39) {$\vp_{2}$}; 
	\node at (-.3, -.4) {$\BP$};
	
	\fill (p12) circle (1.5pt) {};
	\fill (p23) circle (1.5pt) {};
	\fill (p13) circle (1.5pt){};
	
	\foreach \x in {-3,-2, -1, 0, 1, 2}{
		\foreach \y in {-2, -1,0, 1, 2, 3}{
			\fill (\x,\y) circle (1pt) {};
		}
	}
	\end{tikzpicture}
	\caption{Considering the colinearity restrictions of the vertices of $P$ further restricts the regions its vertices can lie in.}
	\label{fig:terminal-triangle-one-shard}
\end{figure}
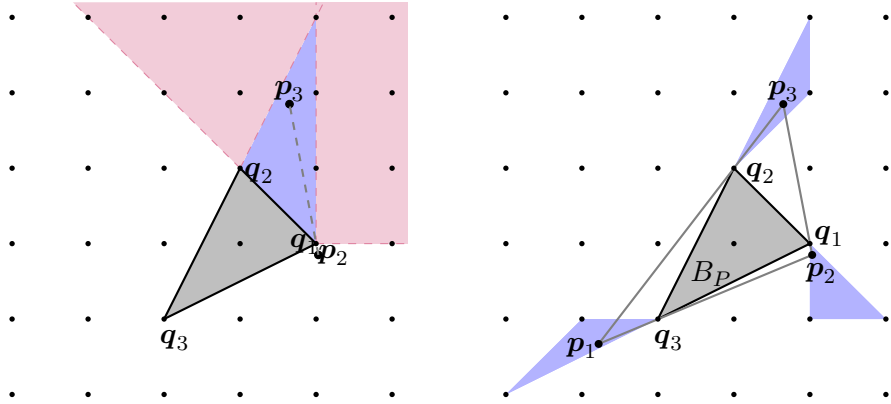

There is a cyclic group of affine unimodular transformations cycling the vertices of $\Bterm$. We can write a generator of this group of symmetries as the linear transformation given by
\begin{equation*}
\sigma = \begin{bmatrix}
0 & -1 \\
1 & -1 
\end{bmatrix},\quad \sigma^3 = I.
\end{equation*}
Moreover, the symmetry $\sigma$ is transitive on the width directions from $\mathcal A$, i.e.,
\begin{equation}\label{eq:action_on_directions}
\sigma^T (- \ve_1)  = \ve_2,\quad
\sigma^T \ve_2  = \ve_1 - \ve_2,\quad
\sigma^T (\ve_1 - \ve_2)  = - \ve_1.
\end{equation}
We further restrict the possible $1$-maximal $P$ under our attention.

\begin{lemma} \label{lem:terminal-triangle-one-width-direction}
	Let $P$ be a $1$-maximal polygon such that $\BP = \Bterm$ whose vertices lie in the regions given by Lemma \ref{lem:terminal-triangle-shards}. Suppose further that there is a \emph{unique} $\vec \phi \in\mathcal A = \{\pm \ve_1, \pm \ve_2, \pm (\ve_2 - \ve_1)\}$ with $\wdt(P,\vec\phi) = \wdt(P ; \mathcal A)$. Then there exists a $P'$ with $B_{P'} = \Bterm$ and whose vertices lie in the regions given by Lemma \ref{lem:terminal-triangle-shards}  such that 
	$\wdt(P;\mathcal A) < \wdt (P', \mathcal A)$.
	%
\end{lemma}

\begin{proof}
	Due to the $\sigma$-symmetry, we can assume that $\vec \phi =\ve_2$. 
	We construct $P'$ by replacing $\vp_2$ with $\vp'_2 = (1-\varepsilon) \vp_2 + \varepsilon \vq_1$.
	Let $\vp'_1$ be the intersection of the line $\vp'_2 \vq_3$ with the line spanned by $\vp_3 \vq_1$.
	
	\begin{figure}[h]
		\centering
		\begin{tikzpicture}[scale = 1.5]
		
		\fill[gray!50] (1, 0) -- (0, 1) -- (-1, -1) -- cycle;
		
		\fill[blue!30] (0, 1) -- (1, 2) -- (1, 3) -- cycle; 
		\fill[blue!30] (1,0) -- (1, -1) -- (2, -1) -- cycle; 
		\fill[blue!30] (-1, -1) -- (-2, -1) -- (-3, -2) -- cycle; 
		
		\coordinate (q1) at (-1, -1);
		\coordinate (q2) at (0, 1);
		\coordinate (q3) at (1, 0);
		
		\draw [thin, black] (q1) -- (q2);
		\draw [thin, black] (q1) -- (q3);
		\draw [thin, black] (q2) -- (q3);

		\node at (-.86, -1.28) {$q_3$};
		\node at (.24, .95) {$q_2$};
		\node at (.7, 0) {$q_1$};
		
		\coordinate (p12) at (-1.78, -1.33);
		\coordinate (p23) at (.65, 1.85);
		\coordinate (p13) at (1.03, -.15);
		
		\node[draw,circle,inner sep=.5pt,fill] at (p12) {};
		\node[draw, circle, inner sep=.5pt, fill] at (p23) {};
		\node[draw, circle, inner sep=.5pt, fill] at (p13) {};
		
		\draw [thin, gray] (p12) -- (p23);
		
		\draw [thin, gray] (p12) -- (p13);
		
		\draw [thin, gray] (p13) -- (p23);
		
		\node  at (-1.83, -1.17) {$p_{1}$};
		
		\node at (.65, 2.02) {$p_{3}$};
		\node at (1.05, -.35) {$p_{2}$}; 
		
		\coordinate (p'13) at (1.015, -.075);
		\node[draw, circle, inner sep = .5pt, fill] at (p'13) {};
		\coordinate (p'12) at (-1.813, -1.38); 
		\node[draw, circle, inner sep = .5pt, fill] at (p'12) {};
		\draw [thin, dotted, gray] (p'13) -- (p'12);
		\draw [thin, dotted, gray] (p'12) -- (p12); 
		\node at (-1.813, -1.53) {$p'_{1}$};
		
		\foreach \x in {-3,-2, -1, 0, 1, 2}{
			\foreach \y in {-2, -1,0, 1, 2, 3}{
				\fill (\x,\y) circle (1.5pt) {};
			}
		}
		
		\node at (-.2, -.3) {$B_P$};
		\end{tikzpicture}
		\cprotect\caption{Constructing a perturbation $P'$ of $P$ such that the width in direction $\ve_2$ increases.}
		\label{fig:shards}
	\end{figure}
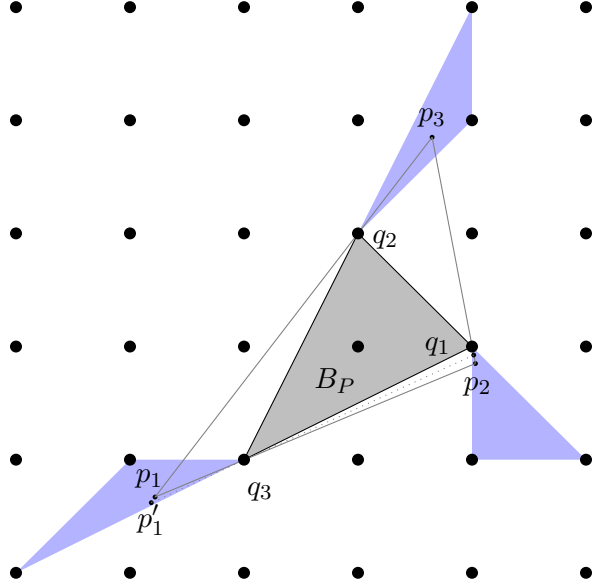
	
	The polytope 
	$P' = \conv \{\vp'_1, \vp'_2, \vp_3\}$
	is a perturbation of $P$, it is still a $1$-maximal polygon with $B_{P'} =\Bterm$ and its width in direction $\ve_2$ is larger than that of $P$, since $\vp_1$ is a non-trivial convex combination of $\vp'_1, \vq_2$. As we have strict inequalities
	\begin{align*}
	\wdt(P, \ve_2)  < \wdt(P, \ve_1),\quad
	\wdt(P, \ve_2)  < \wdt(P, \ve_1 - \ve_2), 
	\end{align*}
	and widths are continuous in Hausdorff distance, we can choose $\varepsilon$ small enough that the same inequalities still hold for $P'$. 
\end{proof}

In view of Lemma \ref{lem:terminal-triangle-one-width-direction}, to prove Lemma \ref{lem:terminal-triangle} we can restrict our attention to those $P$ which attain the width in two of the three directions in $\hex$. Due to the symmetry induced by $\sigma$ we can assume they are $\ve_1, \ve_2$ (cf.\ \eqref{eq:action_on_directions}). 

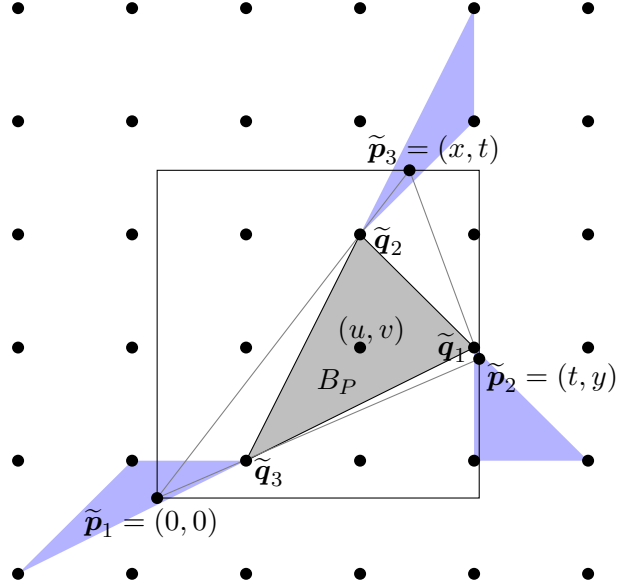
\begin{figure}[h]
	\centering
	\begin{tikzpicture}[scale = 1.5]
	\fill[gray!50] (1, 0) -- (0, 1) -- (-1, -1) -- cycle;

	\fill[blue!30] (0, 1) -- (1, 2) -- (1, 3) -- cycle; 
	\fill[blue!30] (1,0) -- (1, -1) -- (2, -1) -- cycle; 
	\fill[blue!30] (-1, -1) -- (-2, -1) -- (-3, -2) -- cycle; 
	\coordinate (q1) at (-1, -1);
	\coordinate (q2) at (0, 1);
	\coordinate (q3) at (1, 0);
	
	\draw [thin, black] (q1) -- (q2);
	\draw [thin, black] (q1) -- (q3);
	\draw [thin, black] (q2) -- (q3);

	\node at (-.8, -1.1) {$\widetilde{\vq}_3$};
	\node at (.25, .94) {$\widetilde{\vec q}_2$};
	\node at (.81, 0) {$\widetilde{\vq}_1$};
	
	\coordinate (p12) at (-1.78, -1.33);
	\coordinate (p23) at (.43333, 1.56666);
	\coordinate (p13) at (1.045, -.10);
	
	\draw [thin, gray] (p12) -- (p23);
	
	\draw [thin, gray] (p12) -- (p13);
	
	\draw [thin, gray] (p13) -- (p23);
	
	\node  at (-1.83, -1.55) {$\widetilde{\vp}_{1} = (0, 0)$};
	
	\node at (.65, 1.75) {$\widetilde{\vp}_{3} = (x, t)$};
	\node at (1.69, -.263) {$\widetilde{\vec p}_{2} = (t, y)$}; 
	\node at (.1, .15) {$(u, v)$};
	\draw (p12) -- (1.045, -1.33) -- (1.045, 1.56666) -- (-1.78, 1.56666) -- cycle; 
	
	\node at (-.2, -.3) {$B_P$};
	
	\fill  (p12) circle (1.5pt);
	\fill (p23) circle (1.5pt);
	\fill (p13) circle (1.5pt);
	\foreach \x in {-3,-2, -1, 0, 1, 2}{
		\foreach \y in {-2, -1,0, 1, 2, 3}{
			\fill (\x, \y) circle (1.5pt);
		}
	}
	\end{tikzpicture}
	\caption{Parametrizing circumscribers of $B^{\mathrm{term}}$ for which the width is attained in directions $\ve_1,\ve_2$. The lattice $\ZZ^2$ is shifted by $(u, v)$.}
	\label{fig:square-box}
\end{figure}

Let $P$ be a $1$-maximal polygon such that $\BP = \Bterm$, and $\wdt(P, \ve_1)=\wdt(P,\ve_2)$. Given such a polygon $P$, we can shift our perspective by translating $\ZZ^2$ and $P$ by $-\vp_1$ to obtain a polytope
$$\widetilde{P} = \conv \{ \widetilde{\vp}_1, \widetilde{\vp}_2, \widetilde{\vp}_3 \}.$$ 
Furthermore, since the horizontal and vertical width of $P$ have been assumed to be equal,
such $\widetilde{P}$ can be parameterized by real variables $(x,y,u,v,t)$ such that 

\begin{align*}
  \widetilde{\vp}_1 = (0,0), 
  \widetilde{\vp}_2 = (t,y), 
  \widetilde{\vp}_3 = (x,t), \\
  \widetilde{\vq}_i = \vq_i + (u,v), \\
  \widetilde{S}_i = S_i + (u,v),
\end{align*}

Thus $t$ is the width of $\widetilde{P}$ in directions $\ve_1,\ve_2$ (the side length of the square in Figure \ref{fig:square-box}). 
Note that not every choice of $(x, y, u, v, t)$ yields a $1$-maximal polygon $P$ with $\BP=\Bterm$; we further need the following colinearity constraints for $1\leq i \leq  3$ (here we understand the indices as being modulo $3$) 
\begin{equation}
\label{eq:terminal-colinearities}
f_i(x,y,u,v,t) = \det \begin{bmatrix}
\\
\widetilde{\vp}_{i} & \widetilde{\vq}_{i-1} & \widetilde{\vp}_{i+1} \\[1em]
1 & 1 & 1
\end{bmatrix}  = 0.
\end{equation}
Further, we have required $t\leq \wdt(\widetilde{P}, \ve_1 -\ve_2) = 2t - x -y $. 

To prove Lemma \ref{lem:terminal-triangle} we are thus interested in tuples $(x,y,u,v,t)$ solving the optimization problem
\begin{align*}
\max_{x,y,u,v,t}~ t \text{ subject to }
& f_i(x,y,u,v,t) = 0, ~ 1\leq i \leq 3, \\
& t  \leq 2t - x - y, \\
& \widetilde{\vp}_i \in \widetilde{S}_i,~1\leq i \leq 3.
\end{align*}

We denote the polyhedral region cut out by the inequalities in the constraints above by 
$$\mathcal{R} = \{(x, y, u, v, t) : t \leq 2t-x-y,~ \widetilde{\vp}_i \in \widetilde{S}_i \}. $$

We first show that local maximizers of $t$ cannot exist in the interior of $\mathcal{R}$. The following Proposition is a direct application of the methods of Lagrange multipliers.
\begin{proposition}[cf.\ Theorem 14.15 in \cite{GrivaNashSofer08}] \label{prop:thm translation}
	Let $r_* \in \mathcal{R}$ be a local maximizer of $t$, over $\mathcal{R}$, subject to equalities $f_i = 0$. If $r_*$ is a regular point of the constraints, that is, the gradients $\nabla f_i$ are independent at $r_*$, then $\nabla t$ is linearly dependent of the gradient vectors $\nabla f_i$, at $r_*$. 
\end{proposition} 
To prove there are no local maximizers of $t$ in the interior of $\mathcal{R}$, we compute the matrix of gradients (with respect to the variables $x,y,u,v,t$)
\begin{equation*}
\begin{bmatrix}

\nabla \wdt(\widetilde{P},\ve_1) & \nabla f_1 & \nabla f_2 & \nabla f_3

\end{bmatrix}
=
\begin{bmatrix}
0 & \cellcolor{red} 0 & \cellcolor{red} v+1 & \cellcolor{red}-y+v \\
0 & - u + 1 & 0 & -x+u+1\\
0 & \cellcolor{red} -y & \cellcolor{red} -t & \cellcolor{red} y-t\\
0 & \cellcolor{red} t & \cellcolor{red} x & \cellcolor{red} x - t\\
\cellcolor {blue} 1 & v - 1 & -u & u-v+2t-1\\
\end{bmatrix}.
\end{equation*}

The highlighted red minor is positive on the interior of $\mathcal{R}$, as it factors as $(y+1)(t^2-xy)$. 
Indeed, we have $y > 0$ and $x,y\leq t$ by Lemma \ref{lem:terminal-triangle-width-attainers}. We cannot have $x=y=t$ as this is incompatible with the colinearity constraints. Thus, the minor is non-zero and the points in the interior of $\mathcal R$ are regular.
Moreover, the larger matrix has rank $4$ and therefore contradicts the conclusion of Proposition $\ref{prop:thm translation}$.
Hence there are no local optima on the interior of $\mathcal{R}$. 

Having proved that there are no local maximizers in the interior of $\mathcal{R}$ via Proposition \ref{prop:thm translation}, we are left with considering the boundary of $\mathcal{R}$. We first consider if one of the $\widetilde{\vp}_i$ lies on the boundary of $\widetilde{S}_i$, i.e. if one of the $\vp_i$ lies on the boundary of an $S_i$.

Up to the cyclic symmetry interchanging the $\vp_i$, it suffices to handle when $\vp_2 \in \partial S_2$. Moreover, if $\vp_2 \in [(1,0), (1, -1)]$ then $\vp_3 \in [(1,2), (1, 3)]$ and applying the symmetry $\sigma$ yields a symmetric polygon with $\vp_2 \in [(1, -1), (2, -1)]$.

So suppose that $\vp_2 \in [(1, -1), (2, -1)]$. Then we see that $(-1, -1), (0, -1)$, $(1, -1)$ are necessarily boundary points of $P$, so Lemma \ref{lemma:three-colinear-boundary-points} applies. 

If $\vp_2 \in [(1, 0), (2, -1)]$ then $\vp_3 = \vq_2$, contradicting that $\BP = \Bterm$. 

The last remaining boundary condition of $\mathcal{R}$ is when the widths of the polygon $P$ in the directions in $\hex$ are all equal. This can be expressed as $t = x + y$. 

Consider the ideal $I=\langle f_1, f_2, f_3, t-x-y\rangle$. We want to show that all points in the zero set $V(I)$ of $I$ satisfy the constraint $x+y \leq 3$. Notice that in this setup we are forgetting all inequalities coming from the constraints $\vp_i \in S_i$, so we are only requiring that $P$ is circumscribed around $\Bterm$ and that the widths in directions $\hex$ are all the same. Showing $x+y \leq 3$ is thus proving that the width of all such polygons is at most $3$.

To solve this problem, we compute a special Groebner basis $\mathcal{G}$ of $I$ with the property that $\mathcal{G} \cap \KK[x,y]$ is a Groebner basis of $I \cap \KK[x,y]$. The latter is called the \emph{elimination ideal}. 
For a thorough discussion of Groebner bases and elimination ideals see Chapter 3 of \cite{CoxLittleOSheaIVA}.


The computation can be via computer algebra systems, see for example SageMath's \texttt{elimination\_ideal} function documentation \cite{sagemath}.
We compute that the elimination ideal has a single generator,
$$g(x,y)=x^2 + xy + y^2 -3x.$$ 


In particular, $g \in I$, so all points in $V(I)$ must satifsy $g(x,y)=0$. It is thus enough to consider when $x+y \geq 3$ on the ellipse $C=\{(x, y) : g(x,y) = 0 \}$. We observe that, if we apply a change of variable $z=x+y$ and $s=x-y$, the defining equation of $C$ becomes 
\[\frac{(z-1)^2}{4}+\frac{(s-3)^2}{12}=1,\]
and we easily see that points on this curve must satisfy $z-1 \leq 4$, that is, $x+y \leq 3$, with equality if and only if $(x,y)=(3,0)$, in which case $P$ is an affine translate of $3\Tst$.
%
%


So we have established: 

\begin{lemma}\label{lem:terminal-triangle}
	Let $P$ be a $1$-maximal polygon such that $\BP = \Bterm$. Then $\wdt(P; \hex) < 3$, where $\hex = \{\pm \ve_1, \pm \ve_2, \pm (\ve_2 - \ve_1)\}$. In particular we have $\wdt(P)<3$.
\end{lemma}

At this point, all the cases from Proposition \ref{prop:remaining-casework} have been settled. This also concludes the proof of Theorem \ref{thm:flt21}

\section{Consequences}
\label{sec:consequences}

\subsection{The isominwidth theorem for the lattice point enumerator}

%

We now have the necessary ingredients to prove the isominwidth inequality for the lattice point enumerator.

\begin{proof}[Proof of Corollary \ref{cor:lattice_isominwidth}]
We have (cf.\ \eqref{eq:fltdk_eq})
\[
\sup\left\{\frac{\wdt(K)}{\LEinter{K}^{1/2}} : K\in\K^2,~\LEinter{K}>0\right\} = \sup\left\{\frac{\flt(2,k)}{\sqrt{k}} : k\in\NN_{>0}\right\}.
\]
From \eqref{eq:fltdk_ub}, Theorem \ref{thm:hurkens} and \eqref{eq:makai} we see that
\[
\frac{\flt(2,k)}{\sqrt{k}} \leq \flt(2,\infty) + \frac{\flt(2,0)}{\sqrt{k}} = \sqrt{\tfrac 83} + \frac{1+\tfrac 2{\sqrt{3}}}{\sqrt{k}}.
\]
The expression on the right hand side is decreasing in $k$. Evaluating it for $k=3$ shows $\tfrac{\flt(2,k)}{\sqrt{k}} < 3$ for all $k\geq 3$. The claim follows from Theorems \ref{thm:flt21} and \ref{thm:flt22}: We have $\flt(2,1)/\sqrt{1} = 3$ and $\flt(2,2)/\sqrt{2} \approx 2.357$. The characterization of the equality case thus also follows from Theorem \ref{thm:flt21}.
\end{proof}


\subsection{Transference theorems}

We use the classification of the blocking polygons in order to prove a transference theorem.

\begin{proof}[Proof of Theorem \ref{thm:transf}]
By homogeneity, it is sufficient to prove the inequality for origin symmetric convex bodies with $\lambda_1(K) = 1$, in which case the origin is the unique interior lattice point of $K$ and $K$ contains lattice points on its boundary.
Moreover, we may assume that $K$ is 1-maximal. Otherwise, we can find an origin symmetric 1-maximal convex body $K'$ that contains $K$. Then, $K'$ contains the origin as its unique interior point and keeps the boundary points of $K$ on its boundary. This implies $\lambda_1(K') = 1$. At the same time, it follows from inclusion that $\lambda_1((K')^*) \geq \lambda_1(K^*)$. Hence, we can assume $K'=K$ in the following.

A 1-maximal origin symmetric convex body is a polytope whose blocking polytope $B_P$ is itself origin symmetric. By the classification in Section \ref{sec:flt21}, there are only two cases to consider:
\begin{enumerate}[(i)]
\item $B_P = \conv\{ \pm (1,0), \pm(0,1)\}$,
\item $B_P = \conv\{ \pm(1,0), \pm (1,1), \pm(0,1)\}.$
\end{enumerate}
In order to stay close to the notation in Section \ref{sec:flt21} we will consider $\wdt(K)$ rather than $\lambda_1(K^*)$. This is no restriction, since $\wdt(K) = 2\lambda_1(K^*)$. 

\textbf{Case 1:} We shall estimate $\wdt(K)$ by $\wdt(K;E)$, where $E=\{\pm \ve_1, \pm \ve_2\}$. Let $K$ be a 1-maximal origin symmetric polytope with the cross polygon $\conv\{\pm \ve_1, \pm \ve_2\}$ as its blocking polytope. If $K$ has a vertex in the interior of the square $[-1,1]^2$, it must have two opposite ones, and colinearities then force additional interior points, a contradiction. 

Thus, up to symmetry along the line $\{ y=x \}$, we are left with analyzing $K$ with vertices in the regions $S_1', \dots, S_4'$ of Figure \ref{fig:cross-polytope-final-shards} (cf.\ Section \ref{subsub:sym_quad}). $K$ is determined by its vertices $\vp_1=(a,b) \in S_1'$ and $\vp_4=(c,d) \in S_4'$.
After a suitable perturbation of the vertices of such $K$, the horizontal width increases as the vertical width decreases. Thus $K$ maximizes $\wdt(P,E)$ only if the horizontal and the vertical width are equal. 
So we have the following relaxed optimization problem, in which we do not yet regard inequality constraints: 
\begin{equation}
\label{eq:sym_quad_opt}
\begin{split}
\max_{a,b,c,d} ~a\quad\text{s.th.}~0 &= \det\begin{pmatrix}
c & 1 & a\\
d & 0 & b\\
1 & 1 & 1
\end{pmatrix}\\
0 &= \det\begin{pmatrix}
a & 0 & -c\\
b & 1 & -d\\
1 & 1 & 1
\end{pmatrix}\\
0 & = b-c.
\end{split}
\end{equation}

The determinantal constraints indicate the colinearities of the vertices of $K$ with the vertices of $B_K$. The linear constraint is the requirement that the two widths are equal.

Using the method of Lagrangian multipliers, one sees that a local maximum must satisfy, in addition,
\begin{equation}
\label{eq:lagrange_equation}
0 = d-a+1.
\end{equation}
Thus, the tuple $(a,b,c,d)$ is a local maximum in \eqref{eq:sym_quad_opt} (together with the constraints by the regions $S_i'$), if $\vp_1$ or $\vp_4$ lie on the boundary of their region (in which case one readily sees that $\wdt(K)\leq 2$), or if $\vp_1$ and $\vp_4$ lie in the interior of their region and satisfy the four polynomial equations from Equation \eqref{eq:sym_quad_opt} and \eqref{eq:lagrange_equation}.
Let $V\subset\RR$ be the (real) variety defined by those equations. 
One can check that $V$ is 0-dimensional and
that its only point that corresponds to a configuration with the vertices in the regions $S_i'$ is given by
\[
\vp_1=(a,b) = (1/2,(1+\sqrt{2})/2)~\text{and}~\vp_4=(c,d) = ((1+\sqrt{2}/2, -1/2).
\]
The corresponding polytope $K=\conv\{\pm \vp_1,\pm\vp_4\}$ has
 $\wdt(K;E) \leq 1+\sqrt{2}$, i.e.,
\[
\lambda_1(K)\lambda_1(K^*) \leq \frac{1+\sqrt{2}}{2} < \frac 43.
\]

\textbf{Case 2:} The six vertices $\vp_i$ of $K$ lie in the regions $S_i$, $1\leq i \leq 6$, from Proposition \ref{lem:hexagon-shards}. Since $K$ is origin symmetric, we have $\vp_i = -\vp_{i+3}$, where the indices are understood modulo 6.

If one of the $\vp_i$ is on the boundary of $S_i$, the same is true for its antipodal vertex. It is then easily seen that $\wdt(K) \leq 2$. So we assume from here on that $\vp_i\in\inter(S_i)$ for all $1\leq i \leq 6$.

Let $\vq_i$ denote the intersection point of the regions $S_i$ and $S_{i+1}$. Moreover, let $\vr_i$ be the vertex of $S_i$ which is neither $\vq_i$, nor $\vq_{i-1}$. We express $\vp_i$ in barycentric coordinates as follows:
\begin{equation}
\label{eq:bary}
\vp_i = \alpha_i \vq_{i-1} + \beta _i \vr_i + \widetilde\alpha_i \vq_{i},\text{~where~}\alpha_i,\beta_i,\widetilde\alpha_i > 0\text{ and } \alpha_i + \beta_i + \widetilde\alpha_i = 1.
\end{equation}
Since $K$ is origin symmetric it follows that $\alpha_i = -\alpha_{i+3}$, $\beta_i = -\beta_{i+3}$ and $\widetilde\alpha_i = -\widetilde\alpha_{i+3}$.
For simplicity we write, $\vd_1 = -\ve_1$, $\vd_2 = \ve_2-\ve_1$ and $\vd_3 = \ve_2$. Then, origin symmetry implies $\wdt(K,\vd_i) = 2+2\beta_i$. In order to prove the theorem, it suffices to show that $\min_i \beta_i \leq \tfrac 13$ with equality if and only if $\alpha_i=\beta_i=\widetilde\alpha_i = \tfrac 13$ for all $1\leq i \leq 6$.

\begin{claim*}
We have $\beta_{i-1} = \tfrac{\widetilde\alpha_{i-1}\alpha_i}{\beta_i}$ for any $i$.
\end{claim*}

To see this, we recall that the points $\vp_i$, $\vq_{i-1}$ and $\vp_{i-1}$ are colinear. Let $\ell_i$ be the line generated by $\vp_i$ and $\vq_{i-1}$. In view of \eqref{eq:bary} and since $\vr_i = \vq_{i-1} + \vq_i$, we have
\[
\begin{split}
\ell_i &= \{ \vq_{i-1} + \lambda( (\alpha_i-1)\vq_{i-1} + \beta_i\vr_i + \widetilde\alpha_i\vq_i ) : \lambda\in\RR\}\\
&= \{\vq_{i-1} + \lambda( -\widetilde\alpha_i\vq_{i-1} + \beta_i\vq_i + \widetilde\alpha_i \vq_i) : \lambda\in \RR\}.
\end{split}
\]
We observe that $\vq_i = \vq_{i-1} - \vq_{i-2}$ holds for any $i$. Thus, using $\alpha_i+\beta_i+\widetilde\alpha_i = 1$, we obtain
\[
\begin{split}
\ell_i &= \{ \vq_{i-1} + \lambda(\beta_i\vq_{i-1} -\beta_i\vq_{i-2} - \widetilde\alpha_i\vq_{i-2}) : \lambda\in\RR\}\\
&= \{ (1+\lambda\beta_i) \vq_{i-1} -\lambda(1-\alpha_i) \vq_{i-2} : \lambda\in\RR\}.
\end{split}
\]
From $\vp_{i-1}\in\ell_i$ we deduce that there exists some $\lambda\in\RR$ such that
\[
(1+\lambda\beta_i) \vq_{i-1} -\lambda(1-\alpha_i) \vq_{i-2} = \vp_i = (\alpha_{i-1}+\beta_{i-1})\vq_{i-2} +(\widetilde\alpha_{i-1} + \beta_{i-1})\vq_{i-1}.
\]
Extracting the coefficient of $\vq_{i-1}$ gives
\[
\lambda = \frac{\widetilde\alpha_{i-1} +\beta_{i-1} - 1}{\beta_i}=-\frac{\alpha_{i-1}}{\beta_i}.
\]
Extracting the coefficient of $\vq_{i-2}$ and rearranging yields
\[
\beta_{i-1} = \frac{\alpha_{i-1}}{\beta_i}(1-\alpha_i) - {\alpha}_{i-1} = \frac{\alpha_{i-1}(\beta_i + \widetilde\alpha_i) - \beta_i\alpha_{i-1}}{\beta_{i}} = \frac{\alpha_{i-1}\widetilde\alpha_i}{\beta_i},
\]
and the claim is proven.\\
\\
To conclude, we apply the claim three times and use $\beta_i = \beta_{i+3}$ to obtain
\[
\beta_i^2 = \frac{\widetilde\alpha_i\alpha_{i+1}\widetilde\alpha_{i+2}\alpha_{i+3}}{\widetilde\alpha_{i+1}\alpha_{i+2}}.
\]
Hence,
\begin{equation}
\label{eq:coeff_product}
\beta_1^2\beta_2^2 \beta_3^2 = \widetilde\alpha_1\alpha_2\widetilde\alpha_3\alpha_4\widetilde\alpha_5\alpha_6
= \alpha_1\widetilde\alpha_1\alpha_2\widetilde\alpha_2\alpha_3\widetilde\alpha_3.
\end{equation}
Assume that $\beta_i \geq \tfrac 13$ holds for all $1\leq i\leq 3$. Then, $\alpha_i + \widetilde\alpha_i = 1-\beta_i \leq \tfrac 23$. Using the inequality of arithmetic and geometric means on the factors $\alpha_i\widetilde\alpha_i$ in \eqref{eq:coeff_product} yields
\[
\left(\tfrac 19\right)^3 \leq \beta_1^2\beta_2^2 \beta_3^2 =  \alpha_1\widetilde\alpha_1\alpha_2\widetilde\alpha_2\alpha_3\widetilde\alpha_3
\leq \prod_{i=1}^3\left( \tfrac 12 (\alpha_i + \widetilde\alpha_i)\right)^2 \leq \left(\tfrac 19\right)^3,
\]
and it follows that $\beta_1 = \beta_2=\beta_3 = \tfrac 13$. Moreover, it follow from the characterization of the equality case in the AGM inequality that $\alpha_i=\widetilde\alpha_i = \tfrac 13$. The proof is finished.
\end{proof}

\section{Outlook}
\label{sec:outlook}

The determination of \(\flt(3) = \flt(3,0)\) remains a major open challenge. This difficulty stems from the fact that inclusion-maximal three-dimensional hollow bodies are not well understood, and a complete classification is currently unavailable. A key obstacle is the substantial change in the structure of the space of \(d\)-dimensional maximal hollow bodies (considered up to unimodular transformations) when passing from \(d=2\) to \(d=3\). For \(d=2\), this space admits a natural decomposition into finitely many cells, whereas no such finite decomposition appears to exist in dimensions \(d \ge 3\). This phenomenon is closely related to the fact that, up to unimodular transformations, there are infinitely many empty lattice polytopes in dimensions \(d \ge 3\).

We proved that $\flt(d, k)$ is an algebraic number in Theorem \ref{alg:number}. However, as discussed in Remark \ref{rem:alg-degree}, it seems that its degree may be low. Bounding the degree could be intesting--we note that $\flt(2,1)$ and $\flt(2,2)$ are moreover \textit{rational}. Is $\flt(d,k)$ rational for $k \geq 1$?

We hope that the methods developed in this paper will prove useful for addressing further instances of the flatness problem and will ultimately contribute to the determination of \(\flt(d,k)\) for additional pairs $(d, k)$.

\bibliographystyle{abbrv}
\bibliography{literature}
\end{document}